%% file: main.tex
\pdfoutput=1
\documentclass[twoside,11pt,preprint]{article}

\usepackage{blindtext}
\usepackage{jmlr2e}
\usepackage{lastpage}

\jmlrheading{xx}{2024}{}{4/24; Revised 0/00}{0/00}{21-0000}{Na and Mahoney}

\usepackage{enumerate}
\usepackage[T1]{fontenc}
\usepackage[usenames]{color}

\usepackage{mathrsfs}
\usepackage[protrusion=true, expansion=true]{microtype}
\usepackage{float}
\usepackage{subfigure}
\usepackage{amsfonts,amsmath,amssymb,mathtools,url,xspace,bm,multirow}
\usepackage{tikz,verbatim}
\usetikzlibrary{arrows,shapes}
\usepackage[bottom]{footmisc}
\usepackage{enumitem}
\usepackage{lscape}
\usepackage{caption}
\usepackage{soul}
\usepackage{graphicx}
\usepackage{url,array,dsfont,makecell,booktabs} 
\usepackage{algorithm,algorithmicx,algpseudocode}
\usepackage{diagbox}
\usepackage[colorlinks,
linkcolor=red,
anchorcolor=blue,
citecolor=blue
]{hyperref}

\makeatletter
\def\my@tag@font{\normalsize}
\def\maketag@@@#1{\hbox{\m@th\normalfont\my@tag@font#1}}

\usepackage{math} 

\allowdisplaybreaks

\ShortHeadings{Statistical Inference of Constrained Stochastic Optimization}{Na and Mahoney}
\firstpageno{1}

\newcommand{\miniscule}{\fontsize{5}{5.5}\selectfont}
\newcommand{\mminiscule}{\fontsize{4}{5}\selectfont}

\begin{document}

\title{Statistical Inference of Constrained Stochastic Optimization via Sketched Sequential Quadratic Programming}

\author{Sen Na \and Michael W. Mahoney\\ \vskip0.4cm
\addr{H. Milton Stewart School of Industrial \& Systems Engineering, Georgia Tech\\ ICSI and Department of Statistics, University of California, Berkeley}\\ 
}


\editor{}

\maketitle

\begin{abstract}
We consider online statistical inference of constrained stochastic nonlinear optimization problems. We apply the Stochastic Sequential Quadratic Programming (StoSQP) method to solve these problems, which can be regarded as applying second-order Newton's method to the Karush-Kuhn-Tucker (KKT) conditions. 
In each iteration, the StoSQP method computes the Newton direction by solving a quadratic program, and then selects a proper \textit{adaptive} stepsize $\baralpha_t$ to update the primal-dual iterate. To reduce dominant computational cost of the method, we \textit{inexactly} solve the quadratic program in each iteration by employing an iterative sketching solver. Notably, the approximation error of the sketching solver need not vanish as iterations proceed, meaning that the per-iteration computational cost does not blow up. For the above StoSQP method, we show that under mild assumptions, the rescaled primal-dual sequence $1/\sqrt{\baralpha_t}\cdot (\bx_t -\bx^\star, \blambda_t - \blambda^\star)$ converges to a mean-zero Gaussian distribution with a nontrivial~covariance matrix depending on the~\mbox{underlying}~\mbox{sketching}~\mbox{distribution}.~To~\mbox{perform}~inference in practice, we also analyze a plug-in covariance matrix estimator. We illustrate the asymptotic \mbox{normality}~\mbox{result}~of~the~method both on benchmark nonlinear \mbox{problems}~in~CUTEst~test~set~and on linearly/nonlinearly constrained regression problems.
 
\end{abstract}

\vskip5pt

\section{Introduction}

We consider equality-constrained stochastic nonlinear optimization problems of the form
\begin{equation}\label{pro:1}
\min_{\bx\in \mR^d}\;\; f(\bx ) =  \mE_{\P}[F(\bx; \xi)], \quad \text{s.t.} \;\;  c(\bx) = \0,
\end{equation}\vskip-0.25cm
\noindent where $f:\mR^d\rightarrow\mR$ is a stochastic objective function, $F(\cdot;\xi):\mR^d\rightarrow\mR$ is a realization with a random variable $\xi\sim \P$, and $c:\mR^d\rightarrow \mR^m$ provides deterministic equality constraints. Problems of this form appear widely in a variety of applications in statistics and machine learning, including constrained $M$-estimation \citep{Geyer1991Constrained, Geyer1994Asymptotics, Wets1999Statistical}, multi-stage stochastic optimization \citep{Dantzig1993Multi, Veliz2014Stochastic}, physics-informed neural networks \citep{Karniadakis2021Physics, Cuomo2022Scientific}, and algorithmic fairness \citep{Zafar2019Fairness}. In 
practice, the random variable $\xi$ \mbox{corresponds}~to~a~data~\mbox{sample};~$F(\bx; \xi)$~is~the~loss~\mbox{occurred}~at~the sample $\xi$ when using the parameter $\bx$ to fit the model; and $f(\bx)$ is the expected loss. Deterministic constraints are prevalent in real examples, which can encode prior model~\mbox{information},~address identifiability issue, and/or reduce searching complexity.

In this paper, we are particularly interested in performing statistical inference on a (local) primal-dual solution $(\tx,\tlambda)$ of Problem \eqref{pro:1}. To that end, the classical (offline) approach often generates $N$ samples $\xi_1,\ldots,\xi_N\sim\P$ iid, and then solves the corresponding~empirical~risk minimization (ERM) problem: 
\begin{equation}\label{pro:2}
\min_{\bx\in \mR^d}\;\; \hat{f}(\bx) = \frac{1}{N}\sum_{i=1}^{N}F(\bx;\xi_i),\quad \text{s.t.} \;\;  c(\bx) = \0.
\end{equation}
Under certain regularity conditions, we can establish the asymptotic consistency and normality of the minimizer $(\hat{\bx}_N, \hat{\blambda}_N)$ of \eqref{pro:2}, also called \textit{constrained $M$-estimator}, given by
\begin{equation}\label{equ:M:est}
\footnotesize \sqrt{N}\begin{pmatrix}
\hat{\bx}_N - \tx\\
\hat{\blambda}_N - \tlambda
\end{pmatrix}\stackrel{d}{\longrightarrow}\mN\rbr{\begin{pmatrix}
\0\\
\0
\end{pmatrix}, \begin{pmatrix}
\nabla_{\bx}^2\mL^\star & (G^\star)^T\\
G^\star & \0
\end{pmatrix}^{-1}\begin{pmatrix}
\cov(\nabla F(\tx; \xi)) & \0\\
\0 & \0
\end{pmatrix}\begin{pmatrix}
\nabla_{\bx}^2\mL^\star & (G^\star)^T\\
G^\star & \0
\end{pmatrix}^{-1}},
\end{equation}
where $\mL(\bx, \blambda) = f(\bx) + \blambda^Tc(\bx)$ is the Lagrangian function with $\blambda\in\mR^m$ being the dual~variables associated with the constraints, $\nabla_{\bx}^2\mL^\star$ is the Lagrangian Hessian with respect to~$\bx$ evaluated at $(\tx,\tlambda)$, and $G^\star = \nabla c(\tx)\in \mR^{m\times d}$ is the constraints Jacobian. See \cite[Chapter 5]{Shapiro2014Lectures} for the result of \eqref{equ:M:est}, and \cite{Duchi2021Asymptotic} and \mbox{\cite{Davis2024Asymptotic}}~for showing \eqref{equ:M:est} attains the minimax optimality. Numerous methods can be applied to solve constrained ERM \eqref{pro:2}, including (exact) penalty methods, augmented Lagrangian methods,~and sequential quadratic programming (SQP) methods \citep{Nocedal2006Numerical}.

Given the prevalence of streaming datasets in modern problems, offline methods that require dealing with a large batch set in each step are less attractive. It is desirable to design \textit{fully online methods}, where only a \textit{single} sample is used in each step, and~to~\mbox{perform}~\mbox{\textit{online}}~\textit{statistical inference} by leveraging those methods. Without constraints, one can apply stochastic gradient descent (SGD) and its many variates, whose statistical properties (e.g., \mbox{asymptotic} normality) have been comprehensively studied from different aspects \citep{Robbins1951stochastic, Kiefer1952Stochastic,Polyak1992Acceleration, Ruppert1988Efficient}. However, unlike~solving unconstrained stochastic programs, there are limited methods proposed for constrained stochastic programs \eqref{pro:1} that enable online statistical inference. We refer to Section \ref{sec:1.4} for a detailed literature review. One potential exception is the \mbox{projection-based}~SGD~\mbox{recently}~studied in \cite{Duchi2021Asymptotic, Davis2024Asymptotic}. Although the literature has shown~that~projected methods also exhibit asymptotic normality, there are two major concerns when applying these methods for practical statistical inference. 
\begin{enumerate}[label=(\alph*),topsep=5pt,leftmargin=0.8\parindent,labelsep=2pt]
\setlength\itemsep{0.0em}
\item It is unclear how to online estimate the limiting covariance matrix based on the \mbox{projected}~{primal} iterates.~{Due to the absence of dual update, the joint primal-dual normality as in \eqref{equ:M:est}~is not (at least, immediately) achievable for projected methods. Even for the primal normality, the covariance matrix still depends on the dual solution through~the~\mbox{Lagrangian}~\mbox{Hessian} (cf. \eqref{equ:M:est}).~However, due to intrinsic objective noise, simply using primal iterates and optimality conditions to solve for the dual solution does not yield a consistent dual estimator for the underlying plug-in covariance estimation. One possible \mbox{resolution}~is~to~draw~\mbox{inspiration} from long-run variance estimations of stationary processes, and design batch-means~\mbox{covariance}~estimators that utilize only the projected primal iterates themselves.} That said, this approach is highly nontrivial for projected methods {(because of the non-stationarity)},~as~\mbox{studied}~in~the context of vanilla SGD methods \citep{Chen2020Statistical, Zhu2021Online}.

\item There are prevalent scenarios where the projection operator becomes intractable. For~example, when the constraint function $c(\bx)$ is nonlinear and nonconvex, as in physics-informed neural networks (cf.~Section \ref{sec:1.5}), the projection onto the manifold $\{\bx\in\mR^d: c(\bx) = \0\}$~is generally intractable. Additionally, if we only have local information about the \mbox{constraint} function $c(\bx)$ (e.g., function evaluation and Jacobian) at any given point $\bx$, as in CUTEst benchmark nonlinear problems (cf.~Section \ref{sec:5}), the projection operator is not computable~either, which requires a global characterization of the constraint set.

\end{enumerate}

To perform online statistical inference of Problem \eqref{pro:1} without relying on projections, we draw inspiration from a recent growing series of literature in numerical optimization, which develops various \textit{stochastic sequential quadratic programming} (StoSQP) methods for \eqref{pro:1}. The SQP methods can be regarded as second-order Newton's \mbox{methods}~\mbox{applied} to the Karush-Kuhn-Tucker (KKT) conditions. 
In particular, the StoSQP methods \mbox{compute}~a~\mbox{stochastic}~\mbox{Newton}~direction in each iteration by solving a quadratic program, whose objective model is estimated using the new sample. Then, the methods select a proper stepsize to achieve a sufficient reduction on the \textit{merit function}, which balances the optimality and feasibility of the iterates. We refer to \cite{Na2022adaptive, Na2023Inequality, Berahas2021Sequential, Berahas2023Stochastic, Fang2024Fully} and references therein~for recent StoSQP designs and their promising performance on various problems. The aforementioned literature established the global convergence of StoSQP methods, where the KKT~residual $\|\nabla\mL(\bx_t,\blambda_t)\|$ converges to zero almost surely or in expectation for any \mbox{initialization}.~However, in contrast to \cite{Duchi2021Asymptotic,Davis2024Asymptotic} studying projection-based methods, these literature~\mbox{overlooked} the \mbox{statistical}~\mbox{properties} and failed~to~quantify the~\mbox{uncertainty} 
inherent in the StoSQP~\mbox{methods},~which is yet~\mbox{crucial}~for~\mbox{applying} these methods~on~\mbox{online}~statistical inference tasks. Thus, we pose the following question:

\vskip5pt
\emph{Can we perform online inference on $(\tx, \tlambda)$ based on the StoSQP iterates, while further reducing the computational cost of existing second-order StoSQP methods?}
\vskip5pt

In this paper, we answer this question by complementing the global convergence guarantees and establishing the local asymptotic properties of existing StoSQP methods. Specifically, we focus on an Adaptive Inexact StoSQP scheme, referred to as \texttt{AI-StoSQP}. By \textit{adaptive} we mean that the scheme inherits the critical merit of numerical StoSQP designs \citep{Berahas2021Sequential, Curtis2021Inexact, Berahas2023Stochastic}, allowing for an adaptive stepsize $\baralpha_t$ for the Newton direction. In other words, we do not compromise the adaptivity of~StoSQP~to~\mbox{establish} the local convergence guarantees. By \textit{inexact} we mean that the~scheme~further reduces~the computational cost of StoSQP methods by~\mbox{applying}~an~\mbox{iterative} sketching solver to \mbox{inexactly} solve the Newton system in each step\hskip0.85pt\citep{Strohmer2008Randomized, Gower2015Randomized, Pilanci2016Iterative, Pilanci2017Newton, Lacotte2020Optimal}. Solving Newton systems is considered the most computationally expensive step of second-order methods; and randomized solvers offer advantages over deterministic solvers by requiring less flops and memory~when~equipped with proper sketching matrices (e.g., sparse sketches). \mbox{Notably},~we~\mbox{perform} a constant~number of sketching steps;~thus,~the per-iteration computational cost remains fixed even near~\mbox{stationarity}.

For the above sketched StoSQP scheme, we quantify its uncertainty consisting of three components: random sampling, random sketching, and random stepsize. We establish~the~asymptotic normality of the primal-dual iterate:
\begin{equation}\label{nequ:1}
1/\sqrt{\baralpha_t}\cdot (\bx_t-\bx^\star, \blambda_t-\blambda^\star)\stackrel{d}{\longrightarrow}\N(0, \Xi^\star),
\end{equation}
where the limiting covariance $\Xi^\star$ solves a Lyapunov equation that depends on the underlying sketching distribution used in the sketching solver.~Let $\tOmega$ denote the \mbox{limiting}~\mbox{covariance}~of~constrained $M$-estimator in \eqref{equ:M:est}.
Our result suggests that if $\baralpha_t\asymp 1/t$, then we have two cases:
\begin{equation}\label{nequ:2}
\tXi = \tOmega \quad \text{for the exact solver},\hskip2cm \tXi\succeq\tOmega \quad \text{for the sketching solver}.
\end{equation}
This implies that (i) if we solve Newton systems exactly, then online StoSQP \mbox{estimator}~(even with adaptive stepsizes) achieves the same estimation efficiency as offline \mbox{$M$-estimator}.~{In~fact, if we focus solely on the primal variables $\bx$, the marginal covariance of $\tXi\hskip-1pt =\hskip-1pt \tOmega$~also~matches~the limiting covariance of online projection-based estimators established in \cite{Duchi2021Asymptotic, Davis2024Asymptotic}, which is known to be \textit{asymptotic minimax optimal} (see \mbox{Remark}~\ref{rem:2}).}~(ii) If we solve Newton systems inexactly, then the~{sketching}~solver~hurts~the~{asymptotic}~\mbox{optimality}~of StoSQP as $\tXi\succeq \tOmega$.~Fortunately, the hurt is tolerable as seen from the bound (cf.~\mbox{Corollary}~\ref{cor:3})
\begin{equation*}
\|\tXi - \tOmega\| \lesssim \rho^\tau \hskip0.3cm \text{for some } \rho\in(0,1), 
\end{equation*}
where $\tau$ is the number of iterations we run for the sketching solver at each step. 
In~\mbox{addition}~to asymptotic normality, we also present some by-product results of independent interest, including the local convergence rate, sample complexity, and the Berry-Esseen bound that quantitatively measures the convergence rate in \eqref{nequ:1}.~To facilitate practical \mbox{inference},~we~also~\mbox{analyze}~a plug-in covariance estimator that can be computed in online fashion. We illustrate our results on benchmark nonlinear problems in CUTEst test set and on~\mbox{linearly/nonlinearly}~constrained regression problems.

\vskip5pt
\noindent\textit{Structure of the paper.} 
We introduce some motivating examples of Problem \eqref{pro:1} and provide a literature review in Section \ref{sec:11}. Then, we introduce \texttt{AI-StoSQP} in Section \ref{sec:2}~and~prove~the global almost sure convergence with iteration complexity in Section \ref{sec:3}. Asymptotic normality with~covariance estimation is established in Section \ref{sec:4}. Experiments and conclusions~are~\mbox{presented}~in Sections \ref{sec:5} and \ref{sec:6}, respectively. We defer all the proofs to the appendices.

\vskip5pt
\noindent\textit{Notation.} Throughout the paper, we use $\|\cdot\|$ to denote $\ell_2$ norm for vectors and spectral norm for matrices. For scalars $a$, $b$, $a\vee b = \max(a,b)$ and $a\wedge b = \min(a,b)$. We use $O(\cdot)$ (or $o(\cdot)$) to denote big (or small) $O$ notation in usual almost sure sense. For a~sequence~of~\mbox{compatible}~matrices $\{A_i\}_i$, we let $\prod_{k=i}^{j}A_k = A_jA_{j-1}\cdots A_i$ if $j\geq i$ and $I$ (the identity matrix) if $j<i$. We use the bar notation, $\bar{(\cdot)}$, to denote algorithmic quantities that are random (i.e., depending~on realized samples), except for the iterates. We reserve the notation $G(\bx)$ to~denote~the~constraints Jacobian, i.e., $G(\bx) = \nabla c(\bx) = (\nabla c_1(\bx),\ldots,\nabla c_m(\bx))^T\in \mR^{m\times d}$.

\section{Applications and Literature Review}\label{sec:11}

We present two motivating examples of \eqref{pro:1} in Section \ref{sec:1.5}, and then review related literature in Section \ref{sec:1.4}.

\subsection{Motivating examples}\label{sec:1.5}

Many statistical and machine learning problems can be cast into the form of Problem \eqref{pro:1}.$\quad$

\begin{example}[Constrained regression problems]\label{exp:1}

Let $\xi_t = (\xi_{\ba_t}, \xi_{b_t})$ be the $t$-th sample, where $\xi_{\ba_t}\in\mR^d$ is the feature vector independently drawn from some multivariate distribution and $\xi_{b_t}$ is the response. We consider different regression models, such as\vskip-0.45cm
\begin{alignat*}{3}
\text{linear models:} \quad & \hskip1.5cm\xi_{b_t} && = \xi_{\ba_t}^T\tx + \epsilon_t \quad \quad  & &\text{ with } \epsilon_t \text{ iid noise},\\
\text{logistic models:} \quad & P\rbr{\xi_{b_t} \mid \xi_{\ba_t}} &&= \frac{\exp\rbr{\xi_{b_t}\cdot \xi_{\ba_t}^T\tx}}{1 + \exp\rbr{\xi_{b_t}\cdot \xi_{\ba_t}^T\tx}}\quad \quad & &\text{ with }  \xi_{b_t} \in\{-1,1\},
\end{alignat*}\vskip-0.1cm
\noindent where $\tx\in \mR^d$ is the true model parameter. For the above models, we define the corresponding loss functions at $\bx$:\vskip-0.55cm
\begin{align*}
\text{linear models:} \quad & F(\bx; \xi_t) = \frac{1}{2}(\xi_{\ba_t}^T\bx - \xi_{b_t})^2,\\
\text{logistic models:} \quad & F(\bx;\xi_t) = \log\rbr{1+\exp(-\xi_{b_t}\cdot\xi_{\ba_t}^T\bx)}.
\end{align*}\vskip-0.1cm
\noindent Then, we can verify that $\tx = \argmin_{\bx} f(\bx) = \argmin_{\bx}\mE[F(\bx; \xi)]$.
In many cases, we access prior information about the model parameters, which is encoded as constraints.~For~\mbox{example},~in portfolio selection, $\bx$ represents the portfolio allocation vector that satisfies $\bx^T\b1 = 1$ (total allocation is 100\%). We may also fix the target percentage on each sector or each region, which is translated into the constraint $A\bx = \bd$. See \cite[(4.3, 4.4)]{Fan2007Variable}, \cite[(2.1)]{Fan2012Vast}, \cite[(1)]{Du2022High} and references therein for such applications. In principle component analysis and semiparametric single/multiple index regressions, we enforce $\bx$ to have a unit norm to address the identifiability issue, leading to a nonlinear constraint $\|\bx\|^2=1$.~We point to \cite{Kaufman1978Method, Kirkegaard1972POSITRONFIT, Sen1979Asymptotic, Nagaraj1991Estimation, Dupacova1988Asymptotic, Shapiro2000asymptotics, Na2019High, Na2021High} for various examples of linearly/nonlinearly~\mbox{constrained}~\mbox{estimation}~\mbox{problems}.~{For~some~\mbox{constrained}~estimation problems, projecting into the feasible set can be intractable.~For instance,~in~\mbox{factor}~\mbox{analysis}, researchers may estimate a covariance matrix $\Sigma$ under so-called \textit{tetrad constraints}: $\Sigma_{i_1i_2}\Sigma_{i_3i_4} - \Sigma_{i_1i_4}\Sigma_{i_2i_3} = 0$ for every set of four distinct variables $\{i_1,i_2,i_3,i_4\}$ \citep{Bollen2000tetrad, Drton2016Wald, Sturma2024Testing}.~For such highly nonlinear \mbox{constraints},~the linear-quadratic approximation performed in SQP can be a promising resolution.$\hskip3.5cm$
}

\end{example}

\begin{example}[Physics-informed machine learning]
	
Recent decades have seen machine~learning (ML) making significant inroads into science. The major task in ML is to learn an unknown mapping $\bz(\cdot): \mA\rightarrow\mB$ from data that can perform well in the downstream tasks.~Since $\bz(\cdot)$ is infinite-dimensional, one key step in ML is to use neural networks (NNs) to parameterize $\bz(\cdot)$ as $\bz_{\bx}(\cdot)$, and learn the optimal weight parameters $\bx\in \mR^d$ instead (called function approximation). One of the trending topics in ML now is physics-informed ML, where one requires $\bz(\cdot)$ to obey some physical principles that are often characterized by partial differential equations (PDEs) \citep{Karniadakis2021Physics, Cuomo2022Scientific}.~In such applications, we can use the squared loss function, defined for the $t$-th sample $\xi_t = (\xi_{\ba_t},\xi_{\bb_t})\in \mA\times\mB$~as \vskip-0.15cm
\begin{equation*}
F(\bx; \xi_t) = \frac{1}{2}\rbr{\bz_{\bx}(\xi_{\ba_t}) - \xi_{\bb_t}}^2.
\end{equation*}
Here, $\xi_{\ba_t}$ is NN inputs that can be spatial and/or temporal coordinates; $\xi_{\bb_t}$ is measurements that can be speed, velocity, and temperature, etc; and $\bz_{\bx}\in\mC^\infty(\mA, \mB)$ is NN architecture. Let $\mF: \mC^{\infty}(\mA, \mB)\rightarrow \mC^{\infty}(\mA, \mB)$ be the PDE operator, which encodes the underlying physical law (e.g., energy conservation law). We aim to find optimal weights $\tx$ that not only minimize~the mean squared error of observed data, but also satisfy the constraints $\mF(\bz_{\bx}) = \0$. To this~end, we select some leverage points $\{\xi_{\ba_i}'\}_{i=1}^m$ in $\mA$ and impose \textit{deterministic} constraints: 
\begin{equation*}
\mF(\bz_{\bx})(\xi_{\ba_i}') = \0, \quad \forall i = 1,2,\ldots, m.
\end{equation*}
Here, we abuse the notation $\0$ to denote either a zero mapping of $\mC^{\infty}(\mA, \mB)$ or a zero element of $\mB$. For more details on this problem~formulation, see \cite[(2.3)]{Lu2021Physics},~\cite[(2)]{Krishnapriyan2021Characterizing}, and references therein.~{Due to the nonlinearity nature of~NNs,~projection-free methods are desired, and SQP can achieve competitive performance compared~to~penalty~methods and augmented Lagrangian methods \citep{Cheng2024Physics}. 
}

\end{example}

\subsection{Related literature and contribution}\label{sec:1.4}

There are numerous methods for solving constrained optimization problems, such as projection-based methods, penalty methods, augmented Lagrangian methods, and sequential quadratic programming (SQP) methods \citep{Nocedal2006Numerical}. This paper particularly considers solving constrained stochastic optimization problems via Stochastic SQP (StoSQP) methods, which can be regarded as an application of stochastic Newton's method on constrained~problems. \cite{Berahas2021Sequential} designed the very first online StoSQP scheme. At each step, the method selects a suitable penalty parameter of an $\ell_1$-penalized objective; ensures the Newton direction produces a sufficient reduction on the penalized objective; and then selects~an~adaptive stepsize $\beta_t\leq \baralpha_t\leq \eta_t=\beta_t+\chi_t$ based on input sequences~$\beta_t$~and~$\chi_t = O(\beta_t^2)$.~An~alternative StoSQP scheme was then reported in \cite{Na2022adaptive}, where $\baralpha_t$ is selected by performing stochastic line search on the augmented Lagrangian with batch sizes increasing as iteration proceeds. Subsequently, \cite{Curtis2021Inexact, Na2023Inequality, Berahas2023Stochastic, Fang2024Fully} proposed different variates of StoSQP to cope with inequality~constraints, degenerate constraints, etc. These works all proved the global convergence of StoSQP methods --- the KKT residual $\|\nabla\mL(\bx_t,\blambda_t)\|$ converges to zero from any initialization.~However,~they~fall~short of uncertainty quantification and online statistical inference goals.

On the other hand, a growing body of literature leverages optimization procedures~to~facilitate online inference, starting with \cite{Robbins1951stochastic, Kiefer1952Stochastic} and continuing through \cite{Robbins1971convergence, Fabian1973Asymptotically, Ermoliev1983Stochastic}. To study the asymptotic distribution of stochastic gradient descent (SGD), \cite{Ruppert1988Efficient} and \cite{Polyak1992Acceleration} averaged SGD iterates and established the optimal central limit theorem rate. \cite{Toulis2017Asymptotic} designed an implicit SGD method and showed the asymptotics of averaged implicit SGD iterates. \cite{Li2018Statistical}~\mbox{designed}~an~\mbox{inference}~\mbox{procedure} for constant-stepsize SGD by averaging the iterates with recurrent burn-in periods. \cite{Mou2020linear} further showed the asymptotic covariance of constant-stepsize SGD with Ployak-Ruppert averaging. \cite{Liang2019Statistical} designed a moment-adjusted SGD method and provided non-asymptotic results that characterize the statistical distribution as the batch size of each step tends to infinity. \cite{Chen2020Statistical} and \cite{Zhu2021Online} proposed different covariance matrix estimators constructed by grouping SGD iterates. Additionally, \cite{Chen2021First} designed a distributed method for the inference of non-differentiable~convex problems; \cite{Roy2023Online} analyzed a batch-mean covariance estimator under a Markovian sampling setup; and \cite{Duchi2021Asymptotic} and \cite{Davis2024Asymptotic} applied projection-based SGD methods for the inference of inequality-constrained convex~problems.~The~aforementioned literature all studied first-order methods with deterministic stepsizes.

The asymptotics of second-order Newton's methods for unconstrained problems have recently been investigated. \cite{Bercu2020Efficient} designed an online Newton's method for logistic regression, and \cite{Boyer2023asymptotic} generalized that method to general regression problems. Compared to first-order methods that often consider~\mbox{averaged}~\mbox{iterates}~and/or exclude the stepsize $1/t$ due to technical challenges, both works showed the normality of the \textit{last} iterate with $1/t$ stepsize.~However, those analyses are not applicable to our~study~for~two~reasons. 
First, they studied unconstrained regression problems with objectives in the form~$F(\bx^T\xi)$, resulting in objective Hessians owning rank-one updates that cannot be employed for our general problem \eqref{pro:1}. 
Second, they solved Newton systems exactly and utilized $1/t$ deterministic stepsize. In contrast, we use a randomized sketching solver to solve~\mbox{Newton}~\mbox{systems}~\mbox{inexactly} to reduce the computational cost associated with higher-order methods, along with an~\mbox{adaptive} random stepsize inspired by numerical designs in \cite{Berahas2021Sequential}. Both of these components affect the uncertainty quantification and lead to a different normality result (cf.~\eqref{nequ:2}). 
To our knowledge, this is the first work that performs online inference by taking into account not only the randomness of samples but also the randomness of computation (i.e., sketching and stepsize); the latter is particularly important for making second-order methods~computationally promising.

{ 
We briefly review the literature on sketched Newton methods.~Compared to the works~\mbox{below}, the sketching step in StoSQP is only a subroutine for solving linear-quadratic programs; a complete method also involves merit function reduction and stepsize selection (here, stepsize refers to that of StoSQP rather than the sketching solver). This paper focuses~on~uncertainty quantification and statistical inference of online (sketched) Newton methods, which differs significantly from the following literature that focuses on design and convergence of sketched Newton methods. In particular, for many (regression) problems, the objective Hessian can be expressed as $\Hb = AA^T\in\mR^{d\times d}$ with a data matrix $A\in\mR^{d\times n}$ and $n\geq d$. Then, one can generate a sketch matrix $S\hskip-2pt \in\hskip-2pt\mR^{n\times s}$ and compute the approximate Hessian $\hat{\Hb}\hskip-2pt = \hskip-2pt ASS^TA^T$.~{\cite{Pilanci2016Iterative}} developed an iterative Hessian sketch algorithm for solving least-squares problems $\min_{\bx} \|A\bx-b\|^2$ (subject to convex constraints). 
The authors sketched only the data matrix~$A$ rather than~both~the~data~\mbox{matrix} $A$ and vector~$b$,~and~\mbox{established}~a~high-probability convergence result.~\cite{Lacotte2020Optimal}~later~\mbox{extended} this study by \mbox{showing}~the~\mbox{optimal}~\mbox{stepsize}~and convergence rate for Haar sketches.~\cite{Pilanci2017Newton} designed a sketched Newton method that approximates the Hessian using the Johnson–Lindenstrauss transform.~Building on this, \cite{Agarwal2017Second, Derezinski2019Distributed, Derezinski2020Debiasing, Derezinski2020Precise, Derezinski2021Newton, Lacotte2021Adaptive} introduced various sketching methods to explore the trade-off~between the computational cost of $\hat{\Hb}$ and the convergence rate of the~algorithm.~In~addition~to the above series of literature, another type of sketched Newton method~is~based~on Sketch-and-Project framework, where one approximates a generic Hessian inverse $\Hb^{-1}$ by $S(S^T\Hb S)^\dagger S^T$ for a sketch matrix $S\in\mR^{d\times s}$.~See \cite{Strohmer2008Randomized, Gower2015Randomized, Luo2016Efficient, Doikov2018Randomized, Gower2019RSN, Derezinski2024Sharp} and references therein for the convergence properties of this family of methods.$\hskip1.5cm$

Compared to deterministic methods for solving Newton systems, such as conjugate~gradient and broad preconditioned Krylov (or minimal residual) methods, randomized sketching methods may behave better in terms of improved convergence rates and range of convergence, while requiring less computation and memory (when using suitable sketches)~to~be~\mbox{scalable} and parallelizable \citep{Gower2016Sketch}.~We refer to \cite{Hong2023Constrained} for an empirical~demonstration of the advantages of sketching solvers over deterministic solvers in the context~of~SQP~methods.\;\;\;

}

\section{Adaptive Inexact StoSQP Method}\label{sec:2}

Let $\mL(\bx,\blambda) = f(\bx) + \blambda^Tc(\bx)$ be the Lagrangian function of \eqref{pro:1}, where $\blambda\in\mR^m$ is~the~dual~vector.~Under certain constraint qualifications (introduced later), a necessary condition~for~$(\bx^\star, \blambda^\star)$ being a local solution to \eqref{pro:1} is the KKT conditions: $\nabla\mL^\star = (\nabla_{\bx}\mL^\star, \nabla_{\blambda}\mL^\star) = (\0,\0)$.

\texttt{AI-StoSQP} applies Newton's method to the equation $\nabla\mL(\bx, \blambda) = \0$, involving three~steps: estimating the objective gradient and Hessian, (inexactly) solving Newton's system, and~updating the primal-dual iterate. We detail each step as follows. For simplicity, we denote~$c_t = c(\bx_t)$ (similarly, $G_t = \nabla c(\bx_t)$, $\nabla\mL_t = \nabla\mL(\bx_t,\blambda_t)$, etc.).

\vskip5pt
\noindent\textbf{$\bullet$ Step 1:\hskip2.5pt Estimate the gradient and Hessian}.~We realize a sample $\xi_t\sim \P$ and~\mbox{estimate}~{the gradient $\nabla f_t$ and Hessian $\nabla^2 f_t$ of the objective as}
\begin{equation*}
\barg_t = \nabla F(\bx_t; \xi_t)\quad\quad \text{ and }\quad\quad \barH_t = \nabla^2F(\bx_t; \xi_t).
\end{equation*}
Then, we compute three quantities:
\begin{equation*}
\bnabla_{\bx}\mL_t = \barg_t + G_t^T\blambda_t,\quad \bnabla_{\bx}^2\mL_t = \barH_t + \sum_{i=1}^m(\blambda_t)_i\nabla^2c_i(\bx_t), \quad B_t = \frac{1}{t}\sum_{i=0}^{t-1}\bnabla_{\bx}^2\mL_i + \Delta_t.
\end{equation*}
Here, $\bnabla_{\bx}\mL_t$ and $\bnabla_{\bx}^2\mL_t$ are the estimates of the Lagrangian gradient and Hessian with respect to $\bx$, respectively; and $B_t$ is a regularized averaged Hessian used in the quadratic~program~\eqref{equ:QP}. We let $\Delta_t = \Delta(\bx_t, \blambda_t)$ be any regularization term ensuring $B_t$ to be positive definite in~the~null space $\{\bx\in \mR^d: G_t\bx = \0\}$. Note that the average $\sum_{i=0}^{t-1}\bnabla_{\bx}^2\mL_i/t$~can~be~\mbox{updated}~online.~We~explain the matrix $B_t$ in the following remark.

\begin{remark}

We note that {the Lagrangian Hessian average in $B_t$ is over samples $\{\xi_0, \ldots, \xi_{t-1}\}$, meaning that the Hessian estimate $\bar{H}_t$, which depends on the new sample $\xi_t$, will only~be~used~in the $(t+1)$-th iteration.}~Thus, $B_t$ and $\Delta_t$ are deterministic given $(\bx_t, \blambda_t)$.~In addition,~$\Delta_t$~can~simply be Levenberg-Marquardt type regularization of the form $\delta_t I$ with suitably large $\delta_t>0$ \citep{Nocedal2006Numerical}. The Hessian regularization is standard for nonlinear problems, together with linear independence constraint qualification (LICQ, Assumption~\ref{ass:1}),~\mbox{ensuring}~that the quadratic program \eqref{equ:QP} is solvable. For convex problems, we just set $\Delta_t=\0$, $\forall t$. See \cite{Bertsekas1982Constrained, Nocedal2006Numerical} for various regularization approaches. \mbox{Moreover}, \cite{Na2022Hessian} showed that Hessian averaging accelerates the local rate of Newton's method~on unconstrained deterministic problems.

\end{remark}

\noindent\textbf{$\bullet$ Step 2: Solve the quadratic program}. 
With the above estimates, we solve the~quadratic program (QP):
\begin{equation}\label{equ:QP}
\min_{\tDelta\bx_t\in\mR^d} \;\; \frac{1}{2}\tDelta\bx_t^TB_t\;\tDelta\bx_t + \barg_t^T\tDelta\bx_t,\quad \text{s.t.}\;\; c_t +  G_t\tDelta\bx_t = \0.
\end{equation}
For the above QP, the objective can be seen as a quadratic approximation of $F(\bx;\xi)$~at~$(\bx_t;\xi_t)$, and the constraint can be seen as a linear approximation of $c(\bx)$ at $\bx_t$. It is easy to observe~that solving the above QP is equivalent to solving the following Newton system
\begin{equation}\label{equ:Newton}
\underbrace{\begin{pmatrix}
B_t & G_t^T\\
G_t & \0
\end{pmatrix}}_{K_t}\underbrace{\begin{pmatrix}
\tDelta\bx_t\\
\tDelta\blambda_t
\end{pmatrix}}_{\tbz_t} = -\underbrace{\begin{pmatrix}
\bnabla_{\bx}\mL_t\\
c_t
\end{pmatrix}}_{\bnabla\mL_t},
\end{equation}\vskip-0.1cm
\noindent where $K_t$, $\bnabla\mL_t$ are the Lagrangian Hessian and gradient, and $\tbz_t$ is the exact Newton~direction.

Instead of solving the QP \eqref{equ:QP} exactly, we solve it inexactly by an iterative sketching solver. This approach proves more efficient than deterministic solvers, especially when~equipped~with suitable sketching matrices \citep{Strohmer2008Randomized, Gower2015Randomized, Pilanci2016Iterative, Pilanci2017Newton, Lacotte2020Optimal}.~In particular, we generate a~\mbox{random}~sketching matrix $S \in \mR^{(d+m)\times s}$, whose column dimension $s\geq 1$ can also be random, and transform the original large-scale linear system to the sketched, small-scale system as
\begin{equation*}
K_t\bz_t = -\bnabla\mL_t \hskip1cm \Longrightarrow \hskip1cm S^TK_t\bz_t = -S^T\bnabla\mL_t.
\end{equation*}
Clearly, there are multiple solutions to the sketched system, and $\bz_t=\tbz_t$ is one of them.~We~prefer the solution that is closest to the current solution approximation.~That is,~the~\mbox{$j$-th}~iteration of the sketching solver has the form ($\bz_{t,0} = \0$) 
\begin{equation}\label{equ:sap}
\bz_{t,j+1} = \arg\min_{\bz}\|\bz - \bz_{t,j}\|^2\quad\;\;  \text{s.t.} \;\;\;\; S_{t,j}^TK_t\bz = -S_{t,j}^T\bnabla\mL_t, 
\end{equation}
where $S_{t,j} \sim S, \forall j$ are independent and identically distributed and are also independent of $\xi_t$. An explicit recursion of \eqref{equ:sap} is given by
\begin{equation}\label{equ:pseduo}
\bz_{t,j+1} = \bz_{t,j} - K_tS_{t,j}(S_{t,j}^TK_t^2S_{t,j})^\dagger S_{t,j}^T(K_t\bz_{t,j} +\bnabla\mL_t),
\end{equation}
where $(\cdot)^\dagger$ denotes the Moore–Penrose pseudoinverse.~One can let $s=1$ (i.e., using~\mbox{sketching}~vectors) so that $S_{t,j}^TK_t^2S_{t,j}$ reduces to a scalar and the pseudoinverse reduces to the reciprocal.$\quad\quad\;$

We perform $\tau\geq 1$ iterations of \eqref{equ:pseduo} and use
\begin{equation*}
(\barDelta\bx_t, \barDelta\blambda_t) \coloneqq \bz_{t,\tau}
\end{equation*}
as the approximate Newton direction. We emphasize that $\tau$ is independent of $t$; thus,~we~do~not require a vanishing approximation error and blow up the computational cost as $t\rightarrow \infty$.

\begin{remark}
	
A significant difference between randomized solvers and deterministic solvers~is that the approximation error $\|\bz_{t,j}-\tbz_t\|$ of randomized solvers may not be monotonically decreasing as $j$ increases.  This subtlety challenges both inference and convergence analysis. In classical optimization world, it is unanimously agreed that if the search direction is asymptotically close to the exact Newton direction (here $\tbz_t$), then the algorithm will locally behave just like Newton's method with a similar convergence rate.~A~\mbox{precise}~\mbox{characterization} is called the Dennis-Mor\'e condition \citep{Dennis1974characterization}.~In our study, {although~Lemma~\ref{lem:1}~shows~that the expected error $\mE[\|\bz_{t,\tau} - \tbz_t\|\mid \bx_t,\xi_t]$ decays exponentially in $\tau$}, $\bz_{t,\tau}$ can still be far from $\tbz_t$ for any $t$ and large $\tau$ {in the \textit{almost sure} sense. In particular, \cite[Theorem~4.2]{Patel2021Implicit} proved that only a subsequence of $\{\|\bz_{t,\tau} - \tbz_t\|\}_{\tau=0}^{\infty}$ would decrease monotonically~with~the~subsequence indices being also random.}

\end{remark}

\noindent\textbf{$\bullet$ Step 3: Update the iterate with an adaptive stepsize}.~With the direction~$(\barDelta\bx_t, \barDelta\blambda_t) = \bz_{t,\tau}$ from Step 2, we update the iterate $(\bx_t, \blambda_t)$ by an adaptive stepsize $\baralpha_t$:
\begin{equation}\label{equ:update}
(\bx_{t+1}, \blambda_{t+1}) = (\bx_t, \blambda_t) + \baralpha_t\cdot (\barDelta\bx_t, \barDelta\blambda_t).
\end{equation}
In principle, the stepsize $\baralpha_t$ may rely on the random direction $(\barDelta\bx_t, \barDelta\blambda_t)$, so it is also random. We allow using any adaptive stepsize selection schemes but require a safeguard~condition~on~$\baralpha_t$:
\begin{equation}\label{equ:sandwich}
0< \beta_t\leq \baralpha_t \leq \eta_t\quad  \text{ with }\quad  \eta_t \coloneqq \beta_t + \chi_t,
\end{equation}
where $\{\beta_t, \eta_t\}$ are upper and lower bound sequences and $\chi_t$ is the adaptivity gap. We~do~not~require specific stepsize selection schemes beyond the condition \eqref{equ:sandwich} to achieve our \mbox{online} inference goals.~The schemes reported in \cite{Berahas2021Sequential, Berahas2023Stochastic,Curtis2021Inexact}~all~\mbox{adhere}~to the condition \eqref{equ:sandwich}.~See \cite[Lemma 3.6]{Berahas2021Sequential} and \cite[(25,~28)]{Curtis2021Inexact}~for details. Their numerical experiments suggest that adaptive random~stepsizes~\mbox{offer} promising empirical benefits over non-adaptive deterministic stepsizes (i.e., $\chi_t=0$).~For sake of~completeness, we present a selection scheme from \cite{Berahas2021Sequential} in Appendix~\ref{app:1}.$\hskip3cm$

\begin{algorithm}[!t]
\caption{Adaptive Inexact StoSQP Method}\label{alg:1}
\begin{algorithmic}[1]
\State\textbf{Input:} initial iterate $(\bx_0, \blambda_0)$, positive sequences $\{\beta_t, \eta_t\}$, an integer $\tau>0$,~$B_0 = I$;
\For{$t = 0,1,2,\ldots$}
\State Realize $\xi_t$ and compute $\barg_t = \nabla f(\bx_t;\xi_t)$, $\barH_t = \nabla^2 f(\bx_t;\xi_t)$, and $\bnabla_{\bx}^2\mL_t$;
\State Compute the regularized Hessian average $B_t = \frac{1}{t}\sum_{i=0}^{t-1}\bnabla_{\bx}^2\mL_i + \Delta_t$; 
\State Generate sketching matrices $S_{t,j}\sim S$, $\forall j$ iid and iterate \eqref{equ:pseduo} for $\tau$~times;
\State Select any adaptive stepsize $\baralpha_t$ with $\beta_t\leq \baralpha_t\leq \eta_t$, and update the iterate as \eqref{equ:update};
\EndFor
\end{algorithmic}
\end{algorithm}

\vskip4pt 

We combine the above three steps and summarize \texttt{AI-StoSQP} in Algorithm \ref{alg:1}. To end this section, we introduce a filtration notation for later use. We let $\mF_t = \sigma( \{\xi_i, \{S_{i,j}\}_j, \baralpha_i\}_{i=0}^t)$, $\forall t\geq 0$ be the $\sigma$-algebra generated by the random variables $\{\xi_i, \{S_{i,j}\}_j, \baralpha_i\}_{i=0}^t$. Moreover,~we let $\mF_{t-2/3} =  \sigma(\{\xi_i, \{S_{i,j}\}_j, \baralpha_i\}_{i=0}^{t-1} \cup \xi_t)$, $\mF_{t-1/3} = \sigma(\{\xi_i, \{S_{i,j}\}_j, \baralpha_i\}_{i=0}^{t-1}\cup \xi_t\cup\{S_{t,j}\}_j)$, and~have $\mF_{t-1}\subseteq \mF_{t-2/3}\subseteq \mF_{t-1/3} \subseteq \mF_t$. For consistency, $\mF_{-1}$ is the trivial $\sigma$-algebra.
With these~notation, Algorithm \ref{alg:1} has a generating process as follows: given $(\bx_t, \blambda_t)$,~we~first~\mbox{realize}~$\xi_t$~to get {the estimates of} the gradient $\barg_t$ and Hessian $\barH_t$ and derive $\mF_{t-2/3}$; then we generate~$\{S_{t,j}\}_j$ to obtain the inexact Newton direction and derive $\mF_{t-1/3}$; then we select the stepsize $\baralpha_t$ and derive $\mF_t$. 
We also let $(\Delta\bx_t, \Delta\blambda_t)$ be the exact Newton direction~solved~from~\eqref{equ:Newton}~with~$\bnabla_{\bx}\mL_t$~being replaced by $\nabla_{\bx}\mL_t$.

\section{Global Almost Sure Convergence}\label{sec:3}

In this section, we present the global almost sure convergence\footnote{Global convergence in nonlinear optimization refers to the convergence to a stationary point from~any~initialization, in contrast to the convergence to a global solution, which is not achievable without particular~problem structures \citep{Nocedal2006Numerical}. However, they are equivalent for convex problems as~studied for projection-based methods in \cite{Duchi2021Asymptotic, Davis2024Asymptotic}.} for the StoSQP method.~We~show that the KKT residual $\|\nabla\mL_t\|$ converges to zero from any initialization.~This~global~\mbox{convergence} serves as a preliminary result of our inference analysis in Section \ref{sec:4}.

We use an adapted augmented Lagrangian function as the Lyapunov function to show~the convergence, which has two penalty terms of the form 
\begin{equation*}
\mL_{\mu, \nu}(\bx, \blambda) = \mL(\bx, \blambda) + \frac{\mu}{2}\|c(\bx)\|^2 + \frac{\nu}{2}\|\nabla_{\bx}\mL(\bx, \blambda)\|^2,\quad\quad \text{with } \mu, \nu>0.
\end{equation*}
The first penalty term biases the feasibility error; while, in contrast to standard \mbox{augmented} Lagrangian ($\nu=0$), the extra second penalty term biases the optimality error. 
We will~first show that the inner product between the \textit{exact} Newton direction $(\Delta\bx_t, \Delta\blambda_t)$ and the~augmented Lagrangian gradient $\nabla\mL_{\mu, \nu}$, with proper parameters $\mu$ and $\nu$, is sufficiently negative~(Lemma \ref{lem:2}). Thus, the Newton direction is a descent direction of~$\mL_{\mu, \nu}$.~Then,~we~will~show that the~augmented Lagrangian $\mL_{\mu, \nu}$ decreases at each step even with an \textit{inexact} Newton~\mbox{direction}~(Lemma \ref{lem:3}). This implies that the residual $\|\nabla\mL_t\|$ finally vanishes to zero (Theorem \ref{thm:1}).

\subsection{Assumptions and preliminary results}\label{sec:3.1}

We state the following assumptions that are standard and proposed in the optimization~litera- ture \citep{Kushner1978Stochastic,Bertsekas1982Constrained, Nocedal2006Numerical, Na2022adaptive}.

\begin{assumption}\label{ass:1}

We assume the existence of a closed, bounded, convex set $\mX\times \Lambda$ containing the iterates $\{(\bx_t, \blambda_t)\}_t$, such that $f$ and $c$ are twice continuously differentiable over $\mX$.~We~also assume that the Hessian $\nabla^2\mL$ is $\Upsilon_L$-Lipschitz continuous over $\mX\times \Lambda$. In other words,
\begin{equation}\label{Lip:mL}
\|\nabla^2\mL(\bx, \blambda) - \nabla^2\mL(\bx', \blambda')\| \leq\Upsilon_L \|(\bx-\bx', \blambda-\blambda')\|, \quad \forall\; (\bx, \blambda), (\bx', \blambda')\in \mX\times \Lambda.
\end{equation}
Furthermore, we assume that the constraints Jacobian $G_t$ has full row rank with $G_tG_t^T\succeq \gamma_{G} I$ for a constant $\gamma_{G} >0$. Additionally, the regularization $\Delta_t$ ensures that $B_t$ \mbox{satisfies}~\mbox{$\|B_t\|\leq \Upsilon_B$} and $\bx^TB_t\bx\geq \gamma_{RH}\|\bx\|^2$ for any $\bx\in\{\bx\in \mR^d: G_t\bx = \0\}$, for some constants~$\gamma_{RH}, \Upsilon_B>0$.\;\;\;\;
	
\end{assumption}

Assumption \ref{ass:1} assumes $G_t$ has full row rank, which is referred to as the linear~independence constraint qualification (LICQ). LICQ is a common constraint qualification ensuring~the uniqueness of the dual solution, and is also necessary for the inference analysis (see \cite[Assumption B]{Duchi2021Asymptotic} and \cite[Example 2.1]{Davis2024Asymptotic}).~LICQ~and~\mbox{conditions}~on~$B_t$ are critical for SQP methods; they imply QP \eqref{equ:QP} has a unique solution \cite[Lemma 16.1]{Nocedal2006Numerical}. The Lipschitz continuity of the Hessian matrix is also standard for~analyzing Newton's method.

{ 
The bounded iterates condition is commonly assumed in the literature on (stochastic) nonlinear nonconvex  optimization, both for first-order gradient-based methods \citep{Bolte2013Proximal, Song2014Weak, Davis2016sound,Davis2019Stochastic, Atchade2017perturbed,Asi2019Stochastic, Liu2023Statistical} and for second-order SQP methods (\citep[Proposition 4.15]{Bertsekas1982Constrained}, \citep[Theorem 18.3]{Nocedal2006Numerical}).~This assumption ensures that all functions with their gradients and Hessians are bounded over $\mX \times \Lambda$ as long as they are smooth.~Some literature replaces~the 
bounded iterates condition by directly \mbox{imposing}~\mbox{boundedness} on~the~\mbox{gradients}~and~\mbox{Hessians}~of the objective and constraints, although the main use of the condition in the proof is rather~similar \citep{Berahas2021Sequential, Berahas2023Stochastic, Curtis2021Inexact, Ramprasad2022Online, Liu2023Online}.\;\;\;

We provide two justifications for the boundedness condition.~First, in our study, the~StoSQP iterates presumably track a deterministic feasible set $\{\bx \in \mR^d: c(\bx) = \0\}$, so we believe that an unbounded iteration sequence is generally rare especially when the feasible set is bounded. Second, a practical way to enforce the boundedness condition may be through adaptive truncation \citep{Andrieu2005Stability, Liang2010Trajectory}.~Under some~\mbox{conditions}~on~the~Markov~\mbox{transition}~\mbox{kernel} of the iteration sequence, one can show that the truncation occurs only finitely many~times,~ensuring that the convergence and asymptotic behavior are finally not affected by the truncation.$\quad\;\;$
}

\vskip-0.3cm
We also impose bounded moment conditions on the stochastic estimates $\barg_t$ and $\barH_t$.$\qquad\qquad\qquad$

\begin{assumption}\label{ass:2}
We assume $\mE[\barg_t\mid \bx_t] = \nabla f_t$, $\mE[\barH_t \mid \bx_t] = \nabla^2f_t$, and assume the following moment conditions \textit{when needed}: for a constant $\Upsilon_m>0$, 
\begin{subequations}\label{equ:BM}
\begin{align}
\text{gradient}\quad (\text{bounded 2nd moment}):\quad\quad &\mE[\|\barg_t - \nabla f_t\|^2\mid \bx_t]\leq \Upsilon_m, \label{equ:BM:a}\\
(\text{bounded 3th moment}):\quad\quad &\mE[\|\barg_t - \nabla f_t\|^3\mid \bx_t]\leq \Upsilon_m, \label{equ:BM:b}\\
(\text{bounded 4th moment}):\quad\quad &\mE[\|\barg_t - \nabla f_t\|^4\mid \bx_t]\leq \Upsilon_m, \label{equ:BM:c}
\adjintertext{0cm}{0cm}{and}
\text{Hessian}\quad  (\text{bounded 2nd moment}): \quad\quad & \mE[\|\barH_t - \nabla^2 f_t\|^2\mid \bx_t]\leq \Upsilon_m, \label{equ:BM:d}\\
(\text{bounded 2nd moment}):\quad\quad & \mE[\sup_{\bx\in\mX}\|\nabla^2f(\bx;\xi)\|^2] \leq \Upsilon_m. \label{equ:BM:e}
\end{align}
\end{subequations} 
\end{assumption}

We write $\mE[\cdot\mid \bx_t]$ to express out the conditional variable.~It can also be written~as~\mbox{$\mE[\cdot\mid \mF_{t-1}]$}, meaning the expectation is taken over the randomness of sample $\xi_t$. For conditions \eqref{equ:BM}, we do not impose all of them at once, but impose them step by step. In this section, we only require \eqref{equ:BM:a} to show the convergence of $\nabla\mL_t$.~In the next section,~we~\mbox{require}~\mbox{higher-order}~\mbox{moments}~for inference. In fact, \eqref{equ:BM:c} implies \eqref{equ:BM:b}, which implies \eqref{equ:BM:a}, and \eqref{equ:BM:e} implies~\eqref{equ:BM:d}. $\quad\quad\quad$

Assumption \ref{ass:2} is standard for uncertainty quantification of stochastic methods. We~would like to mention that \eqref{equ:BM:e} is also required for the asymptotic analysis of averaged SGD \citep{Chen2020Statistical}, which ensures the Lipschitz continuity of the mapping $\bx\rightarrow \mE[\nabla f(\bx; \xi)\nabla^T f(\bx; \xi)]$, as proved in \eqref{Lip:map}. Please refer to \cite[Assumption 3.2(2) and Lemma 3.1]{Chen2020Statistical} for further discussions. Moreover, \eqref{equ:BM:e} is satisfied by various objectives, such~as~\mbox{logistic}~and~least squares losses in Example \ref{exp:1}, as long as the feature variable $\xi_{\ba}$ has bounded $4$-th moment.\;\;\;\;\;

In terms of the sketching matrices, we need the following assumption.

\begin{assumption}\label{ass:3}
For $t\geq 0$, we assume that the sketching matrices $S_{t,j}\stackrel{iid}{\sim} S$ satisfy
\begin{equation*}
\mE[K_tS(S^TK_t^2S)^\dagger S^TK_t \mid \bx_t,\blambda_t]\succeq \gamma_{S} I\quad \quad \text{ for some } \gamma_{S}>0.
\end{equation*}
\end{assumption}

Assumption \ref{ass:3} is required for sketching solvers to converge in expectation \cite[Theorem 4.6]{Gower2015Randomized}.~This assumption can be easily verified for various~\mbox{sketching}~matrices.~For example, for randomized Kaczmarz method where {$S\in\mR^{(d+m)\times s}$ has $s$~columns~sampled uniformly (without replacement) from the canonical bases} $\{\be_1,\ldots,\be_{d+m}\}$ \citep{Strohmer2008Randomized}, we have $\hskip20cm$\vskip-0.5cm
\begin{multline}\label{nnequ:1}
\mE[K_tS(S^TK_t^2S)^\dagger S^TK_t \mid \bx_t,\blambda_t] \succeq \frac{\mE[K_tS S^TK_t \mid \bx_t,\blambda_t]}{\sigma_{\max}(K_t^2)} \\ = \frac{{s} K_t^2}{(d+m) \sigma_{\max}(K_t^2)} \succeq \frac{{s}I}{(d+m)\kappa (K_t^2)},
\end{multline}
where {$\sigma_{\max}(K_t^2)$ denotes the largest singular value (which is the same as the largest eigenvalue in this case) of $K_t^2$ and}~$\kappa(K_t^2)$ denotes the condition number of $K_t^2$ (it is independent of~$t$~by Assumption \ref{ass:1}).~{The first inequality is by the eigenvalue interlacing \mbox{theorem},~which leads~to $\sigma_{\max}(S^TK_t^2S)\leq \sigma_{\max}(K_t^2)$, and the second equality is by the sampling mechanism:$\qquad\qquad$\vskip-0.15cm
\begin{equation*}
\mE[SS^T] = \frac{1}{\begin{pmatrix}
d+m\\
s
\end{pmatrix}}\sum_{\mA \in \\ \{\text{sets of $s$ indices}\}} I_{\mA} = \frac{\begin{pmatrix}
d+m-1\\
s-1
\end{pmatrix}}{\begin{pmatrix}
d+m\\
s
\end{pmatrix}} I = \frac{s}{d+m} I.
\end{equation*}
Here, $I_\mA\in\mR^{(d+m)\times (d+m)}$ is a diagonal matrix with $[I_\mA]_{i,i}=1$ if the index $i\in \mA$ and~$[I_\mA]_{i,i}=0$ otherwise.~The set $\mA$ contains $s$ distinct indices.~We note that a recent study \cite[\hskip-2pt(1.3)]{Derezinski2024Sharp} showed a similar lower bound to~\eqref{nnequ:1}~for~\mbox{Gaussian}~\mbox{sketching}.~The~\mbox{authors} improved the denominator from $(d+m) \sigma_{\max}^2(K_t)$ to $\|K_t\|_F^2$, though the latter still grows~linearly with the problem dimension without additional spectral decay assumptions.$\qquad\qquad\qquad$} 

Assumption \ref{ass:3} directly leads to the~following~result.\;\;

\begin{lemma}[Guarantees of sketching solvers]\label{lem:1}
	
Under Assumption \ref{ass:3}, for all $t\geq 0$:
\begin{enumerate}[topsep=1pt,itemsep=0em,partopsep=-2pt,parsep=1ex,label=(\alph*):,beginpenalty=10000]
\item Let $\rho = 1-\gamma_{S}$. We have $0\leq \rho < 1$.
\item $\mE[\bz_{t,\tau} - \tbz_t\mid \mF_{t-2/3}] = -(I - \mE[K_tS(S^TK_t^2S)^\dagger S^TK_t \mid \mF_{t-1}])^\tau\tbz_t \eqqcolon C_t\tbz_t$, and $\|C_t\|\leq \rho^\tau$. 
\item $\mE[\|\bz_{t,\tau} - \tbz_t\|^2\mid \mF_{t-2/3}] \leq \rho^\tau \|\tbz_t\|^2$.
\end{enumerate}
	
\end{lemma}

\newpage

{ 
\begin{remark}

We note that the linear convergence rate $\rho$ of the expected error of the~sketching solver depends on the lower bound $\gamma_S \in (0, 1]$ of the projection matrix, which is proportional~to the sketching dimension $s$ and inversely proportional to the problem dimension $d+m$, as~shown in \eqref{nnequ:1}. The condition number $\kappa^2(K_t)$ is assumed to be uniformly bounded in our study.$\quad\;\;\;\;\;$

By the above relation, and given an error threshold $\delta$, we have $\rho^\tau = (1-\gamma_S)^\tau\leq \delta \Longleftrightarrow\tau\geq \log(1/\delta)/\log(1/\{1-\gamma_S\}) = O(1/\gamma_S) = O((d+m)/s)$.~This implies that,~to~decay the expected error below a threshold, the number of sketching steps $\tau$ is proportional~to~the~problem dimension $d+m$ and inversely proportional to the sketching dimension $s$.~Certainly,~a~larger~sketching dimension leads to a higher computational cost per step in \eqref{equ:pseduo}. Using the Kaczmarz~method as an example, the flops per step are $O((d+m)s^2)$ (dominated by the computation of $S^TK_t^2S$). Thus, the total flops over $\tau$ steps are $O((d+m)^2s)$, indicating that a smaller~sketching~dimension is generally preferable.~That said, this analysis only reflects a worst-case scenario~under~the presumption that $\kappa^2(K_t)$ is uniformly bounded. The optimal choice of~$s$~is~often~\mbox{case-by-case} 
and depends on whether $K_t$ exhibits a particular sparsity or eigenvalue decay structure~\citep{Derezinski2024Sharp}.

\end{remark}
}

\subsection{Almost sure convergence}\label{sec:3.2}

We now set the stage to show global almost sure convergence. The first result shows~that~the exact Newton direction $(\Delta\bx_t, \Delta\blambda_t)$ is a descent direction of $\mL_{\mu, \nu}^t = \mL_{\mu, \nu}(\bx_t, \blambda_t)$ if $\mu$ is~\mbox{sufficiently} large and $\nu$ is sufficiently small.

\begin{lemma}\label{lem:2}

Under Assumption \ref{ass:1}, there exists a deterministic constant $\Upsilon_1>0$, depending only on $(\gamma_{G}, \gamma_{RH}, \Upsilon_B)$, such that
\begin{equation*}
\begin{pmatrix}
\nabla_{\bx}\mL_{\mu, \nu}^t\\
\nabla_{\blambda}\mL_{\mu, \nu}^t
\end{pmatrix}^T\begin{pmatrix}
\Delta\bx_t\\
\Delta\blambda_t
\end{pmatrix} \leq - \frac{\nu}{\Upsilon_1}\cbr{\nbr{\begin{pmatrix}
\Delta\bx_t\\
\Delta\blambda_t
\end{pmatrix}}^2 + \nbr{\begin{pmatrix}
\nabla_{\bx}\mL_t\\
c_t
\end{pmatrix}}^2} ,
\end{equation*}
provided $\mu\nu \geq \Upsilon_1$ and $\nu \leq 1/\Upsilon_1$.
\end{lemma}

With Lemmas \ref{lem:1} and \ref{lem:2}, we are able to show the following one-step recursion of $\mL_{\mu, \nu}^t$.\;\;\;

\begin{lemma}\label{lem:3}
Under Assumptions \ref{ass:1}, \ref{ass:2}(\ref{equ:BM:a}), \ref{ass:3}, we suppose that the pair $(\mu, \nu)$ satisfies the condition in Lemma \ref{lem:2} and $\tau$ satisfies $\rho^\tau\leq \nu/(\mu\Upsilon_1)$. Then, there exists a deterministic constant $\Upsilon_2> 0$, depending on $(\gamma_{G},\gamma_{RH}, \Upsilon_B, \Upsilon_L, \Upsilon_m)$, such that
\begin{equation*}
\mE[\mL_{\mu, \nu}^{t+1}\mid \mF_{t-1}] \leq \mL_{\mu, \nu}^t - \frac{\nu}{2\Upsilon_1}\cdot\beta_t\|\nabla\mL_t\|^2 + \Upsilon_2(\chi_t +\eta_t^2).
\end{equation*}
	
\end{lemma}

With Lemma \ref{lem:3}, we can apply Robbins-Siegmund theorem \citep{Robbins1971convergence} to establish the convergence of the KKT residual $\|\nabla\mL_t\|$.

\begin{theorem}[Global convergence]\label{thm:1}

Consider Algorithm \ref{alg:1} under Assumptions \ref{ass:1}, \ref{ass:2}(\ref{equ:BM:a}), \ref{ass:3}. Suppose we perform the sketching solver (\ref{equ:pseduo}) for $\tau$ steps with $\tau \geq 4\log \Upsilon_1/\log\{1/(1-\gamma_{S})\}$, where $\Upsilon_1$ is from Lemma \ref{lem:2}. Also, we let $\{\beta_t, \eta_t = \beta_t+\chi_t\}$ satisfy
\begin{equation}\label{cond:step}
\sum_{t=0}^{\infty} \beta_t = \infty, \quad\quad \sum_{t=0}^{\infty} \beta_t^2 <\infty, \quad\quad \sum_{t=0}^{\infty}\chi_t  <\infty.
\end{equation}
Then, we have $\|(\bx_{t+1}-\bx_t, \blambda_{t+1}-\blambda_t)\|\rightarrow0$  and $\|\nabla\mL_t\|\rightarrow 0$ as $t\rightarrow\infty$ almost surely.
\end{theorem}

Theorem \ref{thm:1} indicates that all limiting points of the primal-dual iteration sequence~$(\bx_t, \blambda_t)$ are stationary.~Our global convergence guarantee aligns with the guarantee of \mbox{deterministic}~SQP methods \cite[Theorem 18.3]{Nocedal2006Numerical}, as well as the guarantee of recent stochastic SQP methods \citep{Na2022adaptive, Na2023Inequality}, despite the fact that Algorithm \ref{alg:1} \mbox{possesses}~an~additional source of randomness from the sketching solver at each step. The convergence of $(\bx_t, \blambda_t)$ is equivalent to the convergence of $\nabla\mL_t$ if Problem \eqref{pro:1} is convex. Furthermore, the results $\|\nabla\mL_t\|\vee \|(\bx_{t+1}-\bx_t, \blambda_{t+1}-\blambda_t)\| \stackrel{a.s.}{\longrightarrow}0$ imply the existence of an attraction region around~the local solution $(\bx^\star, \blambda^\star)$. Once $(\bx_t, \blambda_t)$ lies in the neighborhood, all subsequent iterates will stay in the neighborhood and $(\bx_t, \blambda_t)\stackrel{a.s.}{\longrightarrow} (\tx, \tlambda)$ \cite[\mbox{Chapter}~4.4]{Bertsekas1982Constrained}.

Based on Theorem \ref{thm:1}, we can immediately show the worst-case iteration complexity.~Due to the online nature of the method, the iteration complexity is equivalent to~the~sample~complexity, as we observe one sample in each iteration.

\begin{corollary}\label{cor:1}

Consider Algorithm \ref{alg:1} under Assumptions \ref{ass:1}, \ref{ass:2}(\ref{equ:BM:a}), \ref{ass:3}. Suppose $\tau$ satisfies the condition in Theorem \ref{thm:1}, and let $\beta_t = (t+1)^{-a}$, $\chi_t = (t+1)^{-b}$ where $a\in(0, 1)$ and $a<b$. Also, define $\T_{\epsilon} = \inf_{t}\cbr{t\geq 1: \mE[\|\nabla\mL_t\|] \leq \epsilon}$. Then, we have
\begin{equation*}
\T_{\epsilon} = O\rbr{\epsilon^{-\frac{2}{a\wedge (1-a)\wedge (b-a)}}}.
\end{equation*}
In particular, $\T_{\epsilon}$ attains the minimum $O(\epsilon^{-4})$ with $a = 1/2$ and $b = 1$.
\end{corollary}

We should mention that a recent work \cite{Curtis2023Worst} also showed an $O(\epsilon^{-4})$ iteration complexity with Newton systems being exactly solved. Our result matches theirs~while~allowing for the use of sketching solvers to inexactly solve Newton systems. We highlight~that~Corollary \ref{cor:1} is based on a \textit{non-asymptotic} convergence rate of the averaged expected KKT residual $\frac{1}{t}\sum_{i=0}^{t-1}\mE[\|\nabla\mL_i\|]$. This non-asymptotic result is in contrast to our main inference analysis in Section \ref{sec:4}, where the results hold asymptotically.

\section{Statistical Inference via StoSQP}\label{sec:4}

We perform online statistical inference for Problem \eqref{pro:1} by leveraging StoSQP. To segue into inference analysis, we suppose in this section that the method converges~to a local solution $(\tx, \tlambda)$ of \eqref{pro:1}; specifically, $G^\star = \nabla c^\star$ has full row rank and $\nabla_{\bx}^2\mL^\star$ is positive definite in~the null space $\{\bx\in\mR^d: G^\star\bx = \0\}$. These optimality conditions ensure~that~the~Lagrangian~Hessian $K^\star = \nabla^2\mL^\star$ is non-singular, as necessary for $M$-estimators in \eqref{equ:M:est}.

\subsection{Iteration recursion}\label{sec:4.1}

From a high-level view, our method generates a stochastic sequence
\begin{equation}\label{equ:Markov}
\begin{pmatrix}
\bx_{t+1} - \bx^\star\\
\blambda_{t+1} - \blambda^\star
\end{pmatrix} = (1 - \baralpha_t) \begin{pmatrix}
\bx_t - \bx^\star\\
\blambda_t - \blambda^\star
\end{pmatrix} + \baralpha_t\begin{pmatrix}
\btheta_{\bx}^t\\
\btheta_{\blambda}^t
\end{pmatrix} + \baralpha_t\begin{pmatrix}
\bdelta_{\bx}^t\\
\bdelta_{\blambda}^t
\end{pmatrix},
\end{equation}
where $(\btheta_{\bx}^t, \btheta_{\blambda}^t)$ is a martingale difference with $\mE[(\btheta_{\bx}^t, \btheta_{\blambda}^t) \mid \mF_{t-1}] = \0$, and $(\bdelta_{\bx}^t, \bdelta_{\blambda}^t)$ is~the~remaining error term.~Compared to existing stochastic first- and second-order methods~\citep{Chen2020Statistical,Bercu2020Efficient,Duchi2021Asymptotic, Davis2024Asymptotic, Boyer2023asymptotic}, the randomness brought by the adaptivity and inexactness (AI)~of~the~method affects all the terms in \eqref{equ:Markov}. This includes the random stepsize $\baralpha_t$, as well as the random approximation errors in $(\btheta_{\bx}^t, \btheta_{\blambda}^t)$ and $(\bdelta_{\bx}^t, \bdelta_{\blambda}^t)$ associated with the sketching solver.

We formalize the recursion \eqref{equ:Markov} in the following lemma.

\begin{lemma}\label{lem:4}
Let $\varphi_t = (\beta_t+\eta_t)/2$. The iteration sequence of Algorithm \ref{alg:1} can be expressed~as$\quad$
\begin{equation*}
\begin{pmatrix}
\bx_{t+1} - \bx^\star\\
\blambda_{t+1} - \blambda^\star
\end{pmatrix} = \I_{1,t} + \I_{2,t} + \I_{3,t}
\end{equation*}
where\allowdisplaybreaks
\begin{subequations}\label{rec}
\begin{align}
\I_{1,t} & = \sum_{i=0}^t\prod_{j=i+1}^{t}\cbr{I - \varphi_j(I+C^\star)}\varphi_i\begin{pmatrix}
\btheta_{\bx}^i\\
\btheta_{\blambda}^i
\end{pmatrix}, \label{rec:a}\\
\I_{2,t} & = \sum_{i=0}^t\prod_{j=i+1}^{t} \cbr{I - \varphi_j(I+C^\star)}\rbr{\baralpha_i - \varphi_i}\begin{pmatrix}
\barDelta\bx_i\\
\barDelta\blambda_i
\end{pmatrix}, \label{rec:b}\\
\I_{3,t} & = \prod_{i=0}^{t}\cbr{I - \varphi_i(I+C^\star)}\begin{pmatrix}
\bx_0 - \bx^\star\\
\blambda_0 - \blambda^\star
\end{pmatrix} + \sum_{i=0}^t\prod_{j=i+1}^{t}\cbr{I - \varphi_j(I+C^\star)}\varphi_i\begin{pmatrix}
\bdelta_{\bx}^i\\
\bdelta_{\blambda}^i
\end{pmatrix}, \label{rec:c}
\end{align}
\end{subequations}
and
\begin{subequations}\label{rec:def}
\begin{flalign}
&\hskip-0.48cm C^\star = -(I - \mE[K^\star S(S^T(K^\star)^2S)^\dagger S^TK^\star])^\tau, \label{rec:def:a}\\ 
&\hskip-1.1cm  \begin{pmatrix}
\btheta_{\bx}^i\\
\btheta_{\blambda}^i
\end{pmatrix}  = -(I+C_i)K_i^{-1}\begin{pmatrix}
\barg_i - \nabla f_i\\
\0
\end{pmatrix} + \cbr{\begin{pmatrix}
\barDelta\bx_i\\
\barDelta\blambda_i
\end{pmatrix} - (I + C_i)\begin{pmatrix}
\tDelta\bx_i\\
\tDelta\blambda_i
\end{pmatrix}}, \label{rec:def:b}\\ 
&\hskip-1.1cm \begin{pmatrix}
\bdelta_{\bx}^i\\
\bdelta_{\blambda}^i
\end{pmatrix}  = -(I+C_i)\cbr{(K^\star)^{-1}\begin{pmatrix}
\bpsi_{\bx}^i\\
\bpsi_{\blambda}^i
\end{pmatrix} + \cbr{K_i^{-1} - (K^\star)^{-1}}\begin{pmatrix}
\nabla_{\bx}\mL_i\\
c_i
\end{pmatrix} } - (C_i-C^\star)\begin{pmatrix}
\bx_i - \bx^\star\\
\blambda_i - \blambda^\star
\end{pmatrix},\label{rec:def:c}\\ 
&\hskip-1.2cm \begin{pmatrix}
\bpsi_{\bx}^i\\
\bpsi_{\blambda}^i
\end{pmatrix} = \begin{pmatrix}
\nabla_{\bx}\mL_i\\
c_i
\end{pmatrix} - K^\star\begin{pmatrix}
\bx_i-\bx^\star\\
\blambda_i-\blambda^\star
\end{pmatrix}. \label{rec:def:d} 
\end{flalign}
\end{subequations}
Under Assumptions \ref{ass:2}, \ref{ass:3}, $\btheta^i = (\btheta_{\bx}^i, \btheta_{\blambda}^i)$ is a martingale difference~with $\mE[\btheta^i\mid \mF_{i-1}] = \0$.	
\end{lemma}

From Lemma \ref{lem:4}, we observe that the recursion consists of three terms. $\I_{1,t}$ is a~martingale that accounts for the randomness of sampling $\xi_t$ to estimate $\nabla f_t$ and the randomness of sketching $\{S_{t,j}\}_j$ to solve QP \eqref{equ:QP}. $\I_{2,t}$ captures the randomness of the stepsize $\baralpha_t$. $\I_{3,t}$ contains all~the~remainder terms. The asymptotic analysis of each term is provided~in~Appendix~\ref{pf:thm:3}.

Next, we establish a continuity property for the projection matrix $K_tS(S^TK_t^2S)^\dagger S^TK_t$, a critical quantity of the sketching solver appeared in $C_t$ and $C^\star$ (cf. Lemma \ref{lem:1}(b) and~\eqref{rec:def:a}).

\begin{lemma}\label{lem:5}
Suppose $K_t, K^\star \in \mR^{(d+m)\times(d+m)}$ are non-singular. For any $S\in \mR^{(d+m)\times s}$,
\begin{equation*}
\|K_tS(S^TK_t^2S)^\dagger S^TK_t - K^\star S(S^T(K^\star)^2S)^\dagger S^TK^\star\| \leq \frac{2\|K_t-K^\star\|}{\sigma_{\min}(K^\star)}\cdot\|S\|\|S^\dagger\|,
\end{equation*}
where $\sigma_{\min}(\cdot)$ denotes the least singular value.
\end{lemma}

Lemma \ref{lem:5} indicates that the difference between the projection matrices $K_tS(S^TK_t^2S)^\dagger S^TK_t $ and $K^\star S(S^T(K^\star)^2S)^\dagger S^TK^\star$ is proportional to the difference between the Hessian matrices $K_t$ and $K^\star$, with a random factor scaling with the condition number of the sketching~matrix~$S$. 
In practice, using sketching vectors ($s=1$) can reduce computational cost and result in~a~unit condition number $\|S\|\|S^\dagger\|= 1$. 

Lemma~\ref{lem:5} leads to the following \mbox{condition}~to~\mbox{ensure}~the~convergence of $C_t$, which is the expectation of the product of projection matrices.$\hskip2.5cm$

\begin{assumption}\label{ass:5}
We assume that $S$ satisfies $\mE[\|S\|\|S^\dagger\|] \leq \Upsilon_S$ for a constant $\Upsilon_S>0$.
\end{assumption}

\begin{corollary}\label{cor:2}
Under Assumption \ref{ass:5}, $\|C_t-C^\star\| \leq 2\tau\Upsilon_S\|K_t - K^\star\|/\sigma_{\min}(K^\star)$.
\end{corollary}

\subsection{Asymptotic rate and normality}\label{sec:4.2}

We are now ready to state inference theory. Let $S_1,\ldots,S_\tau\stackrel{iid}{\sim} S$, and define a random~matrix:
\begin{equation}\label{equ:tC}
\tC^\star = -\prod_{j=1}^{\tau}(I - K^\star S_j(S_j^T(K^\star)^2S_j)^\dagger S_j^TK^\star ).
\end{equation}
Clearly, $\mE[\tC^\star] = C^\star$. Also, define the sandwich matrix that appears as the limiting~covariance of $M$-estimators in \eqref{equ:M:est}:
{\small\begin{align}\label{equ:Omega}
\hskip-0.4cm \tOmega =  (K^\star)^{-1}\cov(\bnabla\mL^\star)(K^\star)^{-1} = \begin{pmatrix}
\nabla_{\bx}^2\mL^\star & (G^\star)^T\\
G^\star & \0
\end{pmatrix}^{-1}\begin{pmatrix}
\cov(\nabla F(\tx; \xi)) & \0\\
\0 & \0
\end{pmatrix}\begin{pmatrix}
\nabla_{\bx}^2\mL^\star & (G^\star)^T\\
G^\star & \0
\end{pmatrix}^{-1}.
\end{align}}
\hskip-4pt To allow for general stepsize control sequences $\{\beta_t, \chi_t\}_t$, we define three quantities: \begin{equation}\label{equ:three:quant}
\beta \coloneqq \lim\limits_{t\rightarrow\infty} t\rbr{1 - \frac{\beta_{t-1}}{\beta_t}}, \quad\quad\quad \tbeta \coloneqq \lim\limits_{t\rightarrow\infty} t\beta_t, \quad\quad\quad  \chi \coloneqq \lim\limits_{t\rightarrow\infty}t\rbr{1 - \frac{\chi_{t-1}}{\chi_t}}.
\end{equation}
The polynomial sequences $1/t^\omega$ are specialized in Lemma \ref{lem:10}. For sake of understanding,~we here mention that if $\beta_t=O(1/t^\omega)$ for any $\omega>0$, then we simply have $\beta = -\omega$.

\begin{theorem}[Local convergence rate]\label{thm:3}
Under Assumptions \ref{ass:1}, \ref{ass:2}(\ref{equ:BM:a}, \ref{equ:BM:e}), \ref{ass:3}, we suppose $\{\beta_t, \chi_t\}_t$ satisfy
\begin{equation}\label{cond:n:1}
\chi<\beta <0,\quad\quad \tbeta\in(0, \infty], \quad\quad 1.5(1-\rho^\tau) + \beta/\tbeta >0.
\end{equation}
Then, for any $\upsilon>0$ and any constant $p\in(0,1]$ such that $(1-\rho^\tau) + p(\chi - 0.5\beta)/\tbeta>0$,~we~have
(if $p<1$, the second $O(\cdot)$ in the following results can be strengthened to $o(\cdot)$)
\begin{equation*}
\|(\bx_t- \bx^\star, \blambda_t - \blambda^\star)\| = o\big(\sqrt{\beta_t\{\log(1/\beta_t)\}^{1+\upsilon}}\;\big) + O(\chi_t^p/\beta_t^p)\quad a.s.
\end{equation*}
Furthermore, if (\ref{equ:BM:a}) is strengthened to (\ref{equ:BM:b}), then
\begin{equation*}
\|(\bx_t- \bx^\star, \blambda_t - \blambda^\star)\|  = O\big(\sqrt{\beta_t\log(1/\beta_t)}\;\big) + O(\chi_t^p/\beta_t^p)\quad a.s.
\end{equation*}

\end{theorem}

The asymptotic convergence rate consists of two terms. The first term $o(\sqrt{\beta_t\{\log(1/\beta_t)\}^{1+\upsilon}})$ comes from the strong law of large number for the martingale $\I_{1,t}$, which can be strengthened to $O(\sqrt{\beta_t\log(1/\beta_t)})$ if the stochastic gradient estimate $\barg_t$ has a bounded moment of order higher than two \citep[Theorem 1.3.15]{Duflo1997Random}. That is, the bounded 3rd~\mbox{moment}~in~\eqref{equ:BM:b}~can be directly replaced by a bounded $2+\delta$ moment. 
The second term $O(\chi_t^p/\beta_t^p)$ comes from~$\I_{2,t}$, which characterizes the adaptivity of the random stepsize $\baralpha_t$. This term is suppressed~if~we~degrade the method to a non-adaptive one ($\chi_t=0$).

We will investigate the condition \eqref{cond:n:1} in Lemma \ref{lem:10} and demonstrate that it is~week~enough to allow for different setups of sequences $\{\beta_t, \chi_t\}_t$. Most importantly, the condition \eqref{cond:n:1}~covers the setup of $\beta_t=1/t$, corresponding to $\beta=-1$ and $\tbeta=1$, which leads to the optimal central limit theorem rate (cf. Theorem \ref{thm:4}). 
In fact, \eqref{cond:n:1} reveals a relationship between~the inexactness of the sketching solver (i.e., the parameter $\tau$) and the setup of the stepsize. 
When $\tbeta = \infty$ (e.g., $\beta_t=1/t^\omega$ for $\omega<1$), the third condition in \eqref{cond:n:1} holds trivially.

Furthermore, we note that $(1-\rho^\tau) + p(\chi - 0.5\beta)/\tbeta>0$ is always satisfied for small~$p$~(specifically, $p=1$ if $\tbeta = \infty$). 
Thus, our condition on the adaptivity gap $\chi_t$ is simply $\chi_t=o(\beta_t)$~(i.e., $\chi<\beta$). We have to strengthen this condition to $\chi_t=o(\beta_t^{1.5})$ to enable inference analysis in Theorem \ref{thm:4}. Even that, it's worth mentioning that the methods proposed in \cite{Berahas2021Sequential, Berahas2023Stochastic, Curtis2021Inexact} all set $\chi_t = O(\beta_t^2)$ (i.e., $\chi\leq 2\beta$). Our analysis relaxes these designs to allow for a larger $\chi_t$, resulting in a wider interval \eqref{equ:sandwich} for stepsize adaptivity.$\quad\quad\;$

\begin{theorem}[Asymptotic normality]\label{thm:4}

Under Assumptions \ref{ass:1}, \ref{ass:2}(\ref{equ:BM:b}, \ref{equ:BM:e}), \ref{ass:3}, \ref{ass:5},~we strengthen $\chi<\beta$ in the condition (\ref{cond:n:1}) to $\chi<1.5\beta$. Then, we have
\begin{equation}\label{nequ:normality}
\sqrt{1/\baralpha_t}\cdot( \bx_t - \bx^\star, \blambda_t-\blambda^\star)\stackrel{d}{\longrightarrow}\N(\0, \Xi^\star),
\end{equation}
where $\Xi^\star$ is the solution of the following Lyapunov equation:
\begin{equation}\label{equ:Lya:equ}
\big(\{1+\beta/(2\tbeta) \} I +C^\star\big) \Xi^\star + \Xi^\star \big(\{1+\beta/(2\tbeta) \} I +C^\star\big) = \mE[ (I+\tC^\star)\tOmega(I+\tC^\star)^T].
\end{equation}
Furthermore, let us define $q = 1$ if $1-\rho^\tau + (\chi-\beta)/\tbeta>0$, otherwise $q = \{(1-\rho^\tau)\tbeta + \epsilon\beta\}/(\beta - \chi)\in(0, 1)$ for any $\epsilon\in(0, 1/6]$. Then, for any $\bw = (\bw_{\bx}, \bw_{\blambda})\in \mR^{d+m}$ such~that~$\bw^T\Xi^\star\bw \neq 0$,
\begin{multline}\label{equ:BE:Inequality}
\sup_{z\in\mR}\abr{P\rbr{\frac{\sqrt{1/\baralpha_t}\cdot \bw^T(\bx_t - \bx^\star,\blambda_t-\blambda^\star)}{\sqrt{\bw^T\Xi^\star\bw}} \leq z}  - P\rbr{\N(0,1)\leq z}} \\
= o(\beta_t^{1/6}\log(1/\beta_t)) + O(\chi_t^q/\beta_t^{q+0.5}),
\end{multline}
where $O(\cdot)$ can be strengthened to $o(\cdot)$ if $q<1$.
\end{theorem}

{
We mention that, to our knowledge, \eqref{nequ:normality} provides the first primal-dual asymptotic normality result for a constrained online estimation procedure,\hskip1pt while existing works in~\cite{Duchi2021Asymptotic, Davis2024Asymptotic} have only established primal asymptotic normality.~Although the uncertainty quantification of primal variables $\tx$ has a natural \mbox{meaning}~as~they~\mbox{represent} model parameters, the uncertainty quantification of dual variables $\tlambda$ is also significant~in~two~ways~beyond the technical interest.
\begin{enumerate}[label=(\alph*),topsep=5pt,leftmargin=0.8\parindent,labelsep=2pt]
\setlength\itemsep{0.0em}
\item 
The dual variables are widely used as optimality certificates in algorithmic designs.~In particular, the normality of $(\bx_t, \blambda_t)$ enables the construction of a confidence region for the gradient vector $\nabla_{\bx}\mL(\bx_t, \blambda_t)$; it suggests that the optimality residual, $\|\nabla_{\bx}\mL(\bx_t,\hskip-1pt\blambda_t)\|^2$,~\mbox{exhibits} a (generalized) chi-squared limiting distribution.~See \cite[Section 4.1]{Jones1983Statistical} and \cite[Section 5.6.2]{Shapiro2014Lectures} for analogous constructions. By checking whether the region contains $\0$, we can determine whether we have sufficiently achieved the optimality condition at the given significance level and terminate the algorithm accordingly.$\hskip10cm$

\item 
Methods for solving inequality-constrained problems often involve equality-constrained subproblems\hskip1pt \citep{Na2023Inequality},\hskip1.4pt and constructing confidence intervals for dual variables~\mbox{associated} with inequality constraints is crucial for active-set identification under uncertainty. This paper serves as a first step toward that goal.~Specifically, let $\tlambda\in\mR^m$ be the dual solution associated with the constraints $c(\bx)\leq \0$.~By the strict complementarity condition (\cite{Davis2024Asymptotic}, Example 2.1), we know $\tlambda_i>0 \Leftrightarrow c_i(\tx)= 0$ and $\tlambda_i=0\Leftrightarrow c_i(\tx)<0$. Thus, if the confidence interval for $\tlambda_i$ contains 0, it suggests that~the $i$-th constraint is not identified as active at the given significance level, and vice versa.$\hskip10cm$

\end{enumerate}

}

The explicit form of the solution to \eqref{equ:Lya:equ} is given by
\begin{equation}\label{equ:Xi:n}
\Xi^\star = U\big(\Theta\circ U^T\mE[ (I+\tC^\star)\tOmega(I+\tC^\star)^T]\; U \big) U^T\quad \text{with}\quad  [\Theta]_{k,l} = 1/(\sigma_k+\sigma_l + \beta/\tbeta),
\end{equation}
where $I+C^\star = U\Sigma U^T$ with $\Sigma = \diag(\sigma_1, \ldots, \sigma_{d+m})$ is the eigenvalue decomposition of~$I+C^\star$, and $\circ$ denotes the matrix Hadamard product.

From Theorem \ref{thm:4}, we see that the limiting covariance $\Xi^\star$ depends on the sandwich~matrix $\tOmega$ in \eqref{equ:Omega}, which is the same as the optimal one for $M$-estimators in \eqref{equ:M:est}, but it also~depends on the underlying sketching distribution. The sketching matrices affect both the left- and right-hand sides of the~Lyapunov \mbox{equation}~\eqref{equ:Lya:equ}. If we degrade the randomized sketching~solver~to~an exact QP solver, then $\tau = \infty$, $C^\star = \tC^\star = \0$, and
\begin{equation}\label{equ:cov:exact}
\Xi^\star = \frac{\tOmega}{2+\beta/\tbeta}.
\end{equation}
We present the following corollary to further connect Theorem \ref{thm:4} with the asymptotic~\mbox{minimax} optimality of constrained estimation problems established in \cite[Theorem 1]{Duchi2021Asymptotic} and \cite[Theorem 3.2]{Davis2024Asymptotic}.

\begin{corollary}\label{cor:3}

Let $\beta_t=1/t$, $\chi_t = o(\beta_t^{1.5})$, and $\tau$ such that $\rho^\tau<1/3$. Then,
\begin{enumerate}[topsep=2pt,itemsep=0em,partopsep=-3pt,parsep=0.5ex,label=(\alph*):,beginpenalty=10000]
\item Exact QP solver: $\tXi = \tOmega$,
\item Sketching solver: $\tXi\succeq \tOmega$ but $\|\tXi - \tOmega\|\leq 3\rho^\tau\|\tOmega\|$.
\end{enumerate}

\end{corollary}

By Corollary \ref{cor:3}, the limiting covariance of StoSQP obtained by suppressing the sketching solver matches the asymptotic minimax optimum $\tOmega$ that is achieved by offline $M$-estimators and online projection-based estimators \citep{Duchi2021Asymptotic, Davis2024Asymptotic}.~To~the~best of our knowledge, our estimator based on StoSQP is the first online estimator that does~not~rely on projection operators; we instead replace \mbox{projection}~by~\mbox{employing}~a~\mbox{series} of \mbox{linear-quadratic} approximation of nonlinear problems. 
On the other hand, when employing the sketching~solver to inexactly solve QPs, the limiting covariance of our method exceeds the optimum $\tOmega$,~indicating that the sketching solver (to address the computation concern of second-order methods) indeed compromises optimality. Fortunately, this \mbox{compromise}~is~marginal~since~the~\mbox{distance} between $\tXi$ and $\tOmega$ decays exponentially fast with the number of iterations $\tau$ performed~for~the sketching solver.

The Berry-Esseen bound in \eqref{equ:BE:Inequality} consists of two terms. The first term is due to the random sample and random sketching, while the second term is due to the random stepsize.~Our~choice of $q$ always guarantees that $\chi_t^q/\beta_t^{q+0.5} = o(1)$. We note that $q=1$ when $\tbeta = \infty$ (e.g.,~$\beta_t = 1/t^\omega$ for $\omega<1$). 

{
\begin{remark}\label{rem:2}
Since existing studies in \cite{Duchi2021Asymptotic, Davis2024Asymptotic} heavily used projection notation and only established the normality of the primal estimator, we~further elucidate in this remark the connection between our joint covariance $\tOmega$ (i.e.,~StoSQP~with~exact QP solver) and the covariance in \cite[Corollary 5.2]{Davis2024Asymptotic} and \cite[Theorem 5]{Duchi2021Asymptotic}.

Recall that $G^\star = \nabla c(\tx)\in\mR^{m\times d}$ is the constraints Jacobian. Let $Z^\star\in\mR^{d\times (d-m)}$ be~a matrix whose columns are orthonormal and form the bases of $\text{ker}(\tG)$. Then, using~the~relation $\tG^T(\tG\tG^T)^{-1}\tG + \tZ\tZ^T = I_d$, we can verify that
\begin{equation}\label{equ:2}
\begin{pmatrix}
\nabla_{\bx}^2\mL^\star & \tG^T\\
\tG & \0_m 
\end{pmatrix}^{-1} = \begin{pmatrix}
A_1 & A_2^T\\
A_2 & A_3
\end{pmatrix}
\end{equation}
where
\begin{align*}
A_1 & = \tZ(\tZ^T\nabla_{\bx}^2\mL^\star\tZ)^{-1}\tZ^T,\quad A_2 = (\tG\tG^T)^{-1}\tG(I - \nabla_{\bx}^2\mL^\star A_1),\\
A_3 & = (\tG\tG^T)^{-1}\tG(\nabla_{\bx}^2\mL^\star A_1 \nabla_{\bx}^2\mL^\star - \nabla_{\bx}^2\mL^\star)(\tG\tG^T)^{-1}\tG.
\end{align*}
Note that all above matrix inverses are well-defined under the conditions~of~\cite[Example 2.1 and Section 5.1]{Davis2024Asymptotic} (or, equivalently, under our conditions).~Plugging~the~above~display into \eqref{equ:Omega}, we see that the marginal covariance of $\bx$ is
\begin{equation}\label{equ:1}
\Omega_{\bx,\bx}^\star = A_1\text{cov}(\nabla F(\tx;\xi))A_1.
\end{equation}
Furthermore, we note that $A_1 = (\tZ\tZ^T\nabla_{\bx}^2\mL^\star\tZ\tZ^T)^\dagger$ by verifying~the definition of~the~Moore-Penrose pseudoinverse. Since $\tZ\tZ^T = \text{Proj}_{\text{ker}(\tG)}$ and $\text{ker}(\tG)$ is the tangent space of the manifold $c(\bx)$ at $\tx$, we see \eqref{equ:1} matches the result in \cite[Corollary~5.2]{Davis2024Asymptotic}.\;\;\;\;\;\;\;

\end{remark}

\begin{remark}\label{rem:3}
We note that $\tOmega$ is clearly singular. We investigate in this remark the support subspace of the limiting distribution. Let us write out the expression of $\tOmega$ using the notation in Remark \ref{rem:2}. We have
\begin{equation*}
\tOmega = \begin{pmatrix}
A_1\\
A_2
\end{pmatrix}\text{cov}(\nabla F(\tx;\xi))\begin{pmatrix}
A_1  & A_2^T
\end{pmatrix}
= \begin{pmatrix}
A_1\text{cov}(\nabla F(\tx;\xi)) A_1 & A_1\text{cov}(\nabla F(\tx;\xi)) A_2^T\\
A_2\text{cov}(\nabla F(\tx;\xi)) A_1 & A_2\text{cov}(\nabla F(\tx;\xi))A_2^T
\end{pmatrix}.
\end{equation*}
Note that $\text{rank}(A_1) \hskip-1.4pt = \hskip-1.4pt \text{rank}(\tZ(\tZ^T\nabla_{\bx}^2\mL^\star\tZ)^{-1/2})\hskip-1.4pt =\hskip-1.4pt d-m$, which implies \mbox{$\text{rank}(A_2)\hskip-1.4pt = \hskip-1.4pt m$}~(since the first block-column matrix $[A_1; A_2]$ in \eqref{equ:2} has rank $d$).~In the following~presentation,~we suppose $\text{rank}(\text{cov}(\nabla F(\tx;\xi))) = d$. 

\vskip5pt
\noindent $\bullet$ \textbf{Primal covariance.} Since
\begin{equation*}
\text{rank}(\Omega_{\bx,\bx}^\star) = \text{rank}(A_1\text{cov}(\nabla F(\tx;\xi))A_1^T) = \text{rank}(A_1\text{cov}(\nabla F(\tx;\xi))^{1/2}) = \text{rank}(A_1) = d-m,
\end{equation*}
the support subspace of $\Omega_{\bx,\bx}^\star$ has $d-m$ dimensions.~Specifically, we decompose $\mR^d$ into~$\mR^d = \ker(\tG) \oplus \text{span}(\tG^T)$, where the tangent space $\ker(\tG)$ is a $(d-m)$-dimensional subspace~of~$\mR^d$ and the normal space $\text{span}(\tG^T)$ is an $m$-dimensional subspace of $\mR^d$. Then, by the definition of $A_1$, we know that $\Omega_{\bx,\bx}^\star$ has support in the tangent space $\ker(\tG)$ and vanishes in the~normal space $\text{span}(\tG^T)$.

\vskip5pt
\noindent $\bullet$ \textbf{Dual covariance.} Since
\begin{equation*}
\text{rank}(\Omega_{\blambda,\blambda}^\star) =\text{rank}(A_2\text{cov}(\nabla F(\tx;\xi))A_2^T) = \text{rank}(A_2\text{cov}(\nabla F(\tx;\xi))^{1/2}) = \text{rank}(A_2) = m,
\end{equation*}
we know $\Omega_{\blambda,\blambda}^\star$ is non-degenerate along all directions in $\mR^m$ and has full support over~$\mR^m$.

\vskip5pt
\noindent $\bullet$ \textbf{Joint primal-dual covariance.} Since
$\text{rank}(\tOmega) = d$, we know the support subspace of $\tOmega$ has $d$ dimensions. Let us decompose $\mR^{d+m}$ into $\mR^{d+m} = (\text{ker}(\tG) \otimes \mR^m) \oplus (\text{span}(\tG^T)\otimes \0_m)$. More clearly, $\text{ker}(\tG) \otimes \mR^m$ is a $d$-dimensional subspace of $\mR^{d+m}$ and $\text{span}(\tG^T)\otimes \0_m$~is~an~$m$-dimensional subspace of $\mR^{d+m}$; there bases can be expressed as:
\begin{equation*}
\begin{pmatrix}
\tZ & \0_{d\times m}\\
\0_{m\times (d-m)} & I_m
\end{pmatrix}\in \mR^{(d+m)\times d}\quad\quad \oplus \quad\quad 	\begin{pmatrix}
\tG^T\\
\0_m
\end{pmatrix}\in\mR^{(d+m)\times m}.
\end{equation*}
Then, the joint covariance $\tOmega$ has support in the subspace $\text{ker}(\tG) \otimes \mR^m$ and vanishes~in~the~subspace $\text{span}(\tG^T)\otimes \0_m$.

\end{remark}

}

\subsection{An estimator of the covariance matrix }\label{sec:4.4}

We analyze a plug-in covariance matrix estimator. The sketch-related quantities in~\eqref{equ:Lya:equ},~$C^\star = \mE[\tC^\star]$ and $P^\star = \mE[(I+\tC^\star)\tOmega(I+\tC^\star)^T]$,  can be estimated by computing (note $\tC^\star$ is defined~in \eqref{equ:tC} and $\Omega_t$ is defined~in \eqref{equ:Xi_t})
\begin{gather*}
\tC_t \coloneqq-\prod_{j = 0}^{\tau-1}\rbr{I - K_tS_{t,j}(S_{t,j}^TK_t^2S_{t,j})^\dagger S_{t,j}^TK_t},\\
\hat{C}_t \coloneqq \frac{1}{t}\sum_{i=0}^{t-1}\tC_i,\quad \quad \text{and}\quad\quad \hat{P}_t\coloneqq\frac{1}{t}\sum_{i=0}^{t-1}(I+\tC_i)\Omega_i(I+\tC_i)^T.
\end{gather*}
Since solving Lyapunov equation requires additional effort and $\hat{P}_t$ requires matrix~\mbox{inverse}~even if we do not perform inference at the current step, in what follows, we~are motivated~by Corollary \ref{cor:3} and simply neglect the sketch-related quantities in~\eqref{equ:Lya:equ}~by~estimating \eqref{equ:cov:exact}~instead.~We demonstrate that such negligence only leads to an $O(\rho^\tau)$ error~term,~which~is~generally small even for a moderate $\tau$. Specifically, our estimator of $\Xi^\star$ is defined as:$\hskip1.5cm$
\begin{equation}\label{equ:Xi_t}
\Omega_t = K_t^{-1}\begin{pmatrix}
{sample\_cov}(\{\barg_i\}_{i=0}^{t-1}) & \0\\
\0 & \0
\end{pmatrix}K_t^{-1} \quad \text{ and }\quad \Xi_t = \frac{\Omega_t}{2+\beta/\tbeta},
\end{equation}
where ${sample\_cov}(\{\barg_i\}_{i=0}^{t-1}) = \frac{1}{t}\sum_{i=0}^{t-1}\barg_i\barg_i^T - \rbr{\frac{1}{t}\sum_{i=0}^{t-1}\barg_i}\rbr{\frac{1}{t}\sum_{i=0}^{t-1}\barg_i}^T$ is the sample \mbox{covariance}.

\begin{theorem}\label{thm:5}
Consider \eqref{equ:Xi_t} under the conditions of Theorem \ref{thm:4} with (\ref{equ:BM:c}). For~any~\mbox{$\upsilon>0$},
\begin{equation}\label{equ:Xit:rate}
\|\Xi_t - \Xi^\star \| = O\big(\sqrt{\beta_t \log(1/\beta_t)}\; \big) + o\big(\sqrt{(\log t)^{1+\upsilon}/t }\big) + O(\rho^\tau)\quad \text{a.s.}
\end{equation}	
Furthermore, for any $\bw = (\bw_{\bx}, \bw_{\blambda})\in \mR^{d+m}$ such that $\bw^T\Xi_t\bw \neq 0$,
\begin{multline}\label{equ:BE:prac}
\sup_{z\in\mR}\abr{P\rbr{\frac{\sqrt{1/\baralpha_t}\cdot \bw^T(\bx_t - \bx^\star,\blambda_t-\blambda^\star)}{\sqrt{\bw^T\Xi_t\bw}} \leq z}  - P\rbr{\N(0,1)\leq z}}\\
= o(\beta_t^{1/6}\log(1/\beta_t)) + O(\chi_t^q/\beta_t^{q+0.5}) +O(\rho^\tau),
\end{multline}
where $q\in(0, 1]$ is defined in Theorem \ref{thm:4}.
\end{theorem}

The second term $o\big(\sqrt{(\log t)^{1+\upsilon}/t }\big)$ can be absorbed into the first term $O\big(\sqrt{\beta_t \log(1/\beta_t)}\; \big)$ if $\beta>-1$, e.g., $\beta_t=1/t^\omega$ for $\omega\in(0,1)$. We require condition \eqref{equ:BM:c} for the estimation of the limiting covariance, which ensures the convergence of the sample covariance of $\{\barg_i\}_{i=0}^{t-1}$.~The bounded $4$-th moment of $\barg_t$ is also standard in the literature, as required for analyzing~different covariance estimators for SGD methods \citep{Chen2020Statistical, Zhu2021Online} .

\begin{remark}[Discussion on control sequences $\{\beta_t, \chi_t=\eta_t-\beta_t\}$]\label{rem:1}
	
We note that the~normality in \eqref{nequ:normality} utilizes the adaptive stepsize $\baralpha_t$, which nevertheless is controlled by the sequences $\{\beta_t, \chi_t\}$. We discuss their conditions in this remark. 

The global convergence requires (\ref{cond:step}) (Theorem \ref{thm:1}); the local convergence requires (\ref{cond:n:1}) (Theorem \ref{thm:3}); and the inference additionally requires $\chi<1.5\beta$ (Theorem~\ref{thm:4}). In fact,~by Raabe's test, (\ref{cond:step}) also relates to the quantities $\beta$ and $\chi$:~(\ref{cond:step})~holds~if $-1\leq \beta<-0.5$~and $\chi<-1$. We now specialize $\beta_t$ and $\chi_t$ to be polynomial in $t$, and demonstrate~that~the~conditions can be easily satisfied. 
	
\begin{lemma}\label{lem:10}
Suppose $\beta_t = c_1/t^{c_2}$ and $\chi_t = \beta_t^{c_3}$. Then,
\begin{enumerate}[topsep=2pt,itemsep=0em,partopsep=-2pt,parsep=0.5ex,label=(\alph*):,beginpenalty=10000]
\item (\ref{cond:step}) holds \textbf{if} $c_1>0$, $c_2\in (0.5, 1]$, $c_3>\frac{1}{c_2}$ $\Longrightarrow$ global convergence.
\item (\ref{cond:step}) and (\ref{cond:n:1}) hold \textbf{if} additionally $c_1> \frac{1}{1.5(1-\rho^\tau)}$ when $c_2=1$ $\Longrightarrow$ local convergence.
\item (\ref{cond:step}) and (\ref{cond:n:1}) hold with $\chi<1.5\beta$ \textbf{if} additionally $c_3>1.5\vee\frac{1}{c_2}$ $\Longrightarrow$ asymptotic~inference.
\end{enumerate}
\end{lemma}
The proof of the above lemma is immediate by noting that $\beta = -c_2$, $\chi = -c_2c_3$,~and~$\tbeta = c_1$ if $c_2 = 1$ and $\infty$ if $c_2<1$.~Thus, we omit it.
	
\end{remark}

\section{Numerical Experiments}\label{sec:5}

We provide experimental results in this section. We apply \texttt{AI-StoSQP} to both \mbox{benchmark}~constrained nonlinear optimization problems in CUTEst set \citep{Gould2014CUTEst} and to linearly/ nonlinearly constrained regression problems.~For regression problems, we explore both squared loss and logistic loss.~For conciseness, some of results are deferred to Appendix \ref{sec:exp:more}.$\hskip2cm$

\subsection{Benchmark constrained problems}\label{sec:5.1}

The CUTEst test set collects a number of nonlinear optimization problems with and without constraints.~We implement {eight} equality-constrained problems:~\texttt{MARATOS}, \texttt{ORTHREGB},~\texttt{HS7},~\texttt{HS48}, \texttt{HS78}, {\texttt{BT9}, \texttt{GENHS28}, \texttt{HS39}}. The solution of each problem is solved by IPOPT solver \citep{Waechter2006implementation} {with the initialization specified by the CUTEst package.~Note~that~the~{benchmark} problems may not have unique solutions;~however, we observed that by \mbox{initializing}~at~the same point, our StoSQP method consistently converges to the same~\mbox{solution}~as~IPOPT,~which is also a widely used (deterministic) SQP-based solver.}

For our method, we perform $10^5$ iterations and, at each step, we perform $\tau = 40$ randomized Kaczmarz steps to approximately solve QPs. Given the iterate $\bx_t$, we generate $\barg_t\sim\N(\nabla f_t, \sigma^2(I+\b1\b1^T))$, where $\b1\in\mR^d$ is an all-one vector. We also generate the $(i,j)$ and $(j,i)$ entries of $\barH_t$ from $\N((\nabla^2f_t)_{i,j}, \sigma^2)$. We vary $\sigma^2\in\{10^{-4}, 10^{-2}, 10^{-1}, 1\}$~and~let~$\beta_t =1/t^{0.501}$ (power slightly larger than $0.5$) and $\chi_t = \beta_t^2$. We randomly choose $\baralpha_t\sim\text{Uniform}([\beta_t, \eta_t])$~with $\eta_t=\beta_t+\chi_t$. 
For each problem, we aim to {perform inference for each individual \mbox{variable}~$\{\tx_i\}_{i=1}^d$} by setting the nominal coverage probability to 95\%. {Note from Remark \ref{rem:3} that $\Omega_{\bx,\bx}^\star$ may~vanish along the direction in the normal space $\text{span}(\tG^T)$. Thus, we consider inferring only those $\tx_i$ such that the canonical basis $\be_i \notin \text{span}(\tG^T)$.}
The confidence intervals are constructed~by estimating the limiting covariance using Theorem \ref{thm:5}. The performance of the method~is~measured by the {mean absolute error (MAE) $\|\bx_t-\tx\|$, the average coverage rate (Avg Cov)}~of~the confidence intervals among individuals $\tx_i$, the average~length~(Avg~Len)~of~the~confidence intervals, and {the computational flops per iteration}. We repeat the experiments 200 times~when computing the coverage rate.

\begin{table}[t!]
\centering
\begin{tabular}{|c|c|c|c|c|c|}
\hline &&&&& \\[-2.4ex]
Prob & $\sigma^2$ & MAE ($10^{-2}$) & Ave Cov (\%) & Ave Len ($10^{-2}$)& Flops/iter\\
\hline &&&&& \\[-2.4ex]
\multirow{4}{*}{MARATOS} & $10^{-4}$ & 0.04(0.03) & 96.00 & 0.11(<0.01) & \multirow{4}{*}{137.00} \\
& $10^{-2}$ &0.41(0.31) & 95.00 & 1.10(0.01) & \\
& $10^{-1}$ &1.42(1.07) & 93.50 & 3.48(0.09) & \\
& $1$ & 4.44(3.29) & 95.50 & 10.96(0.74) &\\
\hline &&&&& \\[-2.4ex]
\multirow{4}{*}{ORTHREGB} & $10^{-4}$ & 1.76(1.17) & 88.90 & 27.64(4.80) & \multirow{4}{*}{3867.40} \\
& $10^{-2}$ &4.44(3.51) & 96.42 & 58.02(12.64) & \\
& $10^{-1}$ &9.90(17.99) & 95.84 & 40.80(9.06) & \\
& $1$ & 38.19(44.51) & 96.72 & 24.36(32.92) & \\
\hline &&&&& \\[-2.4ex]
\multirow{4}{*}{HS7} & $10^{-4}$ & 0.02(0.01) & 93.00 & 0.03(<0.01) & \multirow{4}{*}{137.00} \\
& $10^{-2}$ &0.14(0.12) & 92.50 & 0.35(<0.01) & \\
& $10^{-1}$ &0.45(0.34) & 95.00 & 1.10(0.01) & \\
& $1$ & 1.28(0.98) & 97.00 & 3.48(0.09) & \\
\hline &&&&& \\[-2.4ex]
\multirow{4}{*}{HS48} & $10^{-4}$ & 0.03(0.01) & 94.50 & 0.02(0.01) & \multirow{4}{*}{379.00} \\
& $10^{-2}$ &0.25(0.11) & 95.70 & 0.24(0.11) & \\
& $10^{-1}$ &0.85(0.38) & 94.00 & 0.76(0.35) & \\
& $1$ & 2.51(1.25) & 97.50 & 2.41(1.12) & \\
\hline &&&&& \\[-2.4ex]
\multirow{4}{*}{HS78} & $10^{-4}$ & 0.01(<0.01) & 97.80 & 0.02(<0.01) & \multirow{4}{*}{434.01} \\
& $10^{-2}$ &0.15(0.07) & 99.10 & 0.17(0.04) & \\
& $10^{-1}$ &0.51(0.27) & 96.80 & 0.55(0.12) & \\
& $1$ & 1.41(0.69) & 99.10 & 1.74(0.37) & \\
\hline &&&&& \\[-2.4ex]
\multirow{4}{*}{BT9} & $10^{-4}$ & 0.06(0.03) & 93.00 & 0.05(<0.01) & \multirow{4}{*}{308.01} \\
& $10^{-2}$ &0.55(0.29) & 96.25 & 0.55(0.01) & \\
& $10^{-1}$ &1.91(0.92) & 95.50 & 1.76(0.07) & \\
& $1$ & 24.54(1.07) & 94.50 & 6.58(5.57) & \\
\hline &&&&& \\[-2.4ex]
\multirow{4}{*}{GENHS28} & $10^{-4}$ & 0.03(0.01) & 97.75 & 0.04(0.04) & \multirow{4}{*}{1244.07} \\
& $10^{-2}$ &0.26(0.09) & 98.33 & 0.42(0.37) & \\
& $10^{-1}$ &0.76(0.27) & 98.17 & 1.34(1.18) & \\
& $1$ & 2.44(0.94) & 97.83 & 4.22(3.73) & \\
\hline &&&&& \\[-2.4ex]
\multirow{4}{*}{HS39} & $10^{-4}$ & 0.06(0.03) & 93.00 & 0.05(<0.01) & \multirow{4}{*}{308.01} \\
& $10^{-2}$ &0.55(0.29) & 96.25 & 0.55(0.01) & \\
& $10^{-1}$ &1.91(0.92) & 95.50 & 1.76(0.07) & \\
& $1$ & 24.54(1.07) & 94.50 & 6.58(5.57) & \\
\hline
\end{tabular}
\caption{\hskip-2pt\textit{Results of eight benchmark CUTEst problems. \hskip-4ptWe measure performance using {the mean absolute error (MAE)}, the average coverage rate (Ave Cov) and the average~length~(Ave~Len)~of the confidence intervals, and {the computational flops per iteration}. Standard errors~are reported in the bracket for MAE and Ave Len.~We do not report standard errors for Ave Cov~as~they~are meaningless for a 0-1 indicator vector; the standard error is given by $\sqrt{r(1-r)}$ with $r$ being the coverage rate.~{The flops/iter is uniform over different noise level $\sigma^2$ and different~runs~(i.e., the standard error is 0).}
}}\label{tab:1}
\vskip-0.7cm
\end{table}

The results are summarized in Table \ref{tab:1}. From Table \ref{tab:1}, we have the following observations.
\begin{enumerate}[label=\textbf{(\alph*)},topsep=2pt,align=left, leftmargin=0pt, labelindent=0pt, itemindent=!,labelwidth=0pt,]
\setlength\itemsep{-2pt}
\item {In terms of MAE, our method achieves reasonably small MAE values across~all~problems. As the noise level $\sigma^2$ for the objective gradient and Hessian estimates gradually increases, the MAE also increases.~This aligns with our intuition, as noisier \mbox{estimates}~of~the~\mbox{objective}~quantities will require more samples, i.e., longer iterations, to ensure the accuracy of the~\mbox{estimate}~$\bx_t$.\;\;\;}

\item In terms of coverage rate, we observe for the majority of cases that our constructed~confidence intervals cover the true solution with probability of at least $95\%$, and the coverage~rate is robust to the sampling variance $\sigma^2$. 
{There are scenarios, such as \texttt{HS78}~and~\texttt{ORTHREGB}~($\sigma^2=10^{-4}$), where our confidence intervals may exhibit over- or under-coverage. This phenomenon can be attributed to two factors: (i) our limiting covariance estimate has an $O(\rho^\tau)$ bias~due~to sketching techniques, which may either inflate or deflate the estimated variance for each~individual variable, and (ii) we run only a limited number of steps (i.e., the standardized sequence may have not yet reached a stationary stage), given the problem's scale and the~challenges~inherent in the online inference task.} 
Nevertheless, our observation suggests that neglecting~the sketching randomness in the estimation of the limiting covariance~does~not obviously deteriorate the coverage rate.~\mbox{However},~\mbox{using}~(sparse)~\mbox{sketching}~\mbox{vectors} to~solve~QPs~as~in~\eqref{equ:pseduo}~is~computationally more efficient than exact second-order methods. 

\item In terms of confidence intervals' length, we see that the average length gradually increases as $\sigma^2$ increases.~When $\sigma^2$ increases from $10^{-4}$ to $1$, the length increases from~$10^{-4}$~to~$10^{-2}$. This outcome is expected, as the asymptotic covariance $\Xi^\star$ depends on $\cov(\nabla f(\tx; \xi))$~in~\eqref{equ:Lya:equ}.\;\;\;

\item {In terms of computational flops per iteration, it is uniform over different $\sigma^2$ and~independent runs. This quantity basically reflects the problem scale.}
\end{enumerate}

We testify the almost sure convergence rate of our method (Theorem \ref{thm:3}) in Appendix~\ref{sec:exp:more:1}.

\subsection{Constrained regression problems}\label{sec:5.2}

We implement our method on constrained regression problems, considering both linear~regression and logistic regression (see Example \ref{exp:1}). We also allow for either \mbox{linear}~\mbox{constraints}~$A\bx = \bd$ or nonlinear constraints $\|\bx\|^2=b$.~Therefore, there are four cases in total.~{We~\mbox{compare}~our~online inference method with offline constrained $M$-estimation.~For fair comparisons, we~solve~$M$-estimation problems, which are constrained finite-sum problems, using $\ell_1$-penalized SQP~methods with backtracking line search to select proper stepsizes \citep{Nocedal2006Numerical}.~While offline $M$-estimators enjoy asymptotic minimax optimal performance with the least \mbox{covariance} (as introduced in \eqref{equ:M:est}), the proposed online StoSQP method is promising due to its lower~per-iteration computational complexity. (Our method also achieves optimal performance \mbox{under}~appropriate setups; cf. Corollary \ref{cor:3}.)
}

To be specific, for each case (regression model + constraint type), we vary the \mbox{parameter}~dimension $d\hskip-1pt\in\hskip-1pt\{5,20,40,60\}$, and the true solution $\tx$ is  linearly spaced between 0 and~1.~For~each $d$, our method randomly samples a covariate $\xi_{\ba}\sim \mN(\0, 5I+\Sigma_a)$ at each step,~with~three~different choices of $\Sigma_a$ (also considered in \cite{Chen2020Statistical}). (i) Identity: $\Sigma_a = I$.~(ii)~Toeplitz: $[\Sigma_a]_{i,j}=r^{|i-j|}$. (iii) Equi-correlation:~$[\Sigma_a]_{i,j} = r$~for~$i\neq j$~and~$[\Sigma_a]_{i,i} =1$.~For~\mbox{linear}~constraints, we let $A\in\mR^{m\times d}$ with $m=\lceil\sqrt{d}\rceil$ and entries being independently generated~from~standard normal distribution.~For logistic models, we regularize~the~loss~by~a quadratic~penalty~with unit parameter.~{Following Section \ref{sec:5.1}, we perform inference for each individual \mbox{variable}~$\{\tx_i\}_{i=1}^d$ by setting the nominal coverage probability to 95\%.}
We also follow Section \ref{sec:5.1} to implement our method, including the number of iterations, the setup of the \mbox{stepsize}, and the number of sketching steps. 
{In contrast to online methods, the offline $M$-estimation method generates~all $10^5$ samples at once and uses those fixed samples to compute the estimator.~The~covariance~matrix is also estimated by a plug-in estimator.} We \mbox{report}~the~\mbox{results}~only~for~\mbox{$r=0.5$}~for~Toeplitz and $r=0.2$ for~Equi-correlation. The comprehensive comparisons between inexact and exact methods with varying $d, r$ and $\tau$ are reported in Section \ref{sec:exp:more:2}.

\begin{table}[thbp!]
\centering
\scriptsize{
\resizebox{\linewidth}{!}{\begin{tabular}{|c|c|c|c|c|c|c|c|}
\hline
&&&&&&& \\[-2.4ex]
Obj & Cons & $d$ & Design Cov & MAE $(10^{-2})$ & Ave Cov & Ave Len $(10^{-2})$ & Flops/iter\\
\hline &&&&&&& \\[-2.4ex]
\multirow{48}{*}{\vspace*{-1cm} Lin} & \multirow{24}{*}{\vspace*{-0.5cm} Lin} & \multirow{6}{*}{5} & \multirow{2}{*}{Identity} & 0.20(0.08) & 95.40 & 0.19(0.03) & $>1.6\times 10^7$\\
&&&& 3.09(1.19) & 89.00 & 2.43(0.39) & 380.00\\
\cline{4-8} &&&&&&& \\[-2.4ex]
&&& Toeplitz & 0.20(0.08) & 94.60 & 0.19(0.03) & $>1.7\times 10^7$\\
&&& $(r=0.5)$ & 2.84(1.08) & 90.70 & 2.39(0.41) & 380.00\\
\cline{4-8} &&&&&&& \\[-2.4ex]
&&& Equi-corr & 0.21(0.09) & 94.50 & 0.19(0.03) & $>1.6\times 10^7$\\
&&& $(r=0.2)$ & 3.03(1.04) & 90.70 & 2.41(0.40) & 380.00\\
\cline{3-8} &&&&&&& \\[-2.4ex]
&& \multirow{6}{*}{20} & \multirow{2}{*}{Identity} & 0.51(0.09) & 94.98 & 0.23(0.02) & $>7.1\times 10^7$\\
&&&& 6.82(1.18) & 93.67 & 2.87(0.22) & 2340.11\\
\cline{4-8} &&&&&&& \\[-2.4ex]
&&& Toeplitz & 0.52(0.10) & 94.35 & 0.23(0.02) & $>6.5\times 10^7$\\
&&& $(r=0.5)$ & 6.83(1.07) & 93.42 & 2.89(0.22) & 2340.11\\
\cline{4-8} &&&&&&& \\[-2.4ex]
&&& Equi-corr & 0.52(0.09) & 94.85 & 0.23(0.02) & $>6.3\times 10^7$\\
&&& $(r=0.2)$ & 6.70(1.10) & 94.40 & 2.89(0.23) & 2340.11\\
\cline{3-8} &&&&&&& \\[-2.4ex]
&& \multirow{6}{*}{40} & \multirow{2}{*}{Identity} & 0.75(0.09) & 94.99 & 0.23(0.01) & $>2.3\times10^8$\\
&&&& 10.02(1.15) & 94.32 & 3.01(0.14) & 7160.90 \\
\cline{4-8} &&&&&&& \\[-2.4ex]
&&& Toeplitz & 0.75(0.09) & 94.97 & 0.23(0.01) & $>2.0\times10^8$\\
&&& $(r=0.5)$ & 9.84(1.49) & 94.60 & 3.03(0.16) & 7160.90\\
\cline{4-8} &&&&&&& \\[-2.4ex]
&&& Equi-corr & 0.75(0.09) & 95.35 & 0.24(0.01) & $>1.9\times10^8$\\
&&& $(r=0.2)$ & 9.84(1.16) & 94.69 & 3.03(0.15) & 7160.90\\
\cline{3-8} &&&&&&& \\[-2.4ex]
&& \multirow{6}{*}{60} & \multirow{2}{*}{Identity} & 0.95(0.09) & 94.47 & 0.24(0.01) & $>4.4\times10^8$\\
&&&& 12.34(1.18) & 94.91 & 3.12(0.12) & 14382.86 \\
\cline{4-8} &&&&&&& \\[-2.4ex]
&&& Toeplitz & 0.94(0.09) & 94.81 & 0.24(0.01) & $>3.5\times10^8$\\
&&& $(r=0.5)$ & 12.29(1.18) & 95.21 & 3.14(0.12) & 14382.86\\
\cline{4-8} &&&&&&& \\[-2.4ex]
&&& Equi-corr & 0.94(0.10) & 95.07 & 0.24(0.01) & $>2.9\times10^8$\\
&&& $(r=0.2)$ & 12.07(1.22) & 95.66 & 3.14(0.13) & 14382.86\\
\cline{2-8} &&&&&&& \\[-2.4ex]
& \multirow{24}{*}{\vspace*{-0.5cm} Non} & \multirow{6}{*}{5} & \multirow{2}{*}{Identity} & 0.25(0.10) & 94.90 & 0.22(0.03) & $>9.3\times 10^6$\\
&&&& 3.25(1.10) & 93.40 & 2.82(0.37) & 330.00\\
\cline{4-8} &&&&&&& \\[-2.4ex]
&&& Toeplitz & 0.26(0.10) & 93.90 & 0.23(0.03) & $>9.9\times 10^6$\\
&&& $(r=0.5)$ & 3.16(1.01) & 95.00 & 2.88(0.37) & 330.00\\
\cline{4-8} &&&&&&& \\[-2.4ex]
&&& Equi-corr & 0.25(0.08) & 94.80 & 0.23(0.03) & $>9.7\times 10^6$\\
&&& $(r=0.2)$ & 3.21(1.07) & 93.60 & 2.85(0.37) & 330.00\\
\cline{3-8} &&&&&&& \\[-2.4ex]
&& \multirow{6}{*}{20} & \multirow{2}{*}{Identity} & 0.56(0.09) & 94.65 & 0.25(0.01) & $>3.4\times 10^7$\\
&&&& 7.08(1.11) & 94.55 & 3.14(0.09) & 2100.07\\
\cline{4-8} &&&&&&& \\[-2.4ex]
&&& Toeplitz & 0.56(0.10) & 94.83 & 0.25(0.01) & $>3.7\times 10^7$\\
&&& $(r=0.5)$ & 7.15(1.11) & 95.35 & 3.18(0.08) & 2100.07\\
\cline{4-8} &&&&&&& \\[-2.4ex]
&&& Equi-corr & 0.56(0.09) & 95.23 & 0.25(0.01) & $>3.6\times 10^7$\\
&&& $(r=0.2)$ & 7.11(1.14) & 94.77 & 3.18(0.08) & 2100.07\\
\cline{3-8} &&&&&&& \\[-2.4ex]
&& \multirow{6}{*}{40} & \multirow{2}{*}{Identity} & 0.79(0.09) & 95.23 & 0.25(0.01) & $>6.7\times10^7$\\
&&&& 10.32(1.15) & 94.94 & 3.23(0.06) & 6560.62 \\
\cline{4-8} &&&&&&& \\[-2.4ex]
&&& Toeplitz & 0.81(0.09) & 95.20 & 0.25(0.01) & $>8.0\times10^7$\\
&&& $(r=0.5)$ & 10.34(1.20) & 95.45 & 3.27(0.06) & 6560.62\\
\cline{4-8} &&&&&&& \\[-2.4ex]
&&& Equi-corr & 0.81(0.09) & 94.94 & 0.25(0.01) & $>8.0\times10^7$\\
&&& $(r=0.2)$ & 10.53(1.17) & 94.99 & 3.28(0.06) & 6560.62\\
\cline{3-8} &&&&&&& \\[-2.4ex]
&& \multirow{6}{*}{60} & \multirow{2}{*}{Identity} & 0.99(0.09) & 94.89 & 0.25(0.01) & $>1.0\times10^8$\\
&&&& 12.95(1.28) & 95.04 & 3.30(0.06) & 13422.14 \\
\cline{4-8} &&&&&&& \\[-2.4ex]
&&& Toeplitz & 1.00(0.09) & 94.82 & 0.25(0.01) & $>1.1\times10^8$\\
&&& $(r=0.5)$ & 12.80(1.29) & 95.46 & 3.34(0.05) & 13422.14\\
\cline{4-8} &&&&&&& \\[-2.4ex]
&&& Equi-corr & 0.99(0.09) & 95.28 & 0.25(0.01) & $>1.2\times10^8$\\
&&& $(r=0.2)$ & 12.88(1.20) & 95.63 & 3.35(0.05) & 13422.14\\
\hline
\end{tabular}}
}
\vspace{-0.2cm}
\caption{\textit{Comparison results of online StoSQP and offline $M$-estimation for constrained~regression problems (linear models). For each cell, the top row shows the result of the $M$-estimator, while the bottom row shows the result of StoSQP.}}\label{tab:2}
\vspace{-0.2cm}
\end{table}

\begin{table}[thbp!]
\centering
\scriptsize{
\resizebox{\linewidth}{!}{\begin{tabular}{|c|c|c|c|c|c|c|c|}
\hline
&&&&&&& \\[-2.4ex]
Obj & Cons & $d$ & Design Cov & MAE $(10^{-2})$ & Ave Cov & Ave Len $(10^{-2})$ & Flops/iter\\
\hline &&&&&&& \\[-2.4ex]
\multirow{48}{*}{\vspace*{-1cm} Logit} & \multirow{24}{*}{\vspace*{-0.5cm} Lin} & \multirow{6}{*}{5} & \multirow{2}{*}{Identity} & 0.13(0.06) & 95.00 & 0.13(0.03) & $>1.8\times 10^7$\\
&&&& 2.14(0.86) & 86.40 & 1.56(0.42) & 380.00\\
\cline{4-8} &&&&&&& \\[-2.4ex]
&&& Toeplitz & 0.13(0.06) & 95.60 & 0.12(0.03) & $>1.2\times 10^7$\\
&&& $(r=0.5)$ & 1.97(0.77) & 89.20 & 1.53(0.42) & 380.00\\
\cline{4-8} &&&&&&& \\[-2.4ex]
&&& Equi-corr & 0.14(0.06) & 93.40 & 0.12(0.03) & $>1.3\times 10^7$\\
&&& $(r=0.2)$ & 2.03(0.81) & 88.00 & 1.55(0.41) & 380.00\\
\cline{3-8} &&&&&&& \\[-2.4ex]
&& \multirow{6}{*}{20} & \multirow{2}{*}{Identity} & 0.31(0.05) & 95.20 & 0.14(0.01) & $>1.0\times 10^8$\\
&&&& 4.28(0.79) & 92.25 & 1.73(0.15) & 2340.11\\
\cline{4-8} &&&&&&& \\[-2.4ex]
&&& Toeplitz & 0.30(0.05) & 94.60 & 0.13(0.01) & $>9.4\times 10^7$\\
&&& $(r=0.5)$ & 4.01(0.83) & 92.28 & 1.65(0.15) & 2340.11\\
\cline{4-8} &&&&&&& \\[-2.4ex]
&&& Equi-corr & 0.29(0.05) & 95.27 & 0.13(0.01) & $>9.8\times 10^7$\\
&&& $(r=0.2)$ & 3.86(0.73) & 93.20 & 1.61(0.14) & 2340.11\\
\cline{3-8} &&&&&&& \\[-2.4ex]
&& \multirow{6}{*}{40} & \multirow{2}{*}{Identity} & 0.40(0.05) & 95.05 & 0.13(0.01) & $>3.6\times10^8$\\
&&&& 5.13(0.64) & 94.94 & 1.59(0.09) & 7160.90 \\
\cline{4-8} &&&&&&& \\[-2.4ex]
&&& Toeplitz & 0.39(0.05) & 94.65 & 0.12(0.01) & $>3.5\times10^8$\\
&&& $(r=0.5)$ & 4.89(0.68) & 94.84 & 1.51(0.09) & 7160.90\\
\cline{4-8} &&&&&&& \\[-2.4ex]
&&& Equi-corr & 0.35(0.04) & 95.07 & 0.11(0.01) & $>3.2\times10^8$\\
&&& $(r=0.2)$ & 4.28(0.61) & 95.81 & 1.38(0.09) & 7160.90\\
\cline{3-8} &&&&&&& \\[-2.4ex]
&& \multirow{6}{*}{60} & \multirow{2}{*}{Identity} & 0.48(0.05) & 94.74 & 0.12(0.01) & $>6.2\times10^8$\\
&&&& 5.85(0.66) & 95.20 & 1.51(0.08) & 14382.86 \\
\cline{4-8} &&&&&&& \\[-2.4ex]
&&& Toeplitz & 0.45(0.04) & 94.89 & 0.11(0.01) & $>6.2\times10^8$\\
&&& $(r=0.5)$ & 5.40(0.61) & 95.66 & 1.42(0.08) & 14382.86\\
\cline{4-8} &&&&&&& \\[-2.4ex]
&&& Equi-corr & 0.39(0.03) & 95.00 & 0.10(0.01) & $>5.9\times10^8$\\
&&& $(r=0.2)$ & 4.61(0.52) & 96.19 & 1.24(0.07) & 14382.86\\
\cline{2-8} &&&&&&& \\[-2.4ex]

& \multirow{24}{*}{\vspace*{-0.5cm} Non} & \multirow{6}{*}{5} & \multirow{2}{*}{Identity} & 0.17(0.07) & 94.70 & 0.17(0.02) & $>5.2\times 10^6$\\
&&&& 2.74(0.91) & 89.30 & 2.08(0.27) & 330.00\\
\cline{4-8} &&&&&&& \\[-2.4ex]
&&& Toeplitz & 0.18(0.07) & 94.30 & 0.16(0.02) & $>5.0\times 10^6$\\
&&& $(r=0.5)$ & 2.45(0.85) & 92.00 & 2.02(0.28) & 330.00\\
\cline{4-8} &&&&&&& \\[-2.4ex]
&&& Equi-corr & 0.18(0.07) & 94.50 & 0.16(0.02) & $>5.4\times 10^6$\\
&&& $(r=0.2)$ & 2.47(0.93) & 91.10 & 2.05(0.28) & 330.00\\
\cline{3-8} &&&&&&& \\[-2.4ex]
&& \multirow{6}{*}{20} & \multirow{2}{*}{Identity} & 0.34(0.06) & 94.98 & 0.15(0.01) & $>1.8\times 10^7$\\
&&&& 4.64(0.85) & 93.20 & 1.92(0.07) & 2100.07\\
\cline{4-8} &&&&&&& \\[-2.4ex]
&&& Toeplitz & 0.33(0.05) & 94.93 & 0.15(0.01) & $>1.8\times 10^7$\\
&&& $(r=0.5)$ & 4.38(0.73) & 93.52 & 1.83(0.07) & 2100.07\\
\cline{4-8} &&&&&&& \\[-2.4ex]
&&& Equi-corr & 0.32(0.05) & 94.50 & 0.14(0.01) & $>1.7\times 10^7$\\
&&& $(r=0.2)$ & 4.31(0.72) & 92.93 & 1.78(0.07) & 2100.07\\
\cline{3-8} &&&&&&& \\[-2.4ex]
&& \multirow{6}{*}{40} & \multirow{2}{*}{Identity} & 0.45(0.05) & 94.60 & 0.14(0.01) & $>3.6\times10^7$\\
&&&& 5.82(0.71) & 93.62 & 1.72(0.06) & 6560.62 \\
\cline{4-8} &&&&&&& \\[-2.4ex]
&&& Toeplitz & 0.43(0.05) & 94.54 & 0.13(0.01) & $>3.7\times10^7$\\
&&& $(r=0.5)$ & 5.47(0.77) & 93.88 & 1.64(0.06) & 6560.62\\
\cline{4-8} &&&&&&& \\[-2.4ex]
&&& Equi-corr & 0.39(0.04) & 94.89 & 0.12(0.01) & $>3.4\times10^7$\\
&&& $(r=0.2)$ & 4.92(0.73) & 94.35 & 1.50(0.06) & 6560.62\\
\cline{3-8} &&&&&&& \\[-2.4ex]
&& \multirow{6}{*}{60} & \multirow{2}{*}{Identity} & 0.51(0.05) & 95.09 & 0.13(0.01) & $>5.6\times10^7$\\
&&&& 6.48(0.71) & 94.39 & 1.60(0.06) & 13422.14 \\
\cline{4-8} &&&&&&& \\[-2.4ex]
&&& Toeplitz & 0.47(0.04) & 95.36 & 0.12(0.01) & $>5.8\times10^7$\\
&&& $(r=0.5)$ & 6.16(0.65) & 93.99 & 1.51(0.05) & 13422.14\\
\cline{4-8} &&&&&&& \\[-2.4ex]
&&& Equi-corr & 0.42(0.04) & 95.13 & 0.11(0.01) & $>5.6\times10^7$\\
&&& $(r=0.2)$ & 5.30(0.66) & 94.34 & 1.32(0.06) & 13422.14\\
\hline
\end{tabular}}
}
\vspace{-0.2cm}
\caption{\textit{Comparison results of online StoSQP and offline $M$-estimation for constrained~regression problems (logistic models). For each cell, the top row shows the result of the $M$-estimator, while the bottom row shows the result of StoSQP.}}\label{tab:3}
\vspace{-0.2cm}
\end{table}

We summarize the comparison results in Tables \ref{tab:2} and \ref{tab:3}, {including  the mean absolute error (MAE) $\|\bx_t-\tx\|$, the average coverage rate (Avg Cov)~of~the confidence intervals~among individuals $\tx_i$, the average~length~(Avg~Len)~of~the~confidence intervals, and the computational flops per iteration}. From Tables \ref{tab:2} and \ref{tab:3}, we summarize the following observations.
{ \begin{enumerate}[label=\textbf{(\alph*)},topsep=2pt,align=left, leftmargin=0pt, labelindent=0pt, itemindent=!,labelwidth=0pt,]
\setlength\itemsep{-2pt}
\item In terms of MAE and Ave Len, offline constrained $M$-estimators achieve results that are \textit{an order of magnitude smaller} than those of online StoSQP. This gap aligns with~our~analysis~from two perspectives. First, $M$-estimators exhibit $\sqrt{n}$-consistency and optimal \mbox{asymptotic}~covariance, while StoSQP estimators exhibit only $\sqrt{1/\baralpha_n}$-consistency, where $\baralpha_n$ denotes the~stepsize (in this case, $\sqrt{n^{0.501}}$). Clearly, as long as $\baralpha_n \not\asymp 1/n$, the covariance of StoSQP estimators~is not comparable to that of $M$-estimators when the estimators are scaled by the same scalar. Second, the randomization of the sketching solver within StoSQP introduces additional uncertainty to the estimators, further enlarging the asymptotic covariance~(cf.~Corollary~\ref{cor:3}),~although this enlargement is controlled by the precision of the sketching solver and decreases exponentially with the number of sketching steps.

\item In terms of Ave Cov, the proposed online StoSQP method achieves promising coverage rates that are very close to 95\% for both linear and logistic models as well as for both~linear~and nonlinear constraints, matching the performance of offline $M$-estimators.~The~only~\mbox{potential}~exception occurs when $d=5$, where StoSQP may exhibit undercoverage (around 90\%) for linear and logistic models.~Upon closer examination of these scenarios, we find that~the~\mbox{condition}~number of the Lagrangian Hessian of these problems exceeds $d^3 = 125$, indicating that these~problems are ill-conditioned and difficult to solve, particularly when using a sketching solver that seems overkill~(see \eqref{nnequ:1}). Furthermore, SGD-based estimators are also observed to~\mbox{exhibit}~undercoverage for various problems due to significant challenges of online inference tasks \citep{Zhu2021Online}.~To~our knowledge, StoSQP is the first method capable of conducting online~statistical inference for constrained model parameters. While projection-based estimators~\citep{Duchi2021Asymptotic, Davis2024Asymptotic} may demonstrate similar asymptotic normality, estimating their limiting covariance remains unclear.

\item In terms of computational flops per iteration, offline $M$-estimation involves processing the full batch of samples, resulting in significant computational and memory costs. In~\mbox{contrast},~our online method processes a single sample, requiring significantly fewer computations. Additionally, the inexact sketching solver for solving Newton systems further reduces the dominant computational cost of the 
proposed second-order method.~\mbox{Overall},~the~\mbox{reduced}~\mbox{computational}~flops per iteration are a major advantage of our online StoSQP method over offline methods.$\hskip1cm$
\end{enumerate}

Further discussions of exact and inexact StoSQP methods are provided in Section \ref{sec:exp:more:2}.
}

\section{Conclusion and Future Work}\label{sec:6}

We performed statistical inference of nonlinearly constrained stochastic optimization problems using a fully online second-order method called Stochastic Sequential Quadratic Programming (StoSQP). In each iteration, the scheme selects a proper adaptive stepsize and inexactly solves the Newton system (a quadratic program) by a randomized~\mbox{sketching}~solver.~\mbox{Consequently},~the considered method is more adaptive and computationally efficient than existing exact second-order methods. 
For this method, we established an almost sure convergence rate~and~iteration complexity, and proved the asymptotic normality property for the last iterate.~We~\mbox{observed}~that although the limiting covariance is worse than the minimax optimum achieved by constrained $M$-estimators and online projection-based estimators, the gap decays exponentially fast in terms of the number of iterations employed for the sketching solver (e.g.,~the~\mbox{covariance}~matches the optimum if using exact QP solvers).~Additionally, we analyzed a plug-in covariance~\mbox{matrix} estimator. Our analysis precisely quantified the uncertainty of the stochastic process \mbox{generated} by StoSQP methods, which encompasses the randomness of sampling as well as the computation (sketching and stepsize). The randomness of computation is particularly important for second-order methods to be efficient in practice. With our results, one can apply \texttt{AI-StoSQP}~to perform online inference for constrained~estimation problems.

As for future directions, it is of interest to provide a non-asymptotic analysis for~StoSQP methods. Such a result would complement our analysis by bounding the distance between~the distribution of $(\bx_t, \blambda_t)$ and the normal distribution for any given $t$. 
Furthermore, recent~literature on online inference has explored different test statistics, whose asymptotic distributions rely on the application of the Functional Central Limit Theorem \citep{Lee2022Fast, Luo2022Covariance, Li2023Online, Roy2023Online, Chen2024Online}. \mbox{Establishing}~a~functional CLT for second-order methods and studying the limiting distribution of random scaling estimators is also an interesting research direction. 
Finally, incorporating nonlinear inequality constraints into the problems and developing second-order methods without projections also deserves further study in future work.

\acks{SN would like to acknowledge Xinchen Du, Yuefeng Han, and Wanrong Zhu for the helpful discussions of the work.
MWM would like to acknowledge the NSF, ONR, and a J. P. Morgan Chase Faculty Research Award for providing partial support of this work. This work was also supported by the U.S. Department of Energy, Office of Science, Office of Advanced Scientific Computing Research, Scientific Discovery through Advanced Computing (SciDAC) program, under Contract Number DE-AC02-05CH11231 at Lawrence Berkeley National Laboratory.}

\bibliography{ref}

\newpage
\appendix

\input{appendix}

\end{document}

%% file: appendix.tex
\section{An Example of Stepsize Selection Scheme}\label{app:1}

We consider selecting a stepsize to decrease the penalized objective 
\begin{equation*}
\phi_{\nu}(\bx;\xi) \coloneqq \nu F(\bx;\xi) + \|c(\bx)\|.
\end{equation*}
We note that decreasing the objective $F(\bx;\xi)$ only is not reasonable for constrained problems, since we may violate constraints arbitrarily. The local linear approximation of $\phi_{\nu}(\bx;\xi)$ at $(\bx_t;\xi_t)$ along the direction $\barDelta\bx_t$ is 
\begin{equation*}
\phi^{\text{loc}}_{\nu}(\bx_t;\xi_t, \barDelta\bx_t) \coloneqq\nu(F(\bx_t;\xi_t)+\barg_t^T\barDelta\bx_t) +\|c_t+G_t\barDelta\bx_t\|.
\end{equation*}
We can further define the local model reduction, a negative quantity for sufficiently small~$\nu>0$ and approximation error, as
\begin{align}\label{equ:reduction}
\Delta\phi_{\nu}^{\text{loc}}(\bx_t;\xi_t,\barDelta\bx_t) & \coloneqq \phi^{\text{loc}}_{\nu}(\bx_t;\xi_t, \barDelta\bx_t) - \phi^{\text{loc}}_{\nu}(\bx_t;\xi_t, \0)  \nonumber\\
& = \nu \barg_t^T\barDelta\bx_t + \|c_t+G_t\barDelta\bx_t\| - \|c_t\|.
\end{align}
For a given scalar $\kappa_t\in(0, 1)$, we select $\baralpha_t$ such that $\phi_{\nu}(\bx_t+\baralpha_t\barDelta\bx_t;\xi_t)$ decreases $\phi_{\nu}(\bx_t;\xi_t)$ by at least a factor of $\kappa_t\baralpha_t$ of the local model reduction $\Delta\phi_{\nu}^{\text{loc}}$ (so called the Armijo condition). Specifically, we require 
\begin{equation}\label{equ:armijo}
\phi_{\nu}(\bx_t+\baralpha_t\barDelta\bx_t;\xi_t) \leq \phi_{\nu}(\bx_t;\xi_t) + \kappa_t\baralpha_t\cdot \Delta\phi_{\nu}^{\text{loc}}(\bx_t;\xi_t,\barDelta\bx_t).
\end{equation}
To satisfy \eqref{equ:armijo}, we suppose $\nabla F(\bx;\xi_t)$ and $G(\bx)$ are local Lipschitz continuous around~$\bx_t$,~and suppose $\baralpha_t\leq 1$. Then, there exists a constant $\Upsilon_{\nu,t}>0$, such that
\begin{align*}
& \phi_{\nu}(\bx_t  +\baralpha_t\barDelta\bx_t;\xi_t) = \nu F(\bx_t+\baralpha_t\barDelta\bx;\xi_t) + \|c(\bx_t+\baralpha_t\barDelta\bx)\| \\ 
& \leq \phi_{\nu}(\bx_t;\xi_t) + \nu\baralpha_t\barg_t^T\barDelta\bx_t + \|c_t+\baralpha_tG_t\barDelta\bx_t\| - \|c_t\| + \Upsilon_{\nu,t}\baralpha_t^2\|\barDelta\bx_t\|^2 \;\; (\text{since } \nabla F, G \text{ are Lip})\\
& \leq \phi_{\nu}(\bx_t;\xi_t) + \nu\baralpha_t\barg_t^T\barDelta\bx_t  + \baralpha_t \|c_t+G_t\barDelta\bx_t\| - \baralpha_t\|c_t\| + \Upsilon_{\nu,t}\baralpha_t^2\|\barDelta\bx_t\|^2 \;\; (\text{since } \baralpha_t\leq 1)\\
& \stackrel{\mathclap{\eqref{equ:reduction}}}{=}\;  \phi_{\nu}(\bx_t;\xi_t) +\baralpha_t\cdot \Delta\phi_{\nu}^{\text{loc}}(\bx_t;\xi_t,\barDelta\bx_t) + \Upsilon_{\nu,t}\baralpha_t^2\|\barDelta\bx_t\|^2.
\end{align*}
Therefore, \eqref{equ:armijo} is satisfied as long as
\begin{equation}\label{equ:step}
\baralpha_t \leq \frac{(\kappa_t-1)\cdot \Delta\phi_{\nu}^{\text{loc}}(\bx_t;\xi_t,\barDelta\bx_t)}{\Upsilon_{\nu,t}\|\barDelta\bx_t\|^2} \wedge 1 \eqqcolon \baralpha_{t, thres}.
\end{equation}
The Lipschitz constant $\Upsilon_{\nu,t}$ can be estimated around $\bx_t$ \citep{Curtis2018Exploiting}~or simply prespecified as a large constant. The condition \eqref{equ:step} leads us to propose \mbox{$\baralpha_t \coloneqq \text{Proj}_{[\beta_t, \eta_t]}(\baralpha_{t, thres})$}. See \cite{Berahas2021Sequential, Berahas2023Stochastic, Curtis2021Inexact} for detailed random projections and~\cite{Hong2023Constrained} for adaptive selection of parameters of line search functions (e.g. $\nu$).

\section{Preparation Lemmas}\label{appen:A}

\begin{lemma}\label{aux:lem:5}
Suppose $\{\varphi_i\}_i$ is a positive sequence that satisfies $\lim\limits_{i\rightarrow\infty}i(1 - \varphi_{i-1}/\varphi_i) = \varphi$. Then, for any $p \geq 0$, we have $\lim\limits_{i\rightarrow\infty}i\rbr{1 - \varphi_{i-1}^p/\varphi_i^p} = p\cdot \varphi$.
\end{lemma}

\begin{lemma}\label{aux:lem:1}
Let $\{\varphi_i\}_i$ be a positive sequence. If $\lim\limits_{i\rightarrow\infty}i(1 - \varphi_{i-1}/\varphi_i) = \varphi<0$, then~\mbox{$\lim\limits_{i\rightarrow\infty}\varphi_i = 0$}.
\end{lemma}

\begin{lemma}\label{aux:lem:2}
Let $\{\phi_i\}_i$, $\{\varphi_i\}_i$, $\{\sigma_i\}_i$ be three positive sequences. Suppose
\begin{equation}\label{pequ:A1}
\lim\limits_{i\rightarrow \infty} i\rbr{1 - \phi_{i-1}/\phi_i} = \phi<0,\quad \quad\lim\limits_{i\rightarrow\infty}\varphi_i = 0, \quad\quad \lim\limits_{i\rightarrow \infty} i\varphi_i = \tvarphi
\end{equation}
for a constant $\phi$ and a (possibly infinite) constant $\tvarphi \in(0, \infty]$. For any $l\geq 1$, if we further~have $\sum_{k=1}^{l}\sigma_k + p\phi/\tvarphi>0$ for some constant $p\in(0,1]$, then the following results hold as $t\rightarrow \infty$
\begin{enumerate}[topsep=1pt,itemsep=0em,partopsep=-2pt,parsep=1ex,label=(\alph*):,beginpenalty=10000]
\item When $p=1$,
\begin{equation}\label{pequ:A2}
\begin{aligned}
& \frac{1}{\phi_t}\sum_{i=0}^{t}\prod_{j = i+1}^t\prod_{k=1}^{l}(1-\varphi_j\sigma_k)\varphi_i\phi_i \longrightarrow \frac{1}{\sum_{k=1}^{l}\sigma_k + \phi/\tvarphi},\\
& \frac{1}{\phi_t}\cbr{\sum_{i=0}^{t}\prod_{j = i+1}^t\prod_{k=1}^{l}(1-\varphi_j\sigma_k)\varphi_i\phi_i a_i + b\cdot \prod_{j = 0}^t\prod_{k=1}^{l}(1-\varphi_j\sigma_k) }  \longrightarrow 0,
\end{aligned}
\end{equation}
where the second result holds for any constant $b$ and sequence $\{a_t\}_t$ such that $a_t\rightarrow 0$.

\item When $p\in (0,1)$,
\begin{equation}\label{pequ:A2:new}
\begin{aligned}
& \frac{1}{\phi_t^p}\sum_{i=0}^{t}\prod_{j = i+1}^t\prod_{k=1}^{l}(1-\varphi_j\sigma_k)\varphi_i\phi_i \longrightarrow 0,\\
& \frac{1}{\phi_t^p}\cbr{\sum_{i=0}^{t}\prod_{j = i+1}^t\prod_{k=1}^{l}(1-\varphi_j\sigma_k)\varphi_i\phi_i^p a_i + b\cdot \prod_{j = 0}^t\prod_{k=1}^{l}(1-\varphi_j\sigma_k) }  \longrightarrow 0,
\end{aligned}
\end{equation}
where the second result holds for any constant $b$ and sequence $\{a_t\}_t$ such that $a_t\rightarrow 0$.
\end{enumerate}

\end{lemma}

\begin{lemma}\label{aux:lem:3}
For any scalars $a, b$, we have $P\rbr{a<\N(0,1)\leq b} \leq b-a$. Furthermore, if $0<a\leq b$, then $P\rbr{a<\N(0,1)\leq b} \leq b/a-1$.
\end{lemma}

\begin{lemma}\label{aux:lem:4}
Let $A_t$, $B_t$, $C_t$ be three variables depending on the index $t$; also let $\Phi(z) = P(\N(0,1)\leq z)$ be the cumulative distribution function of standard Gaussian variable.~Suppose for the index $t$,
\begin{equation}\label{aequ:1}
\sup_{z\in\mR}\abr{P\rbr{A_t \leq z} - \Phi(z)} \leq a_t, \quad |B_t|\leq b_t,\quad |C_t|\leq c_t\quad \text{almost surely}
\end{equation}
where $a_t, b_t\geq 0$ and $0\leq c_t< 1$. Then, we have
\begin{equation*}
\sup_{z\in\mR}\abr{P\rbr{\frac{A_t+B_t}{\sqrt{1+C_t}}\leq z } - \Phi(z) } \leq a_t + b_t + \frac{c_t}{\sqrt{1-c_t}}.
\end{equation*}	
\end{lemma}

\section{Proofs of Preparation Lemmas}

\subsection{Proof of Lemma \ref{aux:lem:5}}

By the condition, we know $\varphi_{i-1}/\varphi_i = 1 - \varphi/i + o\rbr{1/i}$. Thus, we have
\begin{equation*}
i\rbr{1 - \varphi_{i-1}^p/\varphi_i^p} = i\rbr{1 - \cbr{1 - \varphi/i + o\rbr{1/i}}^p } = p\varphi + o(1).
\end{equation*}
This completes the proof.

\subsection{Proof of Lemma \ref{aux:lem:1}}

By Lemma \ref{aux:lem:5}, we know for any positive constant $p$, $\lim\limits_{i\rightarrow \infty} i \rbr{1-\varphi_{i-1}^p/\varphi_i^p} = p \varphi$. Choosing $p$ large enough such that $p\varphi <-1$,	the Raabe's test indicates that $\sum_{i=0}^{\infty}\varphi_i^p<\infty$. This~implies $\varphi_i \rightarrow 0$ and we complete the proof.

\subsection{Proof of Lemma \ref{aux:lem:2}}

For any scalar $A$, we have
\begin{align*}
& \frac{1}{\phi_t^p}\sum_{i=0}^{t}\prod_{j = i+1}^t\prod_{k=1}^{l}(1-\varphi_j\sigma_k)\varphi_i\phi_i - A \\
& = \frac{1}{\phi_t^p}\prod_{j = 0}^t\prod_{k=1}^{l}(1-\varphi_j\sigma_k)\cbr{\sum_{i=0}^{t}\prod_{j = 0}^i\prod_{k=1}^{l}(1-\varphi_j\sigma_k)^{-1}\varphi_i\phi_i - A\phi_t^p\prod_{j = 0}^t\prod_{k=1}^{l}(1-\varphi_j\sigma_k)^{-1} }.
\end{align*}
For the last term, we have
\begin{align*}
& A\phi_t^p\prod_{j = 0}^t\prod_{k=1}^{l}(1-\varphi_j\sigma_k)^{-1}\\
& = \sum_{i=1}^{t} \rbr{A\phi_i^p \prod_{j = 0}^i\prod_{k=1}^{l}(1-\varphi_j\sigma_k)^{-1} - A\phi_{i-1}^p\prod_{j = 0}^{i-1}\prod_{k=1}^{l}(1-\varphi_j\sigma_k)^{-1} } + A\phi_0^p\prod_{k=1}^{l}(1-\varphi_0\sigma_k)^{-1}\\
& = \sum_{i=1}^{t}A\phi_i^p \prod_{j = 0}^i\prod_{k=1}^{l}(1-\varphi_j\sigma_k)^{-1}\cbr{1 - \frac{\phi_{i-1}^p}{\phi_i^p}\prod_{k=1}^{l}(1-\varphi_i\sigma_k)}  + A\phi_0^p\prod_{k=1}^{l}(1-\varphi_0\sigma_k)^{-1}. 
\end{align*}
Combining the above two displays, we obtain
\begin{align}\label{pequ:A3}
& \frac{1}{\phi_t^p}\sum_{i=0}^{t}\prod_{j = i+1}^t\prod_{k=1}^{l}(1-\varphi_j\sigma_k)\varphi_i\phi_i - A \nonumber\\
& = \frac{1}{\phi_t^p}\prod_{j = 0}^t\prod_{k=1}^{l}(1-\varphi_j\sigma_k)\bigg\{\sum_{i=1}^{t}\prod_{j = 0}^i\prod_{k=1}^{l}(1-\varphi_j\sigma_k)^{-1}\phi_i^p\cbr{\varphi_i\phi_i^{1-p} - A\rbr{1 - \frac{\phi_{i-1}^p}{\phi_i^p}\prod_{k=1}^{l}(1-\varphi_i\sigma_k)}  } \nonumber \\
& \quad\quad +  \phi_0^p\prod_{k=1}^{l}(1-\varphi_0\sigma_k)^{-1}\rbr{\varphi_0\phi_0^{1-p} - A}\bigg\} .
\end{align}
We aim to select $A$ such that the middle term in \eqref{pequ:A3} is small. By \eqref{pequ:A1}, we know
\begin{equation*}
\frac{\phi_{i-1}^p}{\phi_i^p}
 = 1 - \frac{p\phi}{i} + o\rbr{\frac{1}{i}} = 1 - \frac{p\phi}{\tvarphi}\cdot\varphi_i  + o(\varphi_i),
\end{equation*}
where the second equality is due to $1/(i\varphi_i) = 1/\tvarphi + o(1)$ (which is true even if $\tvarphi =~\infty$).~Furthermore, we know
\begin{equation*}
\prod_{k=1}^{l}(1-\varphi_i\sigma_k) = 1 - \varphi_i\sum_{k=1}^{l}\sigma_k + o(\varphi_i).
\end{equation*}
With these two facts, we have
{\small \begin{align}\label{pequ:A4}
& \varphi_i\phi_i^{1-p} - A\big\{1 - \frac{\phi_{i-1}^p}{\phi_i^p}\prod_{k=1}^{l}(1-\varphi_i\sigma_k)\big\} = \varphi_i\phi_i^{1-p} - A\big\{ 1 - (1 - \frac{p\phi}{\tvarphi}\cdot\varphi_i + o(\varphi_i)) (1 - \varphi_i\sum_{k=1}^{l}\sigma_k + o(\varphi_i)) \big\} \nonumber\\
& = \varphi_i\phi_i^{1-p} - A\rbr{\frac{p\phi}{\tvarphi} + \sum_{k=1}^l\sigma_k} \varphi_i +o(\varphi_i).
\end{align}}
\hskip-3.5pt Thus, we let $A = 1/(\sum_{k=1}^l\sigma_k + \phi/\tvarphi)$ if $p =1$ and $A=0$ if $p\in (0, 1)$. Noting that $\phi_i^{1-p}\rightarrow 0$, \eqref{pequ:A3} leads to
\begin{align*}
\frac{1}{\phi_t^p}\sum_{i=0}^{t}\prod_{j = i+1}^t\prod_{k=1}^{l}&(1-\varphi_j\sigma_k)\varphi_i\phi_i  - A \\
& = \frac{1}{\phi_t^p}\prod_{j = 0}^t\prod_{k=1}^{l}(1-\varphi_j\sigma_k)\bigg\{\sum_{i=1}^{t}\prod_{j = 0}^i\prod_{k=1}^{l}(1-\varphi_j\sigma_k)^{-1}\phi_i^p\cdot o(\varphi_i) \\
&\quad\quad +  \phi_0^p\prod_{k=1}^{l}(1-\varphi_0\sigma_k)^{-1}\rbr{\varphi_0\phi_0^{1-p} - A}\bigg\}.
\end{align*}
Comparing the above display with \eqref{pequ:A2} and \eqref{pequ:A2:new}, we note that the first results in \eqref{pequ:A2}~and \eqref{pequ:A2:new} are implied by the second results. Thus, it suffices to prove the second results. We~define
\begin{equation}\label{pequ:A5}
\Psi_t = \frac{1}{\phi_t^p}\cbr{\sum_{i=0}^{t}\prod_{j = i+1}^t\prod_{k=1}^{l}(1-\varphi_j\sigma_k)\varphi_i\phi_i^p a_i + b\cdot \prod_{j = 0}^t\prod_{k=1}^{l}(1-\varphi_j\sigma_k) },
\end{equation}
then
\begin{align*}
\Psi_t & = \frac{1}{\phi_t^p}\cbr{ \varphi_t\phi_t^p a_t +\prod_{k=1}^{l}(1-\varphi_t\sigma_k)\rbr{\sum_{i=0}^{t-1}\prod_{j = i+1}^{t-1}\prod_{k=1}^{l}(1-\varphi_j\sigma_k)\varphi_i\phi_i^p a_i + b\cdot \prod_{j = 0}^{t-1}\prod_{k=1}^{l}(1-\varphi_j\sigma_k)} }\\
&\stackrel{\mathclap{\eqref{pequ:A5}}}{=} \;\; \frac{\phi_{t-1}^p}{\phi_t^p}\prod_{k=1}^{l}(1-\varphi_t\sigma_k)\Psi_{t-1} + \varphi_t a_t.
\end{align*}
By \eqref{pequ:A4}, we know that
\begin{equation*}
\frac{\phi_{t-1}^p}{\phi_t^p}\prod_{k=1}^{l}(1-\varphi_t\sigma_k) = 1 - \rbr{\frac{p\phi}{\tvarphi} + \sum_{k=1}^l\sigma_k}\cdot \varphi_t + o(\varphi_t).
\end{equation*}
Since $\sum_{k=1}^l\sigma_k + p\phi/\tvarphi>0$, we immediately conclude that for a constant $c>0$ and for~all~large enough $t$, $|\Psi_t| \leq (1 - c\varphi_t) |\Psi_{t-1}| + \varphi_t |a_t|$. Let $t_1$ be a fixed integer. We apply this inequality recursively and have for any $t\geq t_1+1$,
\begin{equation*}
|\Psi_t| \leq \prod_{i = t_1+1}^{t}(1 - c\varphi_i)|\Psi_{t_1}| + \sum_{i=t_1+1}^{t}\prod_{j = i+1}^t(1 - c\varphi_j)\varphi_i|a_i|.
\end{equation*}
For any $\epsilon>0$, since $a_i\rightarrow 0$, we select $t_1$ such that $|a_i|\leq \epsilon$, for all $i\geq t_1$. Then, the above inequality leads to
\begin{align*}
|\Psi_t| & \leq \prod_{i = t_1+1}^{t}(1 - c\varphi_i)|\Psi_{t_1}| + \epsilon \sum_{i=t_1+1}^{t}\prod_{j = i+1}^t(1 - c\varphi_j)\varphi_i \\
& = \prod_{i = t_1+1}^{t}(1 - c\varphi_i)|\Psi_{t_1}|  + \frac{\epsilon}{c} \big\{1 - \prod_{j=t_1+1}^{t}(1 - c\varphi_j)\big\}  \leq |\Psi_{t_1}|\exp\rbr{-c\sum_{i=t_1+1}^{t}\varphi_i} + \frac{\epsilon}{c}.
\end{align*}
Since $n\varphi_i \rightarrow \tvarphi\in(0, \infty]$, we know $\sum_{t}\varphi_t \rightarrow \infty$. Thus, for the above $\epsilon>0$, there exists $t_2\geq t_1$ such that $|\Psi_{t_1}|\exp\rbr{-c\sum_{i=t_1+1}^{t}\varphi_i}\leq \epsilon/c$, $\forall t\geq t_2$, which implies $|\Psi_t|\leq 2\epsilon/c$.~This~means $|\Psi_t|\rightarrow0$ and we complete the proof.

\subsection{Proof of Lemma \ref{aux:lem:3}}

The first part of statement holds naturally due to the fact that the density of the standard Gaussian satisfies $\exp(-t^2/2)/\sqrt{2\pi}\leq 1$ for any $t\in\mR$. Moreover, for $0<a\leq b$, we have
\begin{multline*}
P\rbr{a<\N(0,1)\leq b} = \int_{a}^{b} \frac{1}{\sqrt{2\pi}}\exp(-t^2/2)\; dt \leq \frac{b-a}{\sqrt{2\pi}}\exp(-a^2/2) \\
= \rbr{\frac{b}{a}-1}\frac{a}{\sqrt{2\pi}}\exp(-a^2/2) \leq \frac{b}{a}-1,
\end{multline*}
where the last inequality uses $a\exp(-a^2/2) \leq 1$ for all $a$. This completes the proof.

\subsection{Proof of Lemma \ref{aux:lem:4}}

We only prove the result for $z> 0$. The result of $z\leq0$ can be shown in the same~way.~We~know from \eqref{aequ:1} that $\frac{A_t-b_t}{\sqrt{1+c_t}} \leq \frac{A_t + B_t}{\sqrt{1+C_t}} \leq \frac{A_t+b_t}{\sqrt{1-c_t}}$, almost surely. Therefore, we have
{\small \begin{align*}
& P\rbr{\frac{A_t + B_t}{\sqrt{1+C_t}} \leq z}\geq P\rbr{\frac{A_t + b_t}{\sqrt{1-c_t}} \leq z} = P(A_t\leq z(1-c_t)^{1/2}-b_t) \stackrel{\eqref{aequ:1}}{\geq} \Phi(z(1-c_t)^{1/2}-b_t) - a_t \\
& \stackrel{\mathclap{(z\geq 0)}}{=} \;\; \Phi(z) - P\rbr{z(1-c_t)^{1/2}-b_t<\N(0,1)\leq z(1-c_t)^{1/2}} - P\rbr{z(1-c_t)^{1/2}<\N(0,1)\leq z} -a_t \\
& \geq \Phi(z) - b_t - \rbr{\frac{1}{\sqrt{1-c_t}} -1} - a_t\quad (\text{by Lemma \ref{aux:lem:3}})\\
& \geq \Phi(z) - b_t - \frac{c_t}{\sqrt{1-c_t}} -a_t.
\end{align*}}
\hskip-4pt On the other hand, we have
{\small \begin{align*}
& P\rbr{\frac{A_t + B_t}{\sqrt{1+C_t}} \leq z}\leq P\rbr{\frac{A_t - b_t}{\sqrt{1+c_t}} \leq z} = P(A_t\leq z(1+c_t)^{1/2}+b_t)\stackrel{\eqref{aequ:1}}{\leq}\Phi(z(1+c_t)^{1/2}+b_t) + a_t\\
& = \Phi(z) + P\rbr{z<\N(0,1)\leq z(1+c_t)^{1/2}} + P\rbr{z(1+c_t)^{1/2}< \N(0,1)\leq z(1+c_t)^{1/2}+b_t} + a_t\\
& \leq \Phi(z) + \cbr{(1+c_t)^{1/2}-1} + b_t+a_t \quad (\text{by Lemma \ref{aux:lem:3}})\\
& \leq \Phi(z) + c_t + b_t + a_t.
\end{align*}}
\hskip-4pt Combining the above two displays completes the proof.

\section{Proofs of Section \ref{sec:3}}

\subsection{Proof of Lemma \ref{lem:1}}\label{pf:lem:1}

We note that $\gamma_{S} \leq \|\mE[K_tS(S^TK_t^2S)^\dagger S^TK_t \mid \bx_t, \blambda_t]\|\leq \mE[\|K_tS(S^TK_t^2S)^\dagger S^TK_t\| \mid \bx_t, \blambda_t] \leq~1$, where the second inequality is by Jensen's inequality; the third inequality is by the fact that $K_tS(S^TK_t^2S)^\dagger S^TK_t$ is a projection matrix. This shows (a). Let us define for $j=1,\ldots,\tau$, $C_{t,j} = I - K_tS_{t,j}(S_{t,j}^TK_t^2S_{t,j})^\dagger S_{t,j}^TK_t$. Then, we obtain from \eqref{equ:pseduo} that
\begin{equation}\label{equ:z:recur}
\bz_{t,\tau} - \tbz_t = C_{t,\tau-1}(\bz_{t,\tau-1}-\tbz_t) =  \rbr{\prod_{j = 0}^{\tau-1}C_{t,j}}(\bz_{t,0}-\tbz_t) = -\rbr{\prod_{j = 0}^{\tau-1}C_{t,j}}\tbz_t.
\end{equation}
Thus, (b) follows from \eqref{equ:z:recur} and the independence among $\{S_{t,j}\}_j$. Moreover, (c) is proved~by \cite[Theorem 4.6]{Gower2015Randomized}.

\subsection{Proof of Lemma \ref{lem:2}}\label{pf:lem:2}

By Assumption \ref{ass:1}, there exists a constant $\Upsilon_u\geq 1$ such that
\begin{equation}\label{equ:upper:bound}
\|\nabla^2\mL(\bx, \blambda)\| \vee \|\nabla\mL(\bx, \blambda)\| \leq \Upsilon_u, \quad \forall (\bx, \blambda)\in \mX\times\Lambda.
\end{equation}
By~direct~calculation, we have (the evaluation point is suppressed for simplicity)
\begin{equation}\label{equ:AL:der}
\begin{pmatrix}
\nabla_{\bx}\mL_{\mu, \nu}\\
\nabla_{\blambda}\mL_{\mu, \nu}
\end{pmatrix} = \begin{pmatrix}
I + \nu \nabla_{\bx}^2\mL & \mu G^T\\
\nu G & I
\end{pmatrix}\begin{pmatrix}
\nabla_{\bx}\mL\\
c
\end{pmatrix}.
\end{equation}
Using \eqref{equ:AL:der} and the definition of $(\Delta\bx_t, \Delta\blambda_t)$, we have
\begin{align*}
\begin{pmatrix}
\nabla_{\bx}\mL_{\mu, \nu}^t\\
\nabla_{\blambda}\mL_{\mu, \nu}^t
\end{pmatrix}^T\begin{pmatrix}
\Delta\bx_t\\
\Delta\blambda_t
\end{pmatrix} \; & \stackrel{\mathclap{\eqref{equ:AL:der}}}{=} \; \begin{pmatrix}
\Delta\bx_t\\
\Delta\blambda_t
\end{pmatrix}^T\begin{pmatrix}
I + \nu \nabla_{\bx}^2\mL_t & \mu G_t^T\\
\nu G_t & I
\end{pmatrix}\begin{pmatrix}
\nabla_{\bx}\mL_t\\
c_t
\end{pmatrix} \nonumber\\
& \stackrel{\mathclap{\eqref{equ:Newton}}}{=} - \begin{pmatrix}
\Delta\bx_t\\
\Delta\blambda_t
\end{pmatrix}^T\begin{pmatrix}
I + \nu \nabla_{\bx}^2\mL_t & \mu G_t^T\\
\nu G_t & I
\end{pmatrix}\begin{pmatrix}
B_t & G_t^T\\
G_t & \0
\end{pmatrix}\begin{pmatrix}
\Delta\bx_t\\
\Delta\blambda_t
\end{pmatrix} \nonumber\\
& = - \begin{pmatrix}
\Delta\bx_t\\
\Delta\blambda_t
\end{pmatrix}^T\begin{pmatrix}
B_t + \nu \nabla_{\bx}^2\mL_tB_t + \mu G_t^TG_t & G_t^T + \nu\nabla_{\bx}^2\mL_tG_t^T\\
G_t + \nu G_tB_t & \nu G_tG_t^T 
\end{pmatrix}\begin{pmatrix}
\Delta\bx_t\\
\Delta\blambda_t
\end{pmatrix}.
\end{align*}
Furthermore, using \eqref{equ:upper:bound} and Assumption \ref{ass:1}, we obtain\allowdisplaybreaks
{\small\begin{align}\label{pequ:1}
&\hskip-0.2cm \begin{pmatrix}
\nabla_{\bx}\mL_{\mu, \nu}^t\\
\nabla_{\blambda}\mL_{\mu, \nu}^t
\end{pmatrix}^T\begin{pmatrix}
\Delta\bx_t\\
\Delta\blambda_t
\end{pmatrix} \nonumber \\
& \leq  -\Delta\bx_t^TB_t\Delta\bx_t + \nu\Upsilon_B\Upsilon_u\|\Delta\bx_t\|^2 - \mu\|G_t\Delta\bx_t\|^2 - 2\Delta\blambda_t^TG_t\Delta\bx_t  + \nu(\Upsilon_u + \Upsilon_B)\|\Delta\bx_t\|\|G_t^T\Delta\blambda_t\| - \nu\|G_t^T\Delta\blambda_t\|^2 \nonumber\\
& \stackrel{\mathclap{\eqref{equ:Newton}}}{\leq} - \Delta\bx_t^TB_t\Delta\bx_t + \nu\Upsilon_B\Upsilon_u\|\Delta\bx_t\|^2 - \mu\|c_t\|^2 + 2c_t^T\Delta\blambda_t +\frac{\nu(\Upsilon_u+\Upsilon_B)^2}{2}\|\Delta\bx_t\|^2 - \frac{\nu}{2}\|G_t^T\Delta\blambda_t\|^2  \nonumber\\
& \leq - \Delta\bx_t^TB_t\Delta\bx_t + \nu(\Upsilon_B+\Upsilon_u)^2\|\Delta\bx_t\|^2 - \mu\|c_t\|^2 + \frac{8}{\nu\gamma_{G}}\|c_t\|^2 + \frac{\nu\gamma_{G}}{8}\|\Delta\blambda_t\|^2 \nonumber\\
& \quad - \frac{\nu\gamma_{G}}{4}\|\Delta\blambda_t\|^2 - \frac{\nu}{4}\|G_t^T\Delta\blambda_t\|^2\quad \text{(Young's inequality and Assumption \ref{ass:1})} \nonumber\\
& \stackrel{\mathclap{\eqref{equ:Newton}}}{=}  - \Delta\bx_t^TB_t\Delta\bx_t + \nu(\Upsilon_B+\Upsilon_u)^2\|\Delta\bx_t\|^2 - \rbr{\mu - \frac{8}{\nu\gamma_{G}}}\|c_t\|^2  - \frac{\nu\gamma_{G}}{8}\|\Delta\blambda_t\|^2 - \frac{\nu}{4}\|B_t\Delta\bx_t + \nabla_{\bx}\mL_t\|^2 \nonumber\\
& \leq - \Delta\bx_t^TB_t\Delta\bx_t + \nu(\Upsilon_B+\Upsilon_u)^2\|\Delta\bx_t\|^2 - \rbr{\mu - \frac{8}{\nu\gamma_{G}}}\|c_t\|^2 - \frac{\nu\gamma_{G}}{8}\|\Delta\blambda_t\|^2 - \frac{\nu}{8}\|\nabla_{\bx}\mL_t\|^2 + \frac{\nu\Upsilon_B^2}{4}\|\Delta\bx_t\|^2 \nonumber\\
& \leq  - \Delta\bx_t^TB_t\Delta\bx_t +2\nu(\Upsilon_B+\Upsilon_u)^2\|\Delta\bx_t\|^2 - \rbr{\mu - \frac{8}{\nu\gamma_{G}}}\|c_t\|^2 - \frac{\nu\gamma_{G}}{8}\|\Delta\blambda_t\|^2 - \frac{\nu}{8}\|\nabla_{\bx}\mL_t\|^2,
\end{align}}
\noindent where the second last inequality uses $\|B_t\Delta\bx_t + \nabla_{\bx}\mL_t\|^2 \geq \|\nabla_{\bx}\mL_t\|^2/2 - \|B_t\Delta\bx_t\|^2\geq~\|\nabla_{\bx}\mL_t\|^2/2 \\- \Upsilon_B^2\|\Delta\bx_t\|^2$. To further simplify \eqref{pequ:1}, we decompose the step $\Delta\bx_t$ as
\begin{equation*}
\Delta\bx_t = \Delta\bu_t + \Delta\bv_t,\quad \text{ where } \Delta\bu_t \in \text{span}(G_t^T) \text{ and } G_t\Delta\bv_t = \0.
\end{equation*}
Then, the first two terms of \eqref{pequ:1} can be simplified as \begin{align*}
- & \Delta\bx_t^TB_t\Delta\bx_t  +2\nu(\Upsilon_B+\Upsilon_u)^2\|\Delta\bx_t\|^2 \\
& = -\Delta\bu_t^TB_t\Delta\bu_t - 2\Delta\bu_t^TB_t\Delta\bv_t - \Delta\bv_t^TB_t\Delta\bv_t +2\nu(\Upsilon_B+\Upsilon_u)^2\|\Delta\bx_t\|^2 \\
& \leq \Upsilon_B\|\Delta\bu_t\|^2 + 2\Upsilon_B\|\Delta\bu_t\|\|\Delta\bv_t\| - \gamma_{RH}\|\Delta\bv_t\|^2 +2\nu(\Upsilon_B+\Upsilon_u)^2\|\Delta\bx_t\|^2 \quad \text{(Assumption \ref{ass:1})}\\
& \leq \rbr{\Upsilon_B + \frac{2\Upsilon_B^2}{\gamma_{RH}}}\|\Delta\bu_t\|^2 - \frac{\gamma_{RH}}{2}\|\Delta\bv_t\|^2 +2\nu(\Upsilon_B+\Upsilon_u)^2\|\Delta\bx_t\|^2 \quad \text{(Young's inequality)}\\
& = \rbr{\Upsilon_B+\frac{2\Upsilon_B^2}{\gamma_{RH}}+ \frac{\gamma_{RH}}{2} }\|\Delta\bu_t\|^2 - \rbr{\frac{\gamma_{RH}}{2} - 2\nu(\Upsilon_B+\Upsilon_u)^2}\|\Delta\bx_t\|^2\\
& \leq \rbr{\Upsilon_B+\frac{2\Upsilon_B^2}{\gamma_{RH}}+ \frac{\gamma_{RH}}{2} } \frac{1}{\gamma_{G}}\|c_t\|^2 - \rbr{\frac{\gamma_{RH}}{2} - 2\nu(\Upsilon_B+\Upsilon_u)^2}\|\Delta\bx_t\|^2,
\end{align*}
where the last inequality uses the fact that $\|c_t\|^2 =\|G_t\Delta\bx_t\|^2 = \|G_t\Delta\bu_t\|^2 \geq  \gamma_{G}\|\Delta\bu_t\|^2$.~Here, the inequality is due to Assumption \ref{ass:1}. Combining the above display with \eqref{pequ:1}, we have
\begin{align*}
\begin{pmatrix}
\nabla_{\bx}\mL_{\mu, \nu}^t\\
\nabla_{\blambda}\mL_{\mu, \nu}^t
\end{pmatrix}^T\begin{pmatrix}
\Delta\bx_t\\
\Delta\blambda_t
\end{pmatrix} & \leq - \frac{\nu\gamma_{G}}{8}\nbr{\begin{pmatrix}
\Delta\bx_t\\
\Delta\blambda_t
\end{pmatrix}}^2 - \frac{\nu}{8}\nbr{\begin{pmatrix}
\nabla_{\bx}\mL_t\\
c_t
\end{pmatrix}}^2 - \rbr{\frac{\gamma_{RH}}{2}-2\nu(\Upsilon_B+\Upsilon_u)^2 - \frac{\nu\gamma_{G}}{8}}\|\Delta\bx_t\|^2\\
& \quad -\cbr{\mu - \frac{8}{\nu\gamma_{G}} - \rbr{\Upsilon_B+\frac{2\Upsilon_B^2}{\gamma_{RH}}+ \frac{\gamma_{RH}}{2} } \frac{1}{\gamma_{G}} - \frac{\nu}{8}}\|c_t\|^2.
\end{align*}
Thus, choosing $\Upsilon_1$ large enough (depending only on $\gamma_{G}, \gamma_{RH}, \Upsilon_B$), we complete the proof.

\subsection{Proof of Lemma \ref{lem:3}}\label{pf:lem:3}

Let us denote $\bz_t = (\Delta\bx_t, \Delta\blambda_t)$ and recall that $\bz_{t,\tau} = (\barDelta\bx_t, \barDelta\blambda_t)$ and $\tbz_t = (\tDelta\bx_t, \tDelta\blambda_t)$. By Assumption \ref{ass:1} and the expression \eqref{equ:AL:der}, it is straightforward to see that $\nabla\mL_{\mu, \nu}$ is Lipschitz continuous with a constant $\Upsilon_{\mA\mL}>0$ depending on $(\mu, \nu, \Upsilon_L)$. Thus, using \eqref{equ:update} we have
\begin{align}\label{pequ:2}
& \hskip-0.7cm \mL_{\mu, \nu}^{t+1} \leq \mL_{\mu, \nu}^t + \baralpha_t(\nabla\mL_{\mu, \nu}^t)^T\bz_{t,\tau} + \frac{\Upsilon_{\mA\mL}\baralpha_t^2}{2}\nbr{\bz_{t,\tau}}^2\nonumber \\
& \hskip-0.7cm = \mL_{\mu, \nu}^t + \baralpha_t(\nabla\mL_{\mu, \nu}^t)^T(I + C_t)\bz_t + \baralpha_t(\nabla\mL_{\mu, \nu}^t)^T\cbr{\bz_{t,\tau}-(I + C_t)\bz_t }  + \frac{\Upsilon_{\mA\mL}\baralpha_t^2}{2}\nbr{\bz_{t,\tau}}^2,
\end{align}
where $C_t$ is from Lemma \ref{lem:1}(b). By Lemmas \ref{lem:2} and \ref{lem:1}, the second term can be bounded~as 
\begin{align*}
(\nabla\mL_{\mu, \nu}^t)^T(I + C_t)\bz_t & \leq -\frac{\nu}{\Upsilon_1}\rbr{\nbr{\bz_t}^2 + \nbr{\nabla\mL_t}^2} + \|C_t\|\nbr{\nabla\mL_{\mu, \nu}^t}\nbr{\bz_t} \\
& \stackrel{\mathclap{\eqref{equ:AL:der},\eqref{equ:upper:bound}}}{\leq} \;\;\; -\frac{\nu}{\Upsilon_1}\rbr{\nbr{\bz_t}^2 + \nbr{\nabla\mL_t}^2}  + \rho^\tau \rbr{1+(2\nu +\mu)\Upsilon_u}\nbr{\nabla\mL_t}\nbr{\bz_t}\\
& \leq -\frac{\nu}{\Upsilon_1}\rbr{\nbr{\bz_t}^2 + \nbr{\nabla\mL_t}^2}  + 2\rho^\tau\mu\Upsilon_u\nbr{\nabla\mL_t}\nbr{\bz_t}\\
& \leq - \rbr{\frac{\nu}{\Upsilon_1} - \rho^\tau\mu\Upsilon_u}\rbr{\nbr{\bz_t}^2 + \nbr{\nabla\mL_t}^2},
\end{align*}
where the third inequality uses the facts that $\Upsilon_u\geq 1$ and $1+2\nu\leq \mu$ (as long as $\Upsilon_1\geq 2$). Thus, we can re-define $\Upsilon_1$ as $\Upsilon_1 \leftarrow 2\Upsilon_1\Upsilon_u$. If $\rho^\tau \leq \nu/(\mu\Upsilon_1)$, then we have
\begin{equation}\label{pequ:3}
(\nabla\mL_{\mu, \nu}^t)^T(I + C_t)\bz_t \leq -\frac{\nu}{2\Upsilon_1}\rbr{\nbr{\bz_t}^2 + \nbr{\nabla\mL_t}^2}.
\end{equation}
Now, we deal with the last two terms of \eqref{pequ:2}. By Lemma \ref{lem:1}(b), we have
\begin{align}\label{pequ:4}
\mE\sbr{\bz_{t,\tau} \mid \mF_{t-1}} & = \mE\sbr{\mE\sbr{\bz_{t,\tau} \mid \mF_{t-2/3}}\mid \mF_{t-1}} = \mE\sbr{(I+C_t)\tbz_t\mid \mF_{t-1}} \nonumber\\
& \stackrel{\mathclap{\eqref{equ:Newton}}}{=} -(I+C_t)K_t^{-1}\mE\sbr{\bnabla\mL_t\mid \mF_{t-1}} = -(I+C_t)K_t^{-1}\nabla\mL_t\quad \text{(Assumption \ref{ass:2})} \nonumber\\
& \stackrel{\mathclap{\eqref{equ:Newton}}}{=} (I+C_t)\bz_t.
\end{align}
By Lemma \ref{lem:1}(b, c), we also have
\begin{align*}
& \mE\sbr{\nbr{\bz_{t,\tau} - (I+C_t)\bz_t}^2 \mid \mF_{t-1}} \\
& \leq 3\mE\sbr{\nbr{\bz_{t,\tau} - \tbz_t}^2\mid \mF_{t-1}} + 3\mE\sbr{\nbr{\tbz_t - \bz_t}^2\mid \mF_{t-1}} + 3\|C_t\|^2\nbr{\bz_t}^2 \nonumber\\
& \leq 3\rho^\tau\mE\sbr{\nbr{\tbz_t}^2\mid \mF_{t-1}} + 3\mE\sbr{\nbr{\tbz_t - \bz_t}^2\mid \mF_{t-1}} + 3\rho^{2\tau}\nbr{\bz_t}^2 \nonumber\\
& = 3(\rho^\tau + \rho^{2\tau})\nbr{\bz_t}^2 + 3(1+\rho^\tau)\mE\sbr{\nbr{\tbz_t - \bz_t}^2\mid \mF_{t-1}}\quad (\text{bias-variance decomposition}).
\end{align*}
By Assumption \ref{ass:1} and \cite[Lemma 1]{Na2022adaptive}, there exists a constant $\Upsilon_K\geq~1$~depending on $(\gamma_{G}, \gamma_{RH}, \Upsilon_B)$ such that $\|K_t^{-1}\|\leq \Upsilon_K$. Thus, we apply \eqref{equ:Newton} and \eqref{equ:upper:bound}, and~obtain
\begin{multline}\label{pequ:5}
\hskip-0.3cm \mE\sbr{\nbr{\bz_{t,\tau} - (I+C_t)\bz_t}^2 \mid \mF_{t-1}} \leq 3(\rho^\tau + \rho^{2\tau})\Upsilon_K^2\Upsilon_u^2 +  3(1+\rho^\tau)\Upsilon_K^2\mE[\|\barg_t - \nabla f_t\|^2\mid \mF_{t-1}]  \\
\leq 3(1+\rho^\tau)\Upsilon_K^2(\rho^\tau\Upsilon_u^2+\Upsilon_m) \quad \text{(Assumption \ref{ass:2}\eqref{equ:BM:a})}.
\end{multline}
Thus, using \eqref{pequ:4} and \eqref{pequ:5}, we have \begin{align}\label{pequ:6}
&\hskip-1.2cm \mE\sbr{\baralpha_t(\nabla\mL_{\mu, \nu}^t)^T\cbr{\bz_{t,\tau}-(I + C_t)\bz_t } \mid \mF_{t-1}} \nonumber\\
& \stackrel{\mathclap{\eqref{pequ:4}}}{=} \;\; \mE\sbr{\cbr{\baralpha_t - (\beta_t+\eta_t)/2}\cdot (\nabla\mL_{\mu, \nu}^t)^T\cbr{\bz_{t,\tau}-(I + C_t)\bz_t } \mid \mF_{t-1}} \nonumber\\
& \stackrel{\mathclap{\eqref{equ:sandwich}}}{\leq}\; \frac{\eta_t-\beta_t}{2}\mE\sbr{\nbr{\nabla\mL_{\mu, \nu}^t}\nbr{\bz_{t,\tau}-(I + C_t)\bz_t}\mid \mF_{t-1} } \nonumber\\
& \stackrel{\mathclap{\eqref{equ:upper:bound}}}{\leq}\; \;\frac{\eta_t-\beta_t}{2}(1+(2\nu+\mu)\Upsilon_u)\Upsilon_u\mE\sbr{\nbr{\bz_{t,\tau}-(I + C_t)\bz_t}\mid \mF_{t-1}} \nonumber\\
& \leq (\eta_t-\beta_t)\mu\Upsilon_u^2\sqrt{\mE\sbr{\nbr{\bz_{t,\tau}-(I + C_t)\bz_t}^2\mid \mF_{t-1}}}\quad (1\leq \Upsilon_u \text{ and } 1+2\nu\leq \mu) \nonumber\\
&\stackrel{\mathclap{\eqref{pequ:5}}}{\leq}\;\; 2\mu\Upsilon_K\Upsilon_u^2(1+\rho^\tau)(\sqrt{\Upsilon_m}\vee \Upsilon_u)
(\eta_t-\beta_t) \leq 4\mu\Upsilon_K\Upsilon_u^2(\sqrt{\Upsilon_m}\vee\Upsilon_u)(\eta_t-\beta_t),
\end{align}
and
\begin{align}\label{pequ:7}
\mE[\|\bz_{t,\tau}\|^2&\mid \mF_{t-1}]  \; \;  \stackrel{\mathclap{\eqref{pequ:4}}}{=}\;\;\; \nbr{(I+C_t)\bz_t}^2 + \mE\sbr{\nbr{\bz_{t,\tau} - (I+C_t)\bz_t }^2 \mid \mF_{t-1}} \nonumber\\
& \stackrel{\mathclap{\eqref{equ:Newton},\eqref{equ:upper:bound}}}{\leq}\;\;(1+\rho^\tau)^2\Upsilon_K^2\Upsilon_u^2 + \mE\sbr{\nbr{\bz_{t,\tau} - (I+C_t)\bz_t }^2 \mid \mF_{t-1}}\quad (\text{also use Lemma \ref{lem:1}(b)}) \nonumber\\
&\stackrel{\mathclap{\eqref{pequ:5}}}{\leq} (1+\rho^\tau)^2\Upsilon_K^2\Upsilon_u^2 + 3(1+\rho^\tau)\Upsilon_K^2(\rho^\tau\Upsilon_u^2+\Upsilon_m)
\leq 16\Upsilon_K^2(\Upsilon_u^2\vee\Upsilon_m).
\end{align}
Combining \eqref{pequ:7} with \eqref{pequ:6} and \eqref{pequ:3}, plugging into \eqref{pequ:2}, and using \eqref{equ:sandwich}, we obtain
\begin{multline*}
\mE[\mL_{\mu, \nu}^{t+1}\mid \mF_{t-1}]  \leq \mL_{\mu, \nu}^t - \frac{\nu\beta_t}{2\Upsilon_1}\rbr{\nbr{\bz_t}^2 + \nbr{\nabla\mL_t}^2} \\
+ 4\mu\Upsilon_K\Upsilon_u^2(\sqrt{\Upsilon_m}\vee\Upsilon_u)(\eta_t-\beta_t) + 8\Upsilon_{\mA\mL}\Upsilon_K^2(\Upsilon_u^2\vee\Upsilon_m)\eta_t^2.
\end{multline*}
Choosing $\Upsilon_2$ large enough that depends on $(\mu, \nu, \gamma_{G}, \gamma_{RH}, \Upsilon_B, \Upsilon_m, \Upsilon_L)$, and noting that~$(\mu, \nu)$ are determined by $(\gamma_{G}, \gamma_{RH}, \Upsilon_B)$, we complete the proof.

\subsection{Proof of Theorem \ref{thm:1}}\label{pf:thm:1}

Note that the condition of $\tau$ in the statement implies that we can select $(\mu, \nu)$ to satisfy the condition in Lemma \ref{lem:2} and have $\rho^\tau\leq \nu/(\mu\Upsilon_1)$ with $\rho = 1-\gamma_{S}$. Thus, Lemma \ref{lem:3} leads to \begin{equation*}
\mE[\mL_{\mu, \nu}^{t+1} - \min_{\mX\times\Lambda}\mL_{\mu, \nu}\mid \mF_{t-1}] \leq \mL_{\mu, \nu}^t - \min_{\mX\times\Lambda}\mL_{\mu, \nu} - \frac{\nu\beta_t}{2\Upsilon_1}\|\nabla\mL_t\|^2 + \Upsilon_2(\chi_t+\eta_t^2).
\end{equation*}
By Robbins-Siegmund theorem (see \cite{Robbins1971convergence} or \cite[Theorem~1.3.12]{Duflo1997Random}), we conclude that $\sum_t \beta_t\|\nabla\mL_t\|^2<\infty$. Since $\sum_t \beta_t = \infty$ from \eqref{cond:step}, we~know~that $\liminf_{t\rightarrow\infty}\|\nabla \mL_t\| = 0$. Furthermore, we note that
\begin{multline}\label{pequ:33}
\mE\sbr{\nbr{(\bx_{t+1} - \bx_t,\blambda_{t+1}-\blambda_t)}^2} = \mE\sbr{\mE\sbr{\nbr{(\bx_{t+1} - \bx_t, \blambda_{t+1}-\blambda_t)}^2\mid \mF_{t-1}}} \\ \stackrel{\eqref{equ:update}}{\leq}\eta_t^2\mE\sbr{\mE\sbr{\nbr{\bz_{t,\tau}}^2 \mid \mF_{t-1}} }\stackrel{\eqref{pequ:7}}{\leq} 16\Upsilon_K^2(\Upsilon_u^2\vee\Upsilon_m)\cdot \eta_t^2. 
\end{multline}
Summing over $t=1$ to $\infty$, exchanging the expectation and summation by applying Fubini's theorem \citep[Theorem 1.7.2]{Durrett2019Probability}, and noting that $\sum_t\eta_t^2<\infty$, we obtain
\begin{equation*}
\mE\sbr{\sum_{t=1}^{\infty}\nbr{(\bx_{t+1} - \bx_t,				\blambda_{t+1}-\blambda_t)}^2}<\infty.
\end{equation*}
This implies $\sum_{t=1}^{\infty}\|(\bx_{t+1} - \bx_t, \blambda_{t+1}-\blambda_t)\|<\infty$ almost surely and, thus, $\|(\bx_{t+1}-\bx_t, \blambda_{t+1}-\blambda_t)\|\rightarrow0$ as $t\rightarrow\infty$ almost surely. Suppose for any run of the algorithm $\lim_{t\rightarrow\infty}\|\nabla\mL_t\|\neq 0$, then we have $\limsup_{t\rightarrow\infty}\|\nabla\mL_t\| = \epsilon >0$. Then, there exist two index sequences $\{t_{1,i}\}_i, \{t_{2,i}\}_i$ with $t_{1,i+1}>t_{2,i}> t_{1,i}$ such that, for all $i = 1,2,\ldots$, \begin{equation}\label{pequ:8}
\|\nabla\mL_{t_{1,i}}\| \geq \epsilon/2,\quad\quad \|\nabla\mL_j\| \geq \epsilon/3\; \text{ for } j = t_{1,i}+1,\ldots, t_{2,i}-1, \quad\quad \|\nabla\mL_{t_{2,i}}\| < \epsilon/3.
\end{equation}
Since $\sum_t\beta_t\|\nabla\mL_t\|^2<\infty$, we know
\begin{equation}\label{equ:20}
\infty>\sum_{i=1}^{\infty}\sum_{j=t_{1,i}}^{t_{2,i}-1}\beta_j\|\nabla\mL_j\|^2 \stackrel{\eqref{pequ:8}}{\geq} \frac{\epsilon^2}{9}\sum_{i=1}^{\infty}\sum_{j=t_{1,i}}^{t_{2,i}-1}\beta_j. 
\end{equation}
Furthermore, by \eqref{pequ:33}, we have
\begin{multline*}
\mE\sbr{\nbr{(\bx_{t_{2,i}} - \bx_{t_{1,i}},				\blambda_{t_{2,i}} - \blambda_{t_{1,i}})}} \stackrel{\eqref{pequ:33}}{\leq} 4\Upsilon_K(\Upsilon_u\vee\sqrt{\Upsilon_m})\sum_{j=t_{1,i}}^{t_{2,i}-1}\eta_j \\ \stackrel{\eqref{equ:sandwich}}{=} 4\Upsilon_K(\Upsilon_u\vee\sqrt{\Upsilon_m})\cbr{\sum_{j=t_{1,i}}^{t_{2,i}-1}\beta_j + \sum_{j=t_{1,i}}^{t_{2,i}-1}\chi_j}.
\end{multline*}
Summing over $i=1$ to $\infty$, and noting that $\sum_i\sum_{j=t_{1,i}}^{t_{2,i}-1}\beta_j<\infty$ by \eqref{equ:20} and $\sum_i\sum_{j=t_{1,i}}^{t_{2,i}-1}\chi_j \leq \sum_{j=1}^{\infty}\chi_j<\infty$, we exchange the expectation and summation by applying Fubini's theorem again. We know that the sequence $\{(\bx_{t_{2,i}} - \bx_{t_{1,i}}, \blambda_{t_{2,i}} - \blambda_{t_{1,i}})\}_i$ converges to zero~as~$i\rightarrow~\infty$~with probability one. This contradicts with $\|\nabla\mL_{t_{1,i}}\| \geq \epsilon/2$ and $\|\nabla\mL_{t_{2,i}}\| < \epsilon/3$ in \eqref{pequ:8}.~We~complete the proof.

\subsection{Proof of Corollary \ref{cor:1}}\label{pf:cor:1}

Applying Lemma \ref{lem:3} and taking full expectation, we know for some constants $h_1, h_2>0$,
\begin{equation*}
\mE[\mL_{\mu, \nu}^{t+1}] \leq \mE[\mL_{\mu, \nu}^t] - h_1\beta_t\mE[\|\nabla\mL_t\|^2] + h_2(\chi_t + \eta_t^2), \quad \forall t\geq 0.
\end{equation*}
Rearranging the inequality and summing over $t = 0$ to $\T_{\epsilon}-1$, we obtain
\begin{align*}
& h_1\sum_{t=0}^{\T_{\epsilon}-1}\mE[\|\nabla\mL_t\|^2]  \leq \sum_{t=0}^{\T_{\epsilon}-1}\frac{1}{\beta_t}\rbr{(\mE[\mL_{\mu, \nu}^t]-\min_{\mX\times\Lambda}\mL_{\mu, \nu}) - (\mE[\mL_{\mu, \nu}^{t+1}] - \min_{\mX\times\Lambda}\mL_{\mu, \nu}) } + h_2\sum_{t=0}^{\T_{\epsilon}-1}\frac{\chi_t + \eta_t^2}{\beta_t} \\
& \leq \frac{\mE[\mL_{\mu, \nu}^0]-\min_{\mX\times\Lambda}\mL_{\mu, \nu}}{\beta_0} + \sum_{t = 1}^{\T_{\epsilon}-1}\rbr{\frac{1}{\beta_t} - \frac{1}{\beta_{t-1}}}(\mE[\mL_{\mu, \nu}^t]-\min_{\mX\times\Lambda}\mL_{\mu, \nu}) + h_2\sum_{t=0}^{\T_{\epsilon}-1}\frac{\chi_t + \eta_t^2}{\beta_t}.
\end{align*}
Denoting $\Delta\mL_{\mu, \nu} = \max_{\mX\times\Lambda}\mL_{\mu, \nu} - \min_{\mX\times\Lambda}\mL_{\mu, \nu}$, we further have
\begin{align*}
& \hskip-1cm h_1\sum_{t=0}^{\T_{\epsilon}-1}\mE[\|\nabla\mL_t\|^2] \leq \rbr{\Delta\mL_{\mu, \nu}\vee h_2}\cbr{\frac{1}{\beta_0} + \sum_{t=1}^{\T_{\epsilon}-1}\rbr{\frac{1}{\beta_t} - \frac{1}{\beta_{t-1}}} + \sum_{t=0}^{\T_{\epsilon}-1}\frac{\chi_t + \eta_t^2}{\beta_t}}\\
& =  \rbr{\Delta\mL_{\mu, \nu}\vee h_2}\cbr{\frac{1}{\beta_{\T_{\epsilon}-1}} + \sum_{t=0}^{\T_{\epsilon}-1}\frac{\chi_t + \eta_t^2}{\beta_t} } = \rbr{\Delta\mL_{\mu, \nu}\vee h_2}\cbr{\T_{\epsilon}^a + \sum_{t=0}^{\T_{\epsilon}-1}\frac{\chi_t + \eta_t^2}{\beta_t}}.
\end{align*}
For the last term on the right hand side, we have
\begin{align*}
\sum_{t=0}^{\T_{\epsilon}-1}\frac{\chi_t + \eta_t^2}{\beta_t} & = \sum_{t=0}^{\T_{\epsilon}-1}\cbr{(t+1)^{a-b} + (t+1)^a\rbr{(t+1)^{-2a} + 2(t+1)^{-(a+b)} + (t+1)^{-2b}  }}\\
& \leq \sum_{t=0}^{\T_{\epsilon}-1}(t+1)^{a-b} + 4\sum_{t=0}^{\T_{\epsilon}-1}(t+1)^{-a} = 5 + \sum_{t=1}^{\T_{\epsilon}-1}\cbr{(t+1)^{a-b} + 4(t+1)^{-a}}\\
&\leq 5 + \int_0^{\T_{\epsilon}-1}(t+1)^{a-b} + 4(t+1)^{-a}dt \quad (\text{by the convexity of } x^{p} \text{ with } p<0)\\
& \leq \begin{cases*}
5 + \frac{\T_{\epsilon}^{1+a-b}}{1+a-b} + \frac{4\T_{\epsilon}^{1-a}}{1-a} & \text{ if } $1+a>b$,\\
5 + \log(\T_{\epsilon}) + \frac{4\T_{\epsilon}^{1-a}}{1-a} & \text{ if } $1+a=b$,\\
5 + \frac{1}{b-a-1}+ \frac{4\T_{\epsilon}^{1-a}}{1-a} & \text{ if } $1+a<b$.
\end{cases*}
\end{align*}
Combining the above two displays, dividing $\T_{\epsilon}$ on both sides, and using ``$\lesssim$" to neglect constant factors (i.e., not depending on $\T_{\epsilon}$), we obtain {\allowdisplaybreaks
\begin{align*}
\epsilon^2 & \leq \rbr{\frac{1}{\T_{\epsilon}}\sum_{t=0}^{\T_{\epsilon}-1}\mE[\|\nabla\mL_t\|] }^2 \leq \frac{1}{\T_{\epsilon}}\sum_{t=0}^{\T_{\epsilon}-1}\rbr{\mE[\|\nabla\mL_t\|]}^2 \leq  \frac{1}{\T_{\epsilon}}\sum_{t=0}^{\T_{\epsilon}-1}\mE[\|\nabla\mL_t\|^2]\\
& \lesssim \begin{cases*}
\frac{1}{\T_{\epsilon}^{1-a}} + \frac{1}{\T_{\epsilon}^{b-a}} + \frac{1}{\T_{\epsilon}^a} &  \text{ if } $1+a>b$,\\
\frac{1}{\T_{\epsilon}^{1-a}}  + \frac{1}{\T_{\epsilon}^a} & \text{ if } $1+a=b$,\\
\frac{1}{\T_{\epsilon}^{1-a}}  + \frac{1}{\T_{\epsilon}^a} & \text{ if } $1+a<b$,\\
\end{cases*} \quad (\text{use } 1/\T_{\epsilon} \leq 1/\T_{\epsilon}^a \text{ and } \log(\T_{\epsilon})/\T_{\epsilon} \lesssim 1/\T_{\epsilon}^a),\\
& \lesssim \begin{cases*}
\frac{1}{\T_{\epsilon}^{b-a}} + \frac{1}{\T_{\epsilon}^a} &  \text{ if } $1>b$,\\
\frac{1}{\T_{\epsilon}^{1-a}} + \frac{1}{\T_{\epsilon}^a} &  \text{ if } $1\leq b$,
\end{cases*} = \frac{1}{\T_{\epsilon}^{(1\wedge b)-a}} + \frac{1}{\T_{\epsilon}^a} \lesssim \frac{1}{\T_{\epsilon}^{a\wedge (1-a)\wedge (b-a)} }.
\end{align*}}
\noindent \hskip-4pt This completes the proof.

\section{Proofs of Section \ref{sec:4}}

\subsection{Proof of Lemma \ref{lem:4}}\label{pf:lem:4}

For notational brevity, we let $\bomega_t = (\bx_t - \tx, \blambda_t-\tlambda)$. By the scheme of Algorithm \ref{alg:1},~we~have
\begin{align*}
\bomega_{t+1} & \stackrel{\mathclap{\eqref{equ:update}}}{=}\bomega_t + \baralpha_t\bz_{t,\tau} = \bomega_t + \varphi_t\bz_{t,\tau} + \rbr{\baralpha_t -\varphi_t} \bz_{t,\tau}\\
& = \bomega_t +\varphi_t(I + C_t)\tbz_t + \varphi_t\cbr{\bz_{t,\tau} - (I + C_t)\tbz_t} + \rbr{\baralpha_t - \varphi_t}\bz_{t,\tau}\\
&\stackrel{\mathclap{\eqref{equ:Newton}}}{=}\bomega_t - \varphi_t(I + C_t)K_t^{-1}\bnabla\mL_t + \varphi_t\cbr{\bz_{t,\tau} - (I + C_t)\tbz_t} + \rbr{\baralpha_t - \varphi_t} \bz_{t,\tau}\\
& = \bomega_t - \varphi_t(I+C_t)K_t^{-1}\nabla\mL_t - \varphi_t(I+C_t)K_t^{-1}(\bnabla\mL_t - \nabla\mL_t) + \varphi_t\cbr{\bz_{t,\tau}- (I + C_t)\tbz_t} \\
& \quad + \rbr{\baralpha_t - \varphi_t} \bz_{t,\tau}\\
&\stackrel{\mathclap{\eqref{rec:def:b}}}{=} \; \bomega_t - \varphi_t(I+C_t)K_t^{-1}\nabla\mL_t  + \varphi_t\btheta^t+ \rbr{\baralpha_t - \varphi_t} \bz_{t,\tau}\\
& =\bomega_t - \varphi_t(I+C_t)(K^\star)^{-1}\nabla\mL_t - \varphi_t(I+C_t)\cbr{K_t^{-1} - (K^\star)^{-1}}\nabla\mL_t + \varphi_t\btheta^t+ \rbr{\baralpha_t - \varphi_t} \bz_{t,\tau}\\
&\stackrel{\mathclap{\eqref{rec:def:d}}}{=}\; \cbr{I - \varphi_t(I+C_t)}\bomega_t - \varphi_t(I+C_t)(K^\star)^{-1}\bpsi^t - \varphi_t(I+C_t)\cbr{K_t^{-1} - (K^\star)^{-1}}\nabla\mL_t \\
&\quad + \varphi_t\btheta^t+ \rbr{\baralpha_t - \varphi_t} \bz_{t,\tau}\\
& \stackrel{\mathclap{\eqref{rec:def:c}}}{=}\;\; \cbr{I - \varphi_t(I+C^\star)}\bomega_t + \varphi_t(\btheta^t + \bdelta^t) +  \rbr{\baralpha_t - \varphi_t} \bz_{t,\tau}.
\end{align*}
We apply the above equation recursively and show the result.~Moreover, under \mbox{Assumptions}~\ref{ass:2} and \ref{ass:3}, we know $\mE[\barg_i - \nabla f_i \mid \mF_{i-1}] = \0$ and, by \eqref{pequ:4}, $\mE\sbr{\bz_{i,\tau} - (I + C_i)\tbz_i \mid \mF_{i-1}} =\0$. Thus, $\mE[\btheta^i\mid \mF_{i-1}] = \0$ and $\btheta^i$ is a martingale difference.

\subsection{Proof of Lemma \ref{lem:5}}\label{pf:lem:5}

Let us denote $\rank(S) = r$. Since $K_t$, $K^\star$ have full rank, $\rank(K_tS) = \rank(K^\star S) = r$. Let $K^\star S = ED F^T$ be the truncated singular value decomposition of $K^\star S$. We have
{\small\begin{equation*}
E \in \mR^{(d+m)\times r}, \quad F\in \mR^{q\times r}, \quad E^TE = F^TF = I, \quad D = \diag(D_1, \ldots, D_r)\;\; \text{ with } D_1\geq \ldots\geq D_r>0.
\end{equation*}}
\noindent \hskip-3.5pt Similarly, we let $K_t S = E'D'(F')^T$. By direct calculation, we have
\begin{equation}\label{UU:1}
\|K_tS(S^TK_t^2S)^\dagger S^TK_t - K^\star S(S^T(K^\star)^2S)^\dagger S^TK^\star\| = \|EE^T - E'(E')^T\|.
\end{equation}
Define the principle angles $\theta_p$ between $\Span(E)$ and $\Span(E')$ to be $\theta_p = (\theta_{p,1}, \ldots, \theta_{p,r})$,~so~that $E^TE'$ has the singular value decomposition $E^TE' = P\cos(\theta_{p})Q^T$, where $P, Q\in \mR^{r\times r}$ are~orthonormal matrices and $\cos(\theta_{p}) = \diag(\cos(\theta_{p,1}), \ldots, \cos(\theta_{p,r}))$ (similar for $\sin(\theta_p)$).~We~\mbox{further} let $E^\perp\in \mR^{(d+m)\times (d+m-r)}$ be the complement of $E$, and express $E'$ as
\begin{equation}\label{UU:2}
E' = EA + E^\perp B.
\end{equation}
Then, $E^TE' = A = P\cos(\theta_{p})Q^T$ and $I = (E')^TE' = A^TA+B^TB$. By the above formulation,
\begin{align}\label{UU:3}
\|EE^T - E'(E')^T\|\;\; & \stackrel{\mathclap{\eqref{UU:2}}}{=} \;\; \nbr{(E, E^\perp)\begin{pmatrix}
I - AA^T & -AB^T\\
-BA^T & -BB^T
\end{pmatrix}\begin{pmatrix}
E^T\\
(E^\perp)^T
\end{pmatrix}} = \nbr{\begin{pmatrix}
I - AA^T & -AB^T\\
-BA^T & -BB^T
\end{pmatrix}} \nonumber\\
& \leq \nbr{\begin{pmatrix}
I - AA^T & \0\\
\0 & -BB^T
\end{pmatrix}} + \nbr{\begin{pmatrix}
\0 & AB^T\\
BA^T&\0
\end{pmatrix}} \nonumber\\
& \leq \max\{\|I - AA^T\|, \|BB^T\|\} + \|AB^T\| \nonumber\\
& = \max\{\|I - AA^T\|, \|I - A^TA\|\} + \|AB^T\| \nonumber\\
& = \|\sin(\theta_{p})\|^2 + \sqrt{\|P\cos(\theta_{p})\sin^2(\theta_{p})\cos(\theta_{p})P^T\|} \nonumber\\
& = \|\sin(\theta_{p})\|^2 + \|\sin(\theta_{p})\cos(\theta_{p})\| \leq 2\|\sin(\theta_{p})\|.
\end{align}
On the other hand, by Wedin's $\sin(\Theta)$ theorem \cite[(3.1)]{Wedin1972Perturbation}, we know
\begin{equation}\label{UU:4}
\|\sin(\theta_{p})\| \leq \frac{\|(K^\star - K_t)S\|}{D_r}.
\end{equation}
We let $F_r$ be the $r$-th column of $F$ and have $D_r^2 = F_r^TS^T(K^\star)^2S F_r \geq (\sigma_{\min}(K^\star))^2F_r^TS^TSF_r$. Since $\KER(K^\star S) = \KER(S)$ and $F_r\in \KER^\perp(K^\star S)$, we know $F_r \in \KER^\perp(S) =~\Span(S^T)$. Thus, $F_r^TS^TSF_r \geq \lambda_{\min}^+(S^TS)$, where $\lambda_{\min}^+(S^TS) = (\sigma_{\min}^{+}(S))^2$ is the least positive~eigenvalue of $S^TS$. Therefore, we have
\begin{equation}\label{UU:5}
D_r \geq \sigma_{\min}(K^\star)\sigma_{\min}^{+}(S).
\end{equation}
Combining all above derivations, we obtain
\begin{multline*}
\|K_tS(S^TK_t^2S)^\dagger S^TK_t - K^\star S(S^T(K^\star)^2S)^\dagger S^TK^\star\| \stackrel{\eqref{UU:1}}{=}\|EE^T - E'(E')^T\| \stackrel{\eqref{UU:3}}{\leq} 2\|\sin(\theta_{p})\| \\
\stackrel{\eqref{UU:4}}{\leq} \frac{2\|K_t-K^\star\|\cdot \|S\|}{D_r}\stackrel{\eqref{UU:5}}{\leq} \frac{2\|K_t-K^\star\|}{\sigma_{\min}(K^\star)}\cdot\frac{\|S\|}{\sigma_{\min}^{+}(S)}.
\end{multline*}
This completes the proof.

\subsection{Proof of Corollary \ref{cor:2}}\label{pf:cor:2}

Denote $A_t = I - \mE[K_tS(S^TK_t^2S)^\dagger S^TK_t\mid \bx_t, \blambda_t]$ and $A^\star = I -  \mE[K^\star S(S^T(K^\star)^2S)^\dagger S^TK^\star]$.~We~have \begin{align*}
\nbr{C_t-C^\star} & = \nbr{A_t^\tau - (A^\star)^\tau}\leq \nbr{A_t^{\tau-1}(A_t-A^\star)} + \nbr{(A_t^{\tau-1} - (A^\star)^{\tau-1})A^\star}\\
& \leq \|A_t-A^\star\| + \nbr{A_t^{\tau-1} - (A^\star)^{\tau-1}}\quad (\|A_t\| \vee \|A^\star\| \leq 1)\\
& \leq \tau \|A_t-A^\star\| \leq \tau \mE\sbr{\nbr{K_tS(S^TK_t^2S)^\dagger S^TK_t - K^\star S(S^T(K^\star)^2S)^\dagger S^TK^\star}\mid \bx_t, \blambda_t}\\
& \leq \frac{2\tau\|K_t-K^\star\|}{\sigma_{\min}(K^\star)} \mE\sbr{\|S\|\|S^\dagger\|} \leq \frac{2\tau\Upsilon_S}{\sigma_{\min}(K^\star)}\nbr{K_t-K^\star}\quad (\text{by Assumption \ref{ass:5}}).
\end{align*}
This completes the proof.

\subsection{Proof of Theorem \ref{thm:3}}\label{pf:thm:3}

We present some lemmas that bound $\I_{1,t}$, $\I_{2,t}$, and $\I_{3,t}$ in \eqref{rec:a}, \eqref{rec:b}, and \eqref{rec:c},~\mbox{respectively}. The proofs of these lemmas are presented in Appendix \ref{pf:lem:6} -- \ref{pf:lem:11}.

\begin{lemma}\label{lem:6}
Under Assumptions \ref{ass:1}, \ref{ass:2}(\ref{equ:BM:a}, \ref{equ:BM:e}), \ref{ass:3}, and $(\bx_t, \blambda_t)\rightarrow(\tx,\tlambda)$, suppose
\begin{equation}\label{cond:varphi}
\lim\limits_{t\rightarrow\infty} t\rbr{1 - \varphi_{t-1}/\varphi_t} = \varphi<0, \quad\quad \lim\limits_{t\rightarrow\infty} t\varphi_t = \tvarphi\in(0, \infty],\quad\quad 1.5(1-\rho^\tau) + \varphi/\tvarphi>0.
\end{equation}
Then, for any $\upsilon>0$,
\begin{equation}\label{equ:I_1t:as}
\I_{1,t} = o\big(\sqrt{\varphi_t\{\log(1/\varphi_t)\}^{1+\upsilon} }\; \big) \quad \text{ a.s}.
\end{equation}
Furthermore, if (\ref{equ:BM:a}) is strengthened to (\ref{equ:BM:b}), then we have
\begin{enumerate}[topsep=3pt,itemsep=0em,label=(\alph*):]
\item (asymptotic rate) $\I_{1,t} = O(\sqrt{\varphi_t\log(1/\varphi_t)})$ a.s.
\item (asymptotic normality) $\sqrt{1/\varphi_t}\cdot\I_{1,t} \stackrel{d}{\longrightarrow}\N(0, \Xi^\star)$ where $\Xi^\star$ is from (\ref{equ:Xi:n}).
\item (Berry-Esseen bound) For any vector $\bw = (\bw_{\bx}, \bw_{\blambda})\in \mR^{d+m}$ such that $\bw^T\Xi^\star\bw \neq 0$,
\begin{equation*}
\sup_{z\in\mR}\abr{P\rbr{\frac{\sqrt{1/\varphi_t}\cdot \bw^T\I_{1,t}}{\sqrt{\bw^T\Xi^\star\bw}} \leq z}  - P\rbr{\N(0,1)\leq z}} = O\big(\sqrt{\varphi_t}\log(1/\varphi_t)\; \big).
\end{equation*}
\end{enumerate}	

\end{lemma}

\begin{lemma}\label{lem:7}
Under the conditions of Lemma \ref{lem:6} with (\ref{equ:BM:a}) and assume for the adaptivity gap $\chi_t$ that for some $p, q\in(0, 1]$,
\begin{equation}\label{cond:chi}
\lim\limits_{t\rightarrow\infty} t\rbr{1 - \chi_{t-1}/\chi_t} = \chi<\varphi,\;\; (1 - \rho^\tau) + p(\chi - 0.5\varphi)/\tvarphi>0, \;\; (1 - \rho^\tau) + q(\chi - \varphi)/\tvarphi>0.
\end{equation}
Then, for any $\nu>0$ (if $q<1$, the second $O(\cdot)$ can be strengthened to $o(\cdot)$)
\begin{equation*}
\I_{2,t} = o\big(\chi_t^{p}/\varphi_t^{0.5p}\sqrt{\{\log(1/\chi_t)\}^{1+\nu}}\;\big) + O(\chi_t^q/\varphi_t^q)\quad \text{ a.s.}
\end{equation*}
Furthermore, if (\ref{equ:BM:a}) is strengthened to (\ref{equ:BM:b}), then we have (if $p<1$, the first $O(\cdot)$ can~also be strengthened to $o(\cdot)$)
\begin{equation*}
\I_{2,t} = O\big(\chi_t^{p}/\varphi_t^{0.5p}\sqrt{\log(1/\chi_t)}\; \big) + O(\chi_t^q/\varphi_t^q)\quad \text{ a.s.}
\end{equation*}

\end{lemma}

\begin{lemma}\label{lem:8}
Under the conditions of Lemma \ref{lem:7} with (\ref{equ:BM:a}), we have for any $\upsilon>0$,
\begin{equation*}
\I_{3,t} = o\big(\sqrt{\varphi_t\{\log(1/\varphi_t)\}^{1+\upsilon} } \; \big) + o\big(\chi_t^{p}/\varphi_t^{0.5p}\sqrt{\{\log(1/\chi_t)\}^{1+\nu}}\big) + o(\chi_t^q/\varphi_t^q) = o(\I_{1,t}+\I_{2,t})\quad \text{a.s.}
\end{equation*}
If (\ref{equ:BM:a}) is strengthened to (\ref{equ:BM:b}), the above result holds with $\upsilon = 0$.
\end{lemma}

We apply the above lemmas.~We first check the conditions \eqref{cond:varphi} and \eqref{cond:chi}. Since~$\varphi_t = (\beta_t + \eta_t)/2 = \beta_t + \chi_t/2$ and $\chi_t = o(\beta_t)$ (as implied by $\chi<\beta$), we know $\beta_t\leq \varphi_t \leq \beta_t + o(\beta_t)$~and $\lim\limits_{t\rightarrow\infty}t\varphi_t =\lim\limits_{t\rightarrow\infty}t\beta_t =\tbeta$. Furthermore, we have
\begin{align*}
\lim\limits_{t\rightarrow\infty} &t\rbr{1 - \frac{\varphi_{t-1}}{\varphi_t}} = \lim\limits_{t\rightarrow\infty} t\rbr{1 - \frac{\beta_{t-1}}{\beta_t} + \frac{\beta_{t-1}}{\beta_t}\cbr{1 - \frac{2+\chi_{t-1}/\beta_{t-1}}{2+\chi_t/\beta_t}}}\\
& = \beta + \lim\limits_{t\rightarrow\infty} t \rbr{ 1 - \frac{2+\chi_{t-1}/\beta_{t-1}}{2+\chi_t/\beta_t} } = \beta + \frac{1}{2}\lim\limits_{t\rightarrow \infty}t\rbr{\frac{\chi_t}{\beta_t} - \frac{\chi_{t-1}}{\beta_{t-1}}} \quad (\text{since } \chi_t = o(\beta_t))\\
& = \beta + \frac{1}{2}\lim\limits_{t\rightarrow\infty} \frac{\chi_t}{\beta_t}\cdot t\rbr{1 - \frac{\chi_{t-1}}{\chi_t}\cdot \frac{\beta_t}{\beta_{t-1}}} = \beta +  \frac{\chi - \beta}{2}\lim\limits_{t\rightarrow\infty} \frac{\chi_t}{\beta_t} = \beta.
\end{align*}
The above derivations show that $\varphi = \beta$ and $\tvarphi = \tbeta$. Thus, \eqref{cond:n:1} implies \eqref{cond:varphi} holds. Moreover, for any constant $p\in(0,1]$ such that $(1-\rho^\tau) + p(\chi-0.5\beta)/\tbeta = (1-\rho^\tau) + p(\chi-0.5\varphi)/\tvarphi>0$,~we simply let $q=p$ and have $(1-\rho^\tau) + q(\chi-\varphi)/\tvarphi>0$.~Thus, \eqref{cond:chi} holds with $q=p$. We~note~that for any $\nu\geq 0$, 
\begin{align*}
& \lim\limits_{t\rightarrow\infty}t\rbr{1 - \frac{\varphi_{t-1}^{0.5p}\{\log(1/\chi_{t-1})\}^{0.5(1+\nu)}}{\varphi_t^{0.5p}\{\log(1/\chi_t)\}^{0.5(1+\nu)}}} \quad \stackrel{\mathclap{\text{Lem. \ref{aux:lem:5}}}}{=} \quad\; 0.5p\varphi + 0.5(1+\nu)\lim_{t\rightarrow\infty}t\rbr{1 - \frac{\log(1/\chi_{t-1})}{\log(1/\chi_t)}} \\
& = 0.5p\varphi+0.5(1+\nu)\lim\limits_{t\rightarrow\infty} t\rbr{\frac{\log(\chi_{t-1}/\chi_t)}{\log(1/\chi_t)}} = 0.5p\varphi+0.5(1+\nu) \lim_{t\rightarrow\infty} t \rbr{\frac{\frac{\chi_{t-1} - \chi_t}{\chi_t} + O\rbr{\frac{(\chi_{t-1} - \chi_t)^2}{\chi_t^2}}}{\log(1/\chi_t)}}\\
&  = 0.5p\varphi - 0.5(1+\nu)\chi\lim_{t\rightarrow\infty}1/\log(1/\chi_t) = 0.5p\varphi<0 \quad \text{(by \eqref{cond:varphi} and $\chi_t\rightarrow 0$)}.
\end{align*}
Thus, by Lemma \ref{aux:lem:1}, we know $\chi_t^{p}/\varphi_t^{0.5p}\sqrt{\{\log(1/\chi_t)\}^{1+\nu}} = o(\chi_t^p/\varphi_t^p)$.~The convergence~rate~of $(\bx_t, \blambda_t)$ comes from Lemmas \ref{lem:4}, \ref{lem:6}, \ref{lem:7}, \ref{lem:8}, the above fact, and the fact that $\beta_t\leq \varphi_t\leq 2\beta_t$. We complete the proof.

\subsubsection{Proof of Lemma \ref{lem:6}}\label{pf:lem:6}

We need a preparation lemma. Recall that we suppose $(\tx,\tlambda)$ is a local solution of \eqref{pro:1} with $G^\star$ being full row rank and $\nabla_{\bx}^2\mL^\star$ being positive definite in the null space~\mbox{$\{\bx\in\mR^d: G^\star\bx = \0\}$}.

\begin{lemma}\label{lem:11}
Under Assumptions \ref{ass:1}, \ref{ass:2}(\ref{equ:BM:d}) and $(\bx_t, \blambda_t)\rightarrow(\tx,\tlambda)$, we have $\frac{1}{t}\sum_{i=0}^{t-1}\bnabla_{\bx}^2\mL_i \\ \rightarrow \nabla_{\bx}^2\mL^\star$ as $t\rightarrow\infty$. Further, with a small $\gamma_{RH}$ and a large $\Upsilon_B$, $\Delta_t=\0$ for all large enough~$t$.
\end{lemma}

Let $I+C^\star = U\Sigma U^T$ with $\Sigma = \diag(\sigma_1, \ldots, \sigma_{d+m})$ be the eigenvalue decomposition.~Then, \begin{equation}\label{def:I1t}
\I_{1,t} \stackrel{\eqref{rec:a}}{=} \sum_{i=0}^t\prod_{j=i+1}^{t}\cbr{I - \varphi_j(I+C^\star)}\varphi_i\btheta^i = U \sum_{i=0}^t\prod_{j=i+1}^{t}\cbr{I - \varphi_j\Sigma}\varphi_i U^T\btheta^i.
\end{equation}
Since $\mE[\btheta^i | \mF_{i-1}] = \0$, we aim to apply the strong law of large number \citep[Theorem 1.3.15]{Duflo1997Random}, the central limit theorem \citep[Corollary 2.1.10]{Duflo1997Random}, and the Berry-Esseen~inequality \citep[Theorem 2.1]{Fan2019Exact} to show each result in the lemma. We~\mbox{compute}~the~\mbox{conditional}~covariance of $\I_{1,t}$, which is defined as \citep[Proposition 1.3.7]{Duflo1997Random}
\begin{equation}\label{I_1t_q}
\langle\I_{1}\rangle_t \coloneqq U\sum_{i=0}^{t}\prod_{j=i+1}^{t}\cbr{I - \varphi_j\Sigma}\varphi_i^2 U^T\mE[\btheta^i(\btheta^i)^T\mid \mF_{i-1}]U\big(\prod_{j=i+1}^{t}\cbr{I - \varphi_j\Sigma}\big)^TU^T.
\end{equation}
For the term $\mE[\btheta^i(\btheta^i)^T\mid \mF_{i-1}]$, we note that
\begin{align}\label{var:pro}
\hskip-1cm\mE[\btheta^i(\btheta^i)^T  &\mid \mF_{i-1}]  \stackrel{\eqref{rec:def:b}}{=} \mE\big[\cbr{(I+C_i)K_i^{-1}(\bnabla\mL_i - \nabla\mL_i)-\cbr{\bz_{i,\tau}-  (I+C_i)\tbz_i } } \nonumber\\
& \quad \quad \cbr{(I+C_i)K_i^{-1}(\bnabla\mL_i - \nabla\mL_i)-\cbr{\bz_{i,\tau}- (I+C_i)\tbz_i } }^T \mid \mF_{i-1}\big] \nonumber\\
&\stackrel{\mathclap{\eqref{pequ:4}}}{=}\;\; (I+C_i)K_i^{-1}\mE\sbr{(\bnabla\mL_i - \nabla\mL_i)(\bnabla\mL_i - \nabla\mL_i)^T\mid \mF_{i-1}}K_i^{-1}(I+C_i) \nonumber\\
&  \quad + \mE[\cbr{\bz_{i,\tau}-  (I+C_i)\tbz_i }\cbr{\bz_{i,\tau}-  (I+C_i)\tbz_i }^T \mid \mF_{i-1}] \eqqcolon \J_{1,i} + \J_{2,i}. 
\end{align}
For the term $\J_{1,i}$, we apply Assumption \ref{ass:2} and have $\mE[(\barg_i-\nabla f_i)(\barg_i-\nabla f_i)^T \mid \mF_{i-1}] =~\mE[\barg_i\barg_i^T\mid \mF_{i-1}] - \nabla f_i\nabla^T f_i$. We also note that
\begin{align*}
& \nbr{\mE[\barg_i\barg_i^T - \nabla f(\bx^\star; \xi)\nabla^T f(\bx^\star; \xi)\mid \mF_{i-1}]} \\
& \leq 2\mE\sbr{\|\barg_i - \nabla f(\bx^\star;\xi)\|\cdot \|\barg_i\| \mid \mF_{i-1}} + \mE\sbr{\|\barg_i - \nabla f(\bx^\star;\xi)\|^2\mid \mF_{i-1}}\\
& \leq 2\sqrt{\mE\sbr{\|\barg_i - \nabla f(\bx^\star;\xi)\|^2\mid \mF_{i-1}}}\sqrt{\mE\sbr{\|\barg_i\|^2\mid \mF_{i-1}}} + \mE\sbr{\|\barg_i - \nabla f(\bx^\star;\xi)\|^2\mid \mF_{i-1}},
\end{align*}
and
\begin{equation*}
\mE\sbr{\|\barg_i - \nabla f(\bx^\star;\xi)\|^2\mid \mF_{i-1}} \leq \mE[\sup_{\bx\in\mX}\|\nabla^2f(\bx;\xi)\|^2]\cdot \|\bx_i - \bx^\star\|^2 \stackrel{\eqref{equ:BM:e}}{\leq} \Upsilon_m\|\bx_i - \bx^\star\|^2.
\end{equation*}
By Assumptions \ref{ass:1}, \ref{ass:2}\eqref{equ:BM:a}, we suppose $\|\nabla f_i\|\leq \Upsilon_u$ (we abuse $\Upsilon_u$ from \eqref{equ:upper:bound}) and obtain
\begin{equation*}
\mE[\|\barg_i\|^2\mid \mF_{i-1}] =  \|\nabla f_i\|^2 + \mE[\|\barg_i - \nabla f_i\|^2\mid \mF_{i-1}] \leq \Upsilon_u^2 + \Upsilon_m\leq 2(\Upsilon_u^2\vee\Upsilon_m).
\end{equation*}
Combining the above three displays, we have
\begin{multline}\label{Lip:map}
\nbr{\mE[\barg_i\barg_i^T - \nabla f(\bx^\star; \xi)\nabla^T f(\bx^\star; \xi)\mid \mF_{i-1}]} \\ \leq 2\sqrt{2\Upsilon_m}(\Upsilon_u\vee\sqrt{\Upsilon_m})(\|\bx_i - \bx^\star\| + \|\bx_i-\bx^\star\|^2)\rightarrow 0.
\end{multline}
This implies that
\begin{equation}\label{var:lim}
\lim\limits_{i\rightarrow \infty}\mE[(\barg_i-\nabla f_i)(\barg_i-\nabla f_i)^T \mid \mF_{i-1}] = \mE[\nabla f(\bx^\star; \xi)\nabla^T f(\bx^\star; \xi)] - \nabla f(\bx^\star)\nabla^T f(\bx^\star).
\end{equation}
Furthermore, by Lemma \ref{lem:11} we know $K_i\rightarrow K^\star$ as $i\rightarrow \infty$. Since $\|K_iS(S^TK_t^2S)^\dagger S^TK_i\| \leq 1$, we apply dominated convergence theorem \citep[Theorem~1.6.7]{Durrett2019Probability} and Lemma~\ref{lem:5},~and have $\lim_{i\rightarrow \infty}\mE[K_iS(S^TK_i^2S)^\dagger S^TK_i\mid \bx_i, \blambda_i] = \mE[K^\star S(S^T(K^\star)^2S)^\dagger S^TK^\star]$. Here, the~expectation is taken over randomness of $S$. Thus, $C_i\rightarrow C^\star$. By the definition \eqref{equ:Omega}, we obtain
\begin{equation}\label{J_1i}
\J_{1,i} = (I+C^\star)\tOmega(I+C^\star) + O(\K_{1,i})
\end{equation}
with $\K_{1,i} \rightarrow 0$ as $i\rightarrow \infty$ almost surely. For the term $\J_{2,i}$, we apply \eqref{equ:z:recur} and define~$\tC_i\coloneqq-\prod_{j = 0}^{\tau-1}C_{i,j}$. Then, we have
\begin{align}\label{J_2i}
\J_{2,i} & = \mE[(\tC_i - C_i)\tbz_i\tbz_i^T(\tC_i^T-C_i^T) \mid \mF_{i-1}] \stackrel{\mathclap{\eqref{equ:Newton}}}{=} \mE[(\tC_i - C_i)K_i^{-1}\bnabla\mL_i\bnabla^T\mL_iK_i^{-1}(\tC_i^T - C_i^T) \mid \mF_{i-1} ] \nonumber\\
& = \mE[(\tC_i - C_i)K_i^{-1}(\bnabla\mL_i- \nabla\mL_i)(\bnabla\mL_i- \nabla\mL_i)^TK_i^{-1}(\tC_i^T - C_i^T) \mid \mF_{i-1} ] \nonumber\\
& \quad + \mE[(\tC_i - C_i)K_i^{-1}\nabla\mL_i\nabla^T\mL_iK_i^{-1}(\tC_i^T - C_i^T) \mid \mF_{i-1} ] \nonumber\\
& \quad + \mE[(\tC_i - C_i)K_i^{-1}(\bnabla\mL_i- \nabla\mL_i)\nabla^T\mL_iK_i^{-1}(\tC_i^T - C_i^T) \mid \mF_{i-1} ] \nonumber\\
& \quad + \mE[(\tC_i - C_i)K_i^{-1}\nabla\mL_i(\bnabla\mL_i- \nabla\mL_i)^TK_i^{-1}(\tC_i^T - C_i^T) \mid \mF_{i-1} ].
\end{align}
For the last two terms, we apply the tower property of conditional expectation by first~conditioning on the randomness of $\{S_{i,j}\}_j$ to take expectation over the randomness of $\xi_i$,  and then taking expectation~over the randomness of $\{S_{i,j}\}_j$. In particular, we have (similar for~the second last term in \eqref{J_2i})
\begin{align*}
& \mE[(\tC_i - C_i)K_i^{-1}\nabla\mL_i(\bnabla\mL_i - \nabla\mL_i)^TK_i^{-1}(\tC_i^T - C_i^T) \mid \mF_{i-1} ]\\
& = \mE[(\tC_i - C_i)K_i^{-1}\nabla\mL_i \mE[\bnabla\mL_i - \nabla\mL_i\mid \mF_{i-1}\cup \sigma(\{S_{i,j}\}_j)]^TK_i^{-1}(\tC_i^T - C_i^T) \mid \mF_{i-1} ]\\
& = \mE[(\tC_i - C_i)K_i^{-1}\nabla\mL_i\mE[\bnabla\mL_i - \nabla\mL_i\mid \mF_{i-1}]^TK_i^{-1}(\tC_i^T - C_i^T) \mid \mF_{i-1} ] = \0.
\end{align*}
For the second term in \eqref{J_2i}, it converges to zero almost surely as $i\rightarrow\infty$ since $\|\tC_i\|\vee\|C_i\|\leq 1$, $\|K_i^{-1}\|\leq \Upsilon_K$, and $\nabla\mL_i\rightarrow 0$. For the first term in \eqref{J_2i}, we have
\begin{align*}
& \mE[(\tC_i - C_i)K_i^{-1}(\bnabla\mL_i - \nabla\mL_i)(\bnabla\mL_i - \nabla\mL_i)^TK_i^{-1}(\tC_i^T - C_i^T) \mid \mF_{i-1} ]\\
& = \mE[(\tC_i - C_i)K_i^{-1}\mE\sbr{(\bnabla\mL_i - \nabla\mL_i)(\bnabla\mL_i - \nabla\mL_i)^T\mid \mF_{i-1}\cup \sigma(\{S_{i,j}\}_j)} K_i^{-1}(\tC_i^T - C_i^T) \mid \mF_{i-1} ]\\
& = \mE[(\tC_i - C_i)K_i^{-1}\mE\sbr{(\bnabla\mL_i - \nabla\mL_i)(\bnabla\mL_i - \nabla\mL_i)^T\mid \mF_{i-1}}K_i^{-1}(\tC_i^T - C_i^T) \mid \mF_{i-1} ] \\
& \stackrel{\mathclap{\eqref{equ:tC}}}{\longrightarrow}\; \mE[(\tC^\star - C^\star)\tOmega((\tC^\star)^T - C^\star)] = \mE[\tC^\star\tOmega(\tC^\star)^T] - C^\star\tOmega C^\star.
\end{align*}
Again, the convergence here is due to the dominated convergence theorem, \eqref{var:lim}, and~$K_i\rightarrow K^\star$;~and the expectation is taken over the randomness of $\tau$ sketch matrices $S_1,\ldots, S_\tau$ only. Thus, combining the above two displays with \eqref{J_2i}, we have \begin{equation}\label{J_2ii}
\J_{2,i} =  \mE[\tC^\star\tOmega(\tC^\star)^T] - C^\star\tOmega C^\star + O(\K_{2,i})
\end{equation}
with $\K_{2,i}\rightarrow 0$ as $i\rightarrow \infty$ almost surely. Combining \eqref{J_2ii}, \eqref{J_1i}, and \eqref{var:pro}, we obtain
\begin{equation*}
\mE[\btheta^i(\btheta^i)\mid \mF_{i-1}] = \mE[(I+\tC^\star)\tOmega(I+\tC^\star)^T] + O(\K_{1,i}+\K_{2,i}).
\end{equation*}
By the definition of $\langle \I_1\rangle_t$ in \eqref{I_1t_q}, let us denote $\Gamma \coloneqq U^T\mE[(I+\tC^\star)\tOmega(I+\tC^\star)^T] U$.~For~any~$k, l\in \{1,\ldots,d+m\}$, the $(k,l)$ entry of the matrix $U^T\langle \I_{1}\rangle_t U$ can be written as
\begin{equation*}
[U^T\langle \I_{1}\rangle_t U]_{k,l} = \sum_{i=0}^{t}\prod_{j = i+1}^t(1 - \varphi_j\sigma_k)(1-\varphi_j\sigma_l)\varphi_i^2(\Gamma_{kl} + r_{i,kl}),
\end{equation*}
where $r_{i,kl}\rightarrow 0$ as $i\rightarrow \infty$ almost surely. By Lemma \ref{lem:1}(b) and the fact that $C_i\rightarrow C^\star$~as~$i\rightarrow \infty$, we know $\|C^\star\|\leq \rho^\tau$. Since $C^\star\preceq \0$, we have $0<1-\rho^\tau\leq \sigma_i \leq 1$ for $i=1,\ldots, d+m$, which implies $\sigma_k+\sigma_l \geq 2(1 - \rho^\tau)$. Using the condition \eqref{cond:varphi}, Lemmas \ref{aux:lem:1} and \ref{aux:lem:2}, we~obtain $[U^T\langle \I_{1}\rangle_t U]_{k,l}/\varphi_t\rightarrow \Gamma_{kl}/(\sigma_k+\sigma_l + \varphi/\tvarphi)$ as $t\rightarrow \infty$ almost surely. Thus, by \eqref{equ:Xi:n}, we have
\begin{equation}\label{equ:I_t_var}
\langle \I_{1}\rangle_t /\varphi_t\stackrel{a.s.}{\longrightarrow} U(\Theta\circ \Gamma)U^T = \Xi^\star.
\end{equation}
Then, \cite[Theorem 1.3.15]{Duflo1997Random} indicates \eqref{equ:I_1t:as} holds. This shows the first part~of~the~results. For the second part of the results, we assume the condition~$\eqref{equ:BM:b}$ and have \begin{align}\label{pequ:27}
&\mE[\|\btheta^i\|^3\mid \mF_{i-1}]\;\; \stackrel{\mathclap{\eqref{rec:def:b}}}{\leq}\;\; 4\rbr{\mE[\|(I+C_i)K_i^{-1}(\bnabla\mL_i - \nabla\mL_i)\|^3 \mid \mF_{i-1}] + \mE[\|\bz_{i,\tau} - (I+C_i)\tbz_i\|^3\mid \mF_{i-1}]} \nonumber\\
& \stackrel{\mathclap{\eqref{equ:z:recur}}}{\leq} 4\rbr{8\Upsilon_K^3\mE[\|\barg_i - \nabla f_i\|^3\mid \mF_{i-1}] + \mE[\|(\tC_i - C_i)\tbz_i\|^3\mid \mF_{i-1}]}\quad (\|C_i\|\leq 1,  \|K_i^{-1}\|\leq \Upsilon_K) \nonumber\\
& \stackrel{\mathclap{\eqref{equ:BM:b}}}{\leq}\; 4\rbr{8\Upsilon_K^3\Upsilon_m+ 8\mE[\|\tbz_i\|^3\mid \mF_{i-1}]} \quad (\|\tC_i\|\vee\|C_i\|\leq 1) \nonumber\\
& \stackrel{\mathclap{\eqref{equ:Newton}}}{\leq} 4\rbr{8\Upsilon_K^3\Upsilon_m + 8\Upsilon_K^3\mE[\|\bnabla\mL_i\|^3\mid \mF_{i-1}]} \quad (\|K_i^{-1}\|\leq \Upsilon_K) \nonumber\\
&\stackrel{\mathclap{\eqref{equ:Newton}}}{\leq} 4\rbr{8\Upsilon_K^3\Upsilon_m + 8\Upsilon_K^3\cbr{4\|\nabla\mL_i\|^3 + 4\mE[\|\barg_i - \nabla f_i\|^3\mid \mF_{i-1}]}} \nonumber\\
& \stackrel{\mathclap{\eqref{equ:BM:b}}}{\leq} 4\rbr{8\Upsilon_K^3\Upsilon_m + 8\Upsilon_K^3\cbr{4\Upsilon_u^3 + 4\Upsilon_m}}\quad (\text{also use } \eqref{equ:upper:bound}).
\end{align}
Thus, $\btheta^i$ has bounded third moment; and \cite[pp. 554]{Wang1995Asymptotic} together with \eqref{equ:I_t_var} give~the result (a). For (b), we verify the Lindeberg's condition. For any $\epsilon>0$, we have
\begin{align*}
&\frac{1}{\varphi_t}\sum_{i=0}^{t}\mE\big[\big\|\prod_{j=i+1}^{t}\cbr{I - \varphi_j(I+C^\star)}\varphi_i\btheta^i\big\|^2\cdot \b1_{\nbr{\prod_{j=i+1}^{t}\cbr{I - \varphi_j(I+C^\star)}\varphi_i\btheta^i}\geq \epsilon\sqrt{\varphi_t} } \mid \mF_{i-1}\big]\\
& \leq \frac{1}{\epsilon\varphi_t^{3/2}}\sum_{i=0}^{t}\mE\big[\big\|\prod_{j=i+1}^{t}\cbr{I - \varphi_j(I+C^\star)}\varphi_i\btheta^i\big\|^3\mid \mF_{i-1}\big]\\
& = \frac{1}{\epsilon\varphi_t^{3/2}}\sum_{i=0}^{t}\mE\big[\big\|\prod_{j=i+1}^{t}\cbr{I - \varphi_j\Sigma}\varphi_iU^T\btheta^i\big\|^3\mid \mF_{i-1}\big].
\end{align*}
To show the right hand side converges to zero, it suffices to show that each entry of the vector on the right hand side converges to zero. In particular, we show for any $1\leq k\leq d+m$,
\begin{equation*}
\frac{1}{\epsilon\varphi_t^{3/2}}\sum_{i=0}^{t} \prod_{j=i+1}^{t}\abr{1-\varphi_j\sigma_k}^3\varphi_i^3\mE[|[U^T\btheta^i]_k|^3\mid \mF_{i-1}] \longrightarrow 0\quad\quad  \text{ as } \quad t\rightarrow\infty.
\end{equation*}
By \eqref{pequ:27} and $\mE[|[U^T\btheta^i]_k|^3| \mF_{i-1}] \leq \mE[\|\btheta^i\|^3| \mF_{i-1}]$, we only show $\sum_{i=0}^{t} \prod_{j=i+1}^{t}|1-\varphi_j\sigma_k|^3\varphi_i^3 = o(\varphi_t^{3/2})$. Without loss of generality, we suppose $1 - \varphi_j\sigma_k\geq 0$ for all $j\geq1$ and show
\begin{equation}\label{pequ:28}
\sum_{i=0}^{t} \prod_{j=i+1}^{t}(1-\varphi_j\sigma_k)^3\varphi_i^3 = o(\varphi_t^{3/2}).
\end{equation}
Otherwise, since $\varphi<0$ from \eqref{cond:varphi}, Lemma \ref{aux:lem:1} shows that $\varphi_i\rightarrow 0$. Thus, there exists $\tilde{t}$ such that $1 - \varphi_j\sigma_k\geq 0$, $\forall j\geq \tilde{t}$. Then, 
\begin{align}\label{pequ:29}
\sum_{i=0}^{t} \prod_{j=i+1}^{t}\abr{1-\varphi_j\sigma_k}^3\varphi_i^3 & = \sum_{i=0}^{\tilde{t}-2}\prod_{j=i+1}^{t}\abr{1-\varphi_j\sigma_k}^3\varphi_i^3 + \sum_{i=\tilde{t}-1}^{t}\prod_{j=i+1}^{t}(1-\varphi_j\sigma_k)^3\varphi_i^3 \nonumber\\
& = \prod_{j=\tilde{t}}^{t}(1-\varphi_j\sigma_k)^3 \sum_{i=0}^{\tilde{t}-2}\prod_{j=i+1}^{\tilde{t}-1}\abr{1-\varphi_j\sigma_k}^3\varphi_i^3  + \sum_{i=\tilde{t}-1}^{t}\prod_{j=i+1}^{t}(1-\varphi_j\sigma_k)^3\varphi_i^3 \nonumber\\
& = \sum_{i=\tilde{t}-1}^{t}\prod_{j=i+1}^{t}(1-\varphi_j\sigma_k)^3(\varphi_i')^3,
\end{align}
where 
\begin{equation*}
\varphi_{\tilde{t}-1}' = \rbr{\sum_{i=0}^{\tilde{t}-2}\prod_{j=i+1}^{\tilde{t}-1}\abr{1-\varphi_j\sigma_k}^3\varphi_i^3 + \varphi_{\tilde{t}-1}^3}^{1/3}, \quad \text{ and }\quad \varphi_i' = \varphi_i,\quad \forall i\geq \tilde{t}. 
\end{equation*}
Note that \eqref{pequ:29} has the same form as \eqref{pequ:28}, and $\varphi_i'$ differs from $\varphi_i$ only at $i = \tilde{t}-1$.~Thus, \eqref{pequ:29} and \eqref{pequ:28} have the same limit. For \eqref{pequ:28}, we apply Lemma \ref{aux:lem:5} and observe that
\begin{equation*}
\lim\limits_{i\rightarrow\infty}i\rbr{1 - \varphi_{i-1}^2/\varphi_i^2}  \stackrel{\eqref{cond:varphi}}{=} 2\varphi\quad\quad \text{ and }\quad\quad 3\sigma_k + 2\varphi/\tvarphi\stackrel{\eqref{cond:varphi}}{>}0.
\end{equation*}
Thus, Lemma \ref{aux:lem:2} suggests that 
\begin{equation*}
\sum_{i=0}^{t} \prod_{j=i+1}^{t}(1-\varphi_j\sigma_k)^3\varphi_i^3 = O(\varphi_t^2).
\end{equation*}
This verifies \eqref{pequ:28} and further verifies the Lindeberg's condition. Thus, the central~limit~theorem of martingale in \cite[Corollary 2.1.10]{Duflo1997Random} leads to (b). For (c), we apply \citep[Theorem 2.1]{Fan2019Exact} with $\epsilon = \sqrt{\varphi_t}$, $\delta = 0$, $\rho=1$ (in their notation), as proved for~\mbox{verifying}~the~Lindeberg's condition above, and obtain the result immediately. This completes the~proof.

\subsubsection{Proof of Lemma \ref{lem:7}}\label{pf:lem:7}

We have
\begin{equation}\label{def:I2t}
\I_{2,t} \stackrel{\eqref{rec:b}}{=} \sum_{i=0}^t\prod_{j=i+1}^{t}\cbr{I - \varphi_j(I+C^\star)}(\baralpha_i - \varphi_i)\bz_{i,\tau} = U\sum_{i=0}^t\prod_{j=i+1}^{t}\cbr{I - \varphi_j\Sigma}(\baralpha_i - \varphi_i)U^T\bz_{i,\tau}.
\end{equation}
Thus, for any $1\leq k\leq d+m$, we have $[U^T\I_{2,t}]_k = \sum_{i=0}^{t}\prod_{j=i+1}^{t}(1 - \varphi_j\sigma_k)(\baralpha_i - \varphi_i)[U^T\bz_{i,\tau}]_k$. For the same reason as \eqref{pequ:28} and \eqref{pequ:29}, we suppose for any $j\geq 0$ that $1 - \varphi_j\sigma_k \geq 0$.~Then,
\begin{align}\label{equ:UI_2tk}
& \abr{[U^T\I_{2,t}]_k} \leq  \frac{1}{2}\sum_{i=0}^{t}\prod_{j=i+1}^{t}|1 - \varphi_j\sigma_k|\chi_i \abr{[U^T\bz_{i,\tau}]_k} = \frac{1}{2}\sum_{i=0}^{t}\prod_{j=i+1}^{t}(1 - \varphi_j\sigma_k)\chi_i \abr{[U^T\bz_{i,\tau}]_k} \nonumber\\
& = \frac{1}{2}\sum_{i=0}^{t}\prod_{j=i+1}^{t}(1 - \varphi_j\sigma_k)\chi_i \mE\sbr{\abr{[U^T\bz_{i,\tau}]_k} \mid \mF_{i-1}} + \frac{1}{2}\sum_{i=0}^{t}\prod_{j=i+1}^{t}(1 - \varphi_j\sigma_k)\chi_i \big\{\abr{[U^T\bz_{i,\tau}]_k} \nonumber \\
& \quad -  \mE\sbr{\abr{[U^T\bz_{i,\tau}]_k} \mid \mF_{i-1}} \big\} \eqqcolon \J_{3,t,k} + \J_{4,t,k}. 
\end{align}
We analyze $\J_{3,t,k}$ and $\J_{4,t,k}$ separately as follows. We first show $\abr{[U^T\bz_{i,\tau}]_k}$ has bounded~variance. We have
\begin{multline}\label{pequ:9}
\mE\big[\cbr{\abr{[U^T\bz_{i,\tau}]_k} -  \mE\sbr{\abr{[U^T\bz_{i,\tau}]_k} \mid \mF_{i-1}} }^2\mid \mF_{i-1} \big] \\
\leq \mE\big[\abr{[U^T\bz_{i,\tau}]_k}^2\mid \mF_{i-1}\big] \leq \mE\sbr{\|\bz_{i,\tau}\|^2\mid \mF_{i-1}}\stackrel{\eqref{pequ:7}}{\leq} 16\Upsilon_K^2(\Upsilon_u^2\vee \Upsilon_m).
\end{multline}
Thus, $\J_{4,t,k}$ is square integrable. Its variance is bounded by
\begin{align*}
\langle\J_{4,k}\rangle_t & \coloneqq \frac{1}{4}\sum_{i=0}^t\prod_{j=i+1}^{t}(1 - \varphi_j\sigma_k)^2\chi_i^2\mE\sbr{\cbr{\abr{[U^T\bz_{i,\tau}]_k} -  \mE\sbr{\abr{[U^T\bz_{i,\tau}]_k} \mid \mF_{i-1}} }^2\mid \mF_{i-1} }\\
& \stackrel{\mathclap{\eqref{pequ:9}}}{\leq}\;\; 4\Upsilon_K^2(\Upsilon_u^2\vee \Upsilon_m)\sum_{i=0}^t\prod_{j=i+1}^{t}(1 - \varphi_j\sigma_k)^2\chi_i^2.
\end{align*}
Using \eqref{cond:varphi} and \eqref{cond:chi}, we know \begin{multline}\label{com:chi:phi}
\lim\limits_{i\rightarrow \infty}i\rbr{1 - \frac{\chi_{i-1}^2/\varphi_{i-1}}{\chi_i^2/\varphi_i}}  = \lim\limits_{i\rightarrow \infty}i\rbr{1 - \frac{\chi_{i-1}^2}{\chi_i^2} + \frac{\chi_{i-1}^2}{\chi_i^2}\rbr{1 - \frac{\varphi_i}{\varphi_{i-1}}} }\\
= \lim\limits_{i\rightarrow \infty}i\cbr{\rbr{1 - \frac{\chi_{i-1}}{\chi_i}}\rbr{1+ \frac{\chi_{i-1}}{\chi_i}} - \frac{\chi_{i-1}^2}{\chi_i^2}\frac{\varphi_i}{\varphi_{i-1}}\rbr{1 - \frac{\varphi_{i-1}}{\varphi_i}} } = 2\chi - \varphi.
\end{multline}
Further, \eqref{cond:chi} implies $2\sigma_k + p(2\chi - \varphi)/\tvarphi >0$ for some constant $p\in(0, 1]$. Thus, Lemma~\ref{aux:lem:2} leads to $\langle\J_{4,k}\rangle_t = O(\chi_t^{2p}/\varphi_t^p)$ (when $p\in(0,1)$, $O(\cdot)$ can be strengthened to $o(\cdot)$); and the strong law of large number \citep[Theorem 1.3.15]{Duflo1997Random} suggests that for any $\nu>0$,
\begin{equation}\label{pequ:30}
\J_{4,t,k} = o\rbr{\sqrt{\chi_t^{2p}/\varphi_t^p\cdot\{\log(\varphi_t^p/ \chi_t^{2p}) \}^{1+\nu} } } = o\rbr{\sqrt{\chi_t^{2p}/\varphi_t^p\cdot\{\log(1/ \chi_t) \}^{1+\nu} } }.
\end{equation}
If \eqref{equ:BM:a} is strengthened to \eqref{equ:BM:b}, then we follow \eqref{pequ:9}, \eqref{pequ:7}, and \eqref{pequ:27}, and can show $\abr{[U^T\bz_{i,\tau}]_k} -  \mE\sbr{\abr{[U^T\bz_{i,\tau}]_k} \mid \mF_{i-1}}$ has bounded third moment. Thus, \cite[pp. 554]{Wang1995Asymptotic} suggests that $\J_{4,t,k} = O(\chi_t^{p}/\varphi_t^{0.5p}\sqrt{\log (1/\chi_t)})$. When $p\in(0,1)$, $O(\cdot)$ can be strengthened~to $o(\cdot)$ due to $\langle\J_{4,k}\rangle_t = o(\chi_t^{2p}/\varphi_t^p)$. For the term $\J_{3,t,k}$, we have
\begin{multline}\label{pequ:10}
\J_{3,t,k}  \leq \frac{1}{2}\sum_{i=0}^{t}\prod_{j=i+1}^{t}(1 - \varphi_j\sigma_k)\chi_i \sqrt{\mE[\abr{[U^T\bz_{i,\tau}]_k}^2 \mid \mF_{i-1}]} \; \\
\stackrel{\mathclap{\eqref{pequ:9}}}{\leq}\;\; 2\Upsilon_K(\Upsilon_u\vee\sqrt{\Upsilon_m})  \sum_{i=0}^{t}\prod_{j=i+1}^{t}(1 - \varphi_j\sigma_k)\chi_i.
\end{multline}
Using \eqref{cond:varphi}, \eqref{cond:chi}, and the facts that $\lim\limits_{i\rightarrow \infty} i(1 - \frac{\chi_{i-1}/\varphi_{i-1}}{\chi_i/\varphi_i}) =  \chi - \varphi$ and $\sigma_k + q(\chi-\varphi)/\tvarphi >0$ (as implied by \eqref{cond:chi}), we apply Lemma \ref{aux:lem:2} and obtain $\J_{3,t,k} = O(\chi_t^q/\varphi_t^q)$. When $q\in (0,1)$,~$O(\cdot)$ can be strengthened to $o(\cdot)$. Combining with \eqref{equ:UI_2tk} and the bound of $\J_{4,t,k}$, we~complete~the proof.

\subsubsection{Proof of Lemma \ref{lem:8}}\label{pf:lem:8}

Based on the definition of $\I_{3,t}$ in \eqref{rec:c}, we have the recursion
\begin{equation}\label{pequ:18}
\I_{3,t+1} = \cbr{I - \varphi_{t+1}(I+C^\star)}\I_{3,t} + \varphi_{t+1}\bdelta^{t+1}.
\end{equation}
By Assumption \ref{ass:1} and the fact that $\|C_t\|\leq 1$, we have
\begin{align}\label{pequ:32}
&\hskip-1cm \nbr{\bdelta^t}\;\; \stackrel{\mathclap{\eqref{rec:def:c}}}{\leq} \;\; 2\rbr{\|(K^\star)^{-1}\|\|\bpsi^t\| + \|K_t^{-1} - (K^\star)^{-1}\|\cdot \|\nabla\mL_t\|}  + \|C_t - C^\star\|\cdot\nbr{\begin{pmatrix}
\bx_t-\bx^\star\\
\blambda_t-\blambda^\star
\end{pmatrix}}  \nonumber\\
& \leq 2\Upsilon_K\Upsilon_L\nbr{\begin{pmatrix}
\bx_t-\bx^\star\\
\blambda_t-\blambda^\star
\end{pmatrix}}^2 + (2\Upsilon_K^2\Upsilon_u\|K_t - K^\star\| + \|C_t - C^\star\|) \nbr{\begin{pmatrix}
\bx_t-\bx^\star\\
\blambda_t-\blambda^\star
\end{pmatrix}} .
\end{align}
Since $K_t\rightarrow K^\star$ (cf. Lemma \ref{lem:11}) and $C_t\rightarrow C^\star$, we know
\begin{equation}\label{pequ:19}
\bdelta^t = o(\|(\bx_t - \bx^\star, \blambda_t - \blambda^\star)\|).
\end{equation}
Using $\|C^\star\|\leq \rho^\tau$, we know for any $a\in(0, 1)$, there exists an integer~$t_1$~such that~for~any~$t \geq~t_1$,
\begin{align*}
\|\I_{3,t+1}\| & \leq \cbr{1 - \varphi_{t+1}(1-\rho^\tau)}\|\I_{3,t}\| + \varphi_{t+1}\cdot o\rbr{\nbr{(\bx_{t+1} - \bx^\star,	\blambda_{t+1} - \blambda^\star)}}\\
& \leq \cbr{1 - \varphi_{t+1}(1-\rho^\tau) + o(\varphi_{t+1})}\|\I_{3,t}\| + \varphi_{t+1}\cdot o(\|\I_{1,t}\|+\|\I_{2,t}\|) \quad (\text{by Lemma \ref{lem:4}})\\
& \leq\cbr{1 - a(1-\rho^\tau)\varphi_{t+1}}\|\I_{3,t}\| + \varphi_{t+1}\cdot o(\|\I_{1,t}\|+\|\I_{2,t}\|).
\end{align*}
We apply the above inequality recursively and obtain
\begin{multline}\label{pequ:11}
\|\I_{3,t+1}\| \leq \prod_{j=t_1+1}^{t+1}\cbr{1 - a(1-\rho^\tau)\varphi_j}\|\I_{3,t_1}\|  \\
+ \sum_{i=t_1+1}^{t+1}\prod_{j=i+1}^{t+1}\cbr{1 - a(1-\rho^\tau)\varphi_j}\varphi_io(\|\I_{1,i-1}\|+\|\I_{2,i-1}\|).
\end{multline}
We apply Lemmas \ref{lem:6} and \ref{lem:7} for bounding $\|\I_{1,i-1}\|$ and $\|\I_{2,i-1}\|$. In particular, we note~that for any $\upsilon\geq 0$,
\begin{align*}
& \lim\limits_{i\rightarrow\infty}i \rbr{1 - \frac{\sqrt{\varphi_{i-1}\cbr{\log(1/\varphi_{i-1})}^{1+\upsilon} }}{\sqrt{\varphi_{i}\cbr{\log(1/\varphi_{i})}^{1+\upsilon} }}} \;\; \stackrel{\mathclap{\eqref{cond:varphi}}}{=}\;\; \lim\limits_{i\rightarrow\infty} i \rbr{1 - \frac{\sqrt{\varphi_{i-1}}}{\sqrt{\varphi_{i}}}} + \lim\limits_{i\rightarrow\infty} i \rbr{1 - \frac{\cbr{\log(1/\varphi_{i-1})}^{\frac{1+\upsilon}{2}}}{\cbr{\log(1/\varphi_{i})}^{\frac{1+\upsilon}{2}}}}\\
&\stackrel{\mathclap{\eqref{cond:varphi}}}{=} \;\; \frac{\varphi}{2} + \lim\limits_{i\rightarrow\infty} i \rbr{1 - \frac{\cbr{\log(1/\varphi_{i-1})}^{\frac{1+\upsilon}{2}}}{\cbr{\log(1/\varphi_{i})}^{\frac{1+\upsilon}{2}}}} \quad (\text{Lemma \ref{aux:lem:5}}).
\end{align*}
Furthermore, we have
\begin{align*}
\lim\limits_{i\rightarrow\infty} i\rbr{1 - \frac{\log(1/\varphi_{i-1})}{\log(1/\varphi_{i})}} & = \lim\limits_{i\rightarrow\infty}\frac{i\log(\varphi_{i-1}/\varphi_{i})}{\log(1/\varphi_{i})} = \lim\limits_{i\rightarrow\infty}\frac{i\log\rbr{1 + (\varphi_{i-1}-\varphi_{i})/\varphi_{i}}}{\log(1/\varphi_{i})}\\
& = \lim\limits_{i\rightarrow\infty}\frac{i\cbr{\frac{\varphi_{i-1} - \varphi_{i}}{\varphi_{i}} + O\rbr{\frac{(\varphi_{i-1} - \varphi_{i})^2}{\varphi_{i}^2}}} }{\log(1/\varphi_{i})} = \lim\limits_{i\rightarrow\infty}\frac{-\varphi}{\log(1/\varphi_{i})} = 0,
\end{align*}
where the last equality is due to $\varphi_i\rightarrow 0$, as implied by Lemma \ref{aux:lem:1}. Combining the above~two displays with Lemma \ref{aux:lem:5}, we have
\begin{equation}\label{pequ:12}
\lim\limits_{i\rightarrow\infty}i \rbr{1 - \frac{\sqrt{\varphi_{i-1}\cbr{\log(1/\varphi_{i-1})}^{1+\upsilon} }}{\sqrt{\varphi_{i}\cbr{\log(1/\varphi_{i})}^{1+\upsilon} }}} = \frac{\varphi}{2}\quad\quad \text{for any } \;\; \nu\geq 0.
\end{equation}
Moreover, we have for any $p, q\in(0,1]$ and $\nu\geq 0$,
\begin{equation}\label{pequ:13}
\begin{aligned}
\lim\limits_{i\rightarrow \infty}i\rbr{1 - \frac{\chi_{i-1}^p/\varphi_{i-1}^{0.5p}\sqrt{\{\log(1/\chi_{i-1})\}^{1+\nu}} }{\chi_{i}^p/\varphi_{i}^{0.5p}\sqrt{\{\log(1/\chi_{i})\}^{1+\nu}} }} &  \stackrel{\substack{\eqref{pequ:12}\\ \eqref{com:chi:phi}}}{=} p(\chi - 0.5\varphi),\\
\lim\limits_{i\rightarrow \infty}i\rbr{1 - \frac{\chi_{i-1}^q/\varphi_{i-1}^{q}}{\chi_{i}^q/\varphi_{i}^{q} }} &  \stackrel{\substack{\eqref{com:chi:phi}}}{=} q(\chi - \varphi).
\end{aligned}
\end{equation}
For the constants $p,q$ in \eqref{cond:chi}, we let $a$ be any scalar such that
\begin{equation*}
0<\frac{-\varphi/\tvarphi}{2(1-\rho^\tau)}\vee \frac{-p(\chi-0.5\varphi)/\tvarphi}{1-\rho^\tau}\vee \frac{-q(\chi-\varphi)/\tvarphi}{1-\rho^\tau}
<a<1,
\end{equation*}
which is guaranteed to exist due to \eqref{cond:varphi} and \eqref{cond:chi}. Then, we obtain
\begin{equation}\label{pequ:14}
a(1-\rho^\tau) + \frac{\varphi}{2\tvarphi} >0 \quad a(1-\rho^\tau) + \frac{p(\chi-0.5\varphi)}{\tvarphi}>0 \quad \text{and}\quad a(1-\rho^\tau) + \frac{q(\chi-\varphi)}{\tvarphi}>0.
\end{equation}
Thus, combining \eqref{pequ:11}, \eqref{pequ:12}, and \eqref{pequ:13} with Lemma \ref{aux:lem:2}, we obtain the results \mbox{under}~\mbox{either} \eqref{equ:BM:a} or \eqref{equ:BM:b}. This completes the proof.

\subsubsection{Proof of Lemma \ref{lem:11}}\label{pf:lem:11}

We note that
\begin{multline}\label{pequ:25}
\nbr{\frac{1}{t}\sum_{i=0}^{t-1}\bnabla_{\bx}^2\mL_i - \nabla_{\bx}^2\mL^\star} \leq \nbr{\frac{1}{t}\sum_{i=0}^{t-1} \barH_i - \nabla^2f_i} + \frac{1}{t}\sum_{i=0}^{t-1}\nbr{\nabla_{\bx}^2\mL_i - \nabla_{\bx}^2\mL^\star}  \\
\stackrel{\mathclap{\eqref{Lip:mL}}}{\leq} \;\; \nbr{\frac{1}{t}\sum_{i=0}^{t-1} \barH_i - \nabla^2f_i} + \frac{\Upsilon_L}{t} \sum_{i=0}^{t-1}\nbr{\begin{pmatrix}
\bx_i - \bx^\star\\
\blambda_i - \blambda^\star
\end{pmatrix}}.
\end{multline}
Since $(\bx_t-\bx^\star, \blambda_t-\blambda^\star) \rightarrow \0$, by the fact that $a_t\rightarrow a$ implies $\frac{1}{t}\sum_{i=0}^{t-1}a_i\rightarrow a$ (also known as Stolz–Ces\`aro theorem), it suffices to show $(\sum_{i=0}^{t-1} \barH_i - \nabla^2f_i)/t$ converges to~zero.~In~fact,~by~Assumption \ref{ass:2}\eqref{equ:BM:d} that $\mE[\barH_i \mid \mF_{i-1}] = \nabla^2 f_i$ and $\mE[\|\barH_i - \nabla^2f_i\|^2\mid \mF_{i-1}]\leq \Upsilon_m$, we notice $(\sum_{i=0}^{t-1} \barH_i - \nabla^2f_i)/t$ is a square integrable martingale. Thus, \cite[Theorem 1.3.15]{Duflo1997Random} suggests that for any $\upsilon>0$,
\begin{equation}\label{pequ:26}
\nbr{\frac{1}{t}\sum_{i=0}^{t-1} \barH_i - \nabla^2f_i} = o\rbr{\sqrt{\frac{(\log t)^{1+\upsilon}}{t}}}.
\end{equation}
Combining \eqref{pequ:25} and \eqref{pequ:26}, we obtain $\frac{1}{t}\sum_{i=0}^{t-1}\bnabla_{\bx}^2\mL_i \rightarrow \nabla_{\bx}^2\mL^\star$ as $t\rightarrow\infty$. For the~second~result, we suppose $\|\nabla_{\bx}^2\mL^\star\|\leq \Upsilon_B^\star$ and $\bx^T\nabla_{\bx}^2\mL^\star\bx \geq \gamma_{RH}^\star\|\bx\|^2$ in the space $\{\bx\in \mR^d: G^\star\bx = \0\}$. Whenever $\gamma_{RH}<\gamma_{RH}^\star$ and $\Upsilon_B>\Upsilon_B^\star$, we know $\|\frac{1}{t}\sum_{i=0}^{t-1}\bnabla_{\bx}^2\mL_i\|\leq \Upsilon_B$ for large enough $t$. In addition, we let $Z_t, Z^\star\in\mR^{d\times (d-m)}$ be the matrices whose columns are orthonormal and span the spaces of $\ker(G_t)$, $\ker(G^\star)$, respectively. Since $G_t\rightarrow G^\star$, Davis-Kahan $\sin(\theta)$ theorem suggests that $Z_tZ_t^T\rightarrow Z^\star(Z^\star)^T$, implying $\inf_{Q}\|Z_t - Z^\star Q\|\rightarrow 0$ with $Q$ chosen over all $(d-m)\times (d-m)$ orthogonal matrices \citep{Davis1970Rotation}. Thus, we have
\begin{equation*}
\lambda_{\min}(Z_t^T(\sum_{i=0}^{t-1}\bnabla_{\bx}^2\mL_i/t) Z_t) = \lambda_{\min}(QZ_t^T(\sum_{i=0}^{t-1}\bnabla_{\bx}^2\mL_i/t) Z_tQ^T)\rightarrow\lambda_{\min}((Z^\star)^T\nabla_{\bx}^2\mL^\star Z^\star),
\end{equation*}
which implies $\lambda_{\min}(Z_t^T(\sum_{i=0}^{t-1}\bnabla_{\bx}^2\mL_i/t) Z_t)\geq \gamma_{RH}$ for large enough $t$. This completes the~proof.

\subsection{Proof of Theorem \ref{thm:4}}\label{pf:thm:4}

We first improve the rate of $\I_{3,t}$.~Lemma \ref{lem:12} differs from Lemma \ref{lem:8} in the bound~of~$\bdelta^t$.~We~propose a more precise bound on $\bdelta^t$ compared to \eqref{pequ:19}. The new bound~relies~on~the~convergence rate of the Hessian $K_t$ in Lemma \ref{lem:9} and Assumption \ref{ass:5} for~applying~Corollary~\ref{cor:2}.

\begin{lemma}\label{lem:12}
Under the conditions of Theorem \ref{thm:4}, for any $\upsilon>0$,
\begin{equation*}
\I_{3,t} = o\big(\cbr{\varphi_t\log(1/\varphi_t)}^{2/3} \{\log(1/\varphi_t)\}^\nu\big) \quad \text{a.s.}
\end{equation*}
\end{lemma}

We note that
\begin{equation*}
|\sqrt{1/\baralpha_t} - \sqrt{1/\varphi_t}| = |\sqrt{\baralpha_t} - \sqrt{\varphi_t}|/\sqrt{\baralpha_t\varphi_t} = \abr{\baralpha_t - \varphi_t}/(\sqrt{\baralpha_t\varphi_t}(\sqrt{\baralpha_t}+\sqrt{\varphi_t}))\leq \chi_t/(4\beta_t^{1.5}).
\end{equation*}
Since $\chi<1.5\beta$, we know $\chi_t=o(\beta_t^{1.5})$.~By the almost sure convergence of $(\bx_t, \blambda_t)$,~we~only~need to show the normality of $1/\sqrt{\varphi_t}(\bx_t-\tx, \blambda_t-\tlambda)$. Let us choose $p\in(0, 1]$ such that
\begin{equation}\label{nequ:p}
\begin{aligned}
p\chi - 0.5p\varphi-2\varphi/3 <0 & \Longleftrightarrow p>\frac{2\varphi/3}{\chi-0.5\varphi}, \\
(1-\rho^\tau) + p(\chi-0.5\varphi)/\tvarphi > 0 & \Longleftrightarrow p<\frac{(1-\rho^\tau)\tvarphi}{0.5\varphi-\chi},
\end{aligned}
\end{equation}
which is guaranteed to exist due to the fact that
\begin{equation*}
0<\frac{2\varphi/3}{\chi - 0.5\varphi} < 1\wedge \frac{(1-\rho^\tau)\tvarphi}{0.5\varphi - \chi}.
\end{equation*}
We also choose $q$ as stated in the theorem, which guarantees that (by the proof of Theorem~\ref{thm:3}, $\varphi = \beta$ and $\tvarphi = \tbeta$)  
\begin{equation}\label{nequ:q}
q\chi - q\varphi - 0.5\varphi <0 \Leftrightarrow q>\frac{0.5\varphi}{\chi-\varphi} \quad\text{and}\quad (1-\rho^\tau) +q(\chi-\varphi)/\tvarphi >0\Leftrightarrow q<\frac{(1-\rho^\tau)\tvarphi}{\varphi-\chi}.
\end{equation}
With the above choices of $p$ and $q$, we know from \eqref{pequ:13}, Lemmas \ref{lem:7} and \ref{lem:12} that~for~any~\mbox{$\nu>0$}
\begin{equation*}
\I_{2,t} + \I_{3,t} = O(\chi_t^q/\varphi_t^q) + o\big(\cbr{\varphi_t\log(1/\varphi_t)}^{2/3} \cbr{\log(1/\varphi_t)}^\nu\big).
\end{equation*}
The first $O(\cdot)$ can be strengthened to $o(\cdot)$ when $q<1$.~Noting that \mbox{$1/\sqrt{\varphi_t}(\I_{2,t}+\I_{3,t}) = o(1)$}~a.s., the Slutsky's theorem together with Lemma \ref{lem:6} leads to the asymptotic normality. Furthermore, Lemma \ref{aux:lem:4}~with
\begin{equation*}
A_t = \frac{\sqrt{1/\varphi_t}\cdot \bw^T\I_{1,t}}{\sqrt{\bw^T\Xi^\star\bw}}, \;\; B_t = \frac{\sqrt{1/\varphi_t}\cdot\bw^T(\I_{2,t} +\I_{3,t})}{\sqrt{\bw^T\Xi^\star\bw}} + \frac{(\sqrt{1/\baralpha_t} - \sqrt{1/\varphi_t})\cdot\bw^T(\bx_t-\tx,\blambda_t-\tlambda)}{\sqrt{\bw^T\Xi^\star\bw}}
\end{equation*}
and $C_t = 0$ leads to the Berry-Esseen bound. This completes the proof.

\subsubsection{Proof of Lemma \ref{lem:12}}

We need the following lemma to establish the convergence rate of $K_t$. The conditions are the same as those for showing the convergence rate of $(\bx_t-\tx, \blambda_t-\tlambda)$, which are weaker than Theorem \ref{thm:4}. The proof is provided in Appendix \ref{pf:lem:9}.

\begin{lemma}\label{lem:9}
Under Assumptions \ref{ass:1}, \ref{ass:2}(\ref{equ:BM:a}, \ref{equ:BM:e}), \ref{ass:3} and suppose $\{\beta_t, \chi_t\}_t$ satisfy (\ref{cond:n:1}). Then, for any $\nu>0$ and any constants $p, q\in(0, 1]$ such that $(1-\rho^\tau) + p(\chi-0.5\beta)/\tbeta>0$~and $(1-\rho^\tau) + q(\chi-\beta)/\tbeta>0$, we have
\begin{equation*}
\|K_t-K^\star\| = o\big(\sqrt{\beta_t\{\log(1/\beta_t)\}^{1+\upsilon}}\; \big) + o\big(\chi_t^{p}/\beta_t^{0.5p}\sqrt{\{\log(1/\chi_t)\}^{1+\nu}}\;\big) + O(\chi_t^q/\beta_t^q) \;\; \text{a.s.}
\end{equation*}
Furthermore, if (\ref{equ:BM:a}) is strengthened to (\ref{equ:BM:b}), then
\begin{equation*}
\|K_t-K^\star\| = O\big(\sqrt{\beta_t\log(1/\beta_t)}\; \big) + O\big(\chi_t^{p}/\beta_t^{0.5p}\sqrt{\log(1/\chi_t)}\;\big) + O(\chi_t^q/\beta_t^q) + o(\sqrt{(\log t)^{1+\upsilon}/t}\;) \; \text{a.s.}
\end{equation*}
If $p<1$ (and/or $q<1$), the second (and/or third) $O(\cdot)$ in the above results can be strengthened to $o(\cdot)$.

\end{lemma}

Applying Lemma \ref{lem:9} with $p, q$ chosen to satisfy \eqref{nequ:p} and \eqref{nequ:q}, we know for any~$\nu>0$
\begin{equation*}
\|K_t-K^\star\| = O\big(\sqrt{\varphi_t\log(1/\varphi_t)}\; \big) + o(\sqrt{(\log t)^{1+\upsilon}/t}\;).
\end{equation*}
Combining the above result with \eqref{pequ:32}, Lemmas \ref{lem:4}, \ref{lem:6}, \ref{lem:7}, \ref{lem:8}, and Corollary \ref{cor:2},~we~have 
\begin{equation}\label{nequ:7}
\|\bdelta^t\| = O(\varphi_t\log(1/\varphi_t)) + o( \sqrt{\varphi_t\log(1/\varphi_t)} \cdot \sqrt{(\log t)^{1+\upsilon}/t }\;).
\end{equation}
We plug the above bound into the recursion \eqref{pequ:18}. We note that 
\begin{align*}
& \lim_{t\rightarrow\infty} t\rbr{1 - \frac{\varphi_{t-1}\log(1/\varphi_{t-1})}{\varphi_t\log(1/\varphi_t)}} \stackrel{\eqref{pequ:12}}{=} \varphi,\\
& \lim_{t\rightarrow\infty} t\rbr{1 - \frac{\sqrt{\varphi_{t-1}\log(1/\varphi_{t-1})}\sqrt{(\log(t-1))^{1+\upsilon}/(t-1) } }{\sqrt{\varphi_t\log(1/\varphi_t)} \sqrt{(\log t)^{1+\upsilon}/t} } } = \frac{\varphi}{2}-\frac{1}{2}.
\end{align*}
In the following proof, we consider $\varphi$ such that $0.5\varphi - 0.5\geq \varphi$. Otherwise, the second term~in \eqref{nequ:7} is absorbed into the first term. Applying Lemma \ref{aux:lem:2} and noting that
\begin{equation*}
1.5(1-\rho^\tau) + (0.5\varphi - 0.5)/\tvarphi \geq 1.5(1-\rho^\tau) + \varphi/\tvarphi>0,
\end{equation*}
we know
\begin{align*}
\I_{3,t} & = o\big(\cbr{\varphi_t\log(1/\varphi_t)}^{2/3}\big) + o\big(\{\sqrt{\varphi_t\log(1/\varphi_t)} \cdot \sqrt{(\log t)^{1+\upsilon}/t }\}^{2/3}\big)\\
& = o\big(\cbr{\varphi_t\log(1/\varphi_t)}^{2/3} \{\log(1/\varphi_t)\}^\nu\big),
\end{align*}
where the second equality is due to $\sqrt{(\log t)^{1+\upsilon}/t } = O\big(\sqrt{\varphi_t\{\log(1/\varphi_t)\}^{1+\nu}}\big)$ as implied~by~the fact that~$t\varphi_t\rightarrow \tvarphi\in(0, \infty]$. This completes the proof.

\subsubsection{Proof of Lemma \ref{lem:9}}\label{pf:lem:9}

We note from the proof of Theorem \ref{thm:3} that $\varphi = \beta$ and $\tvarphi = \tbeta$. By Lemmas \ref{lem:4}, \ref{lem:6}, \ref{lem:7},~\ref{lem:8}, for any $\nu>0$, if \eqref{equ:BM:a} holds, then we have {\small
\begin{equation}\label{nequ:3}
\|(\bx_t-\tx, \blambda_t-\tlambda)\| = o\big(\sqrt{\varphi_t\{\log(1/\varphi_t)\}^{1+\nu}}\big) + o\big(\chi_t^{p}/\varphi_t^{0.5p}\sqrt{\{\log(1/\chi_t)\}^{1+\nu}}\;\big) + O(\chi_t^q/\varphi_t^q).
\end{equation}}
\hskip-3pt Here, $p, q\in(0, 1]$ are any constants such that $(1-\rho^\tau) +p(\chi-0.5\varphi)/\tvarphi>0$ and $(1-\rho^\tau) + q(\chi-\varphi)/\tvarphi>0$. Furthermore, if $q<1$ the $O(\cdot)$ in the third term~can be strengthened to~$o(\cdot)$. In the following proof, we only consider $p, q$ such that
\begin{equation}\label{nequ:5}
p(\chi-0.5\varphi) \geq 0.5\varphi\quad \quad\text{ and }\quad\quad q(\chi-\varphi) \geq 0.5\varphi.
\end{equation}
Otherwise, by \eqref{pequ:12}, \eqref{pequ:13}, Lemmas \ref{aux:lem:5} and \ref{aux:lem:1}, we know the second (and/or third)~term~in \eqref{nequ:3} can be absorbed into the first term. We now apply \eqref{Lip:mL}, combine \eqref{pequ:25} and~\eqref{pequ:26},~and have for any $\upsilon>0$ and large enough $t$ that {\small
\begin{align}\label{pequ:15}
& \|K_t-K^\star\| \leq \frac{\Upsilon_L}{t}\sum_{i=0}^{t-1}\nbr{\begin{pmatrix}
\bx_i - \bx^\star\\
\blambda_i - \blambda^\star
\end{pmatrix}} + \Upsilon_L\|\bx_t-\bx^\star\| + o\rbr{\sqrt{\frac{(\log t)^{1+\upsilon}}{t}}} \nonumber\\
& = \frac{\Upsilon_L}{t}\nbr{\begin{pmatrix}
\bx_0 - \bx^\star\\
\blambda_0 - \blambda^\star
\end{pmatrix}} + \Upsilon_L\sum_{i=1}^{t-1}\prod_{j=i+1}^{t}\rbr{1 - \frac{1}{j}}\frac{1}{i}\nbr{\begin{pmatrix}
\bx_i - \bx^\star\\
\blambda_i - \blambda^\star
\end{pmatrix}} + \Upsilon_L\|\bx_t-\bx^\star\| + o\rbr{\sqrt{\frac{(\log t)^{1+\upsilon}}{t}}} \nonumber\\
& = \Upsilon_L\sum_{i=1}^{t-1}\prod_{j=i+1}^{t}\rbr{1 - \frac{1}{j}}\frac{1}{i}\nbr{\begin{pmatrix}
\bx_i - \bx^\star\\
\blambda_i - \blambda^\star
\end{pmatrix}} + \Upsilon_L \|\bx_t-\bx^\star\|  + o\rbr{\sqrt{\frac{(\log t)^{1+\upsilon}}{t}}}.
\end{align}}
\hskip-3pt We claim that $\varphi =\beta > -2$. Otherwise, $\varphi +1.5 \leq -0.5<0$. We apply Lemma \ref{aux:lem:5} and have
\begin{equation*}
\lim_{t\rightarrow\infty}t\rbr{1 - \frac{\varphi_{t-1}(t-1)^{1.5}}{\varphi_t t^{1.5}}} = \lim_{t\rightarrow\infty} t\rbr{1 - \frac{\varphi_{t-1}}{\varphi_t} + \frac{\varphi_{t-1}}{\varphi_t} \rbr{1 - \frac{(t-1)^{1.5}}{t^{1.5}}}} \stackrel{\eqref{cond:varphi}}{=}\varphi + 1.5<0.
\end{equation*}
Then, Lemma \ref{aux:lem:1} suggests $\varphi_t t^{1.5}\rightarrow 0$, which cannot hold under \eqref{cond:n:1}. Thus, $\varphi>-2$. Using \eqref{pequ:12}, \eqref{pequ:13}, \eqref{nequ:5}, and Lemma \ref{aux:lem:2}, and noting that \mbox{$1 + p(\chi-0.5\varphi) \geq 1+0.5\varphi>0$}~and~$1 + q(\chi-\varphi)\geq 1+0.5\varphi>0$, we obtain
\begin{multline}\label{pequ:16}
\sum_{i=1}^{t-1}\prod_{j=i+1}^{t}\rbr{1 - \frac{1}{j}}\frac{1}{i}\nbr{\begin{pmatrix}
\bx_i - \bx^\star\\
\blambda_i - \blambda^\star
\end{pmatrix}}  \\
= o\big(\sqrt{\varphi_t\{\log(1/\varphi_t)\}^{1+\nu}}\big) + o\big(\chi_t^{p}/\varphi_t^{0.5p}\sqrt{\{\log(1/\chi_t)\}^{1+\nu}}\;\big) + O(\chi_t^q/\varphi_t^q).
\end{multline}
Following the same derivation, we know that if \eqref{equ:BM:b} holds, then
\begin{equation}\label{nequ:4}
\|(\bx_t-\tx, \blambda_t-\tlambda)\| = O\big(\sqrt{\varphi_t \log(1/\varphi_t)}\big) + O\big(\chi_t^{p}/\varphi_t^{0.5p}\sqrt{\log(1/\chi_t)}\;\big) + O(\chi_t^q/\varphi_t^q)
\end{equation}
and
\begin{multline}\label{nequ:6}
\sum_{i=1}^{t-1}\prod_{j=i+1}^{t}\rbr{1 - \frac{1}{j}}\frac{1}{i}\nbr{\begin{pmatrix}
\bx_i - \bx^\star\\
\blambda_i - \blambda^\star
\end{pmatrix}}  \\
= O\big(\sqrt{\varphi_t \log(1/\varphi_t)}\big) + O\big(\chi_t^{p}/\varphi_t^{0.5p}\sqrt{\log(1/\chi_t)}\;\big) + O(\chi_t^q/\varphi_t^q).
\end{multline}
Combining \eqref{pequ:15}, \eqref{pequ:16}, \eqref{nequ:6} together, and noting that $\beta_t\leq \varphi_t\leq 2\beta_t$ and $o(\sqrt{(\log t)^{1+\upsilon}/t })$ can be absorbed into $ o\big(\sqrt{\varphi_t\{\log(1/\varphi_t)\}^{1+\nu}}\big)$ under \eqref{equ:BM:a} (as~implied by the fact that~$t\varphi_t\rightarrow \tvarphi\in(0, \infty]$), we complete the proof.

\subsection{Proof of Corollary \ref{cor:3}}

The result in (a) is immediate by plugging $\beta = -1$ and $\tbeta = 1$ into \eqref{equ:cov:exact}. For (b),~we~plug~$\beta = -1$ and $\tbeta = 1$ into \eqref{equ:Lya:equ}, and know that $\tXi$ solves the equation
\begin{equation*}
(0.5I + C^\star)\tXi + \tXi(0.5I + C^\star) = \mE[(I+\tC^\star)\tOmega(I+\tC^\star)^T].
\end{equation*}
Thus, we have
\begin{multline*}
(0.5I + C^\star)(\tXi - \tOmega) + (\tXi - \tOmega)(0.5I + C^\star) \\
= \mE[(I+\tC^\star)\tOmega(I+\tC^\star)^T] - \tOmega - C^\star\tOmega - \tOmega C^\star = \mE[\tC^\star\tOmega(\tC^\star)^T]\succeq \0.
\end{multline*}
Since $\|C^\star\|\leq \rho^\tau<0.5$, the basic Lyapunov theorem (cf. \cite[Theorem 4.6]{Khalil2002Khalil})~suggests that $\tXi \succeq \tOmega$. Furthermore, with the notation in \eqref{equ:Xi:n}, we know
\begin{equation}\label{nequ:8}
\tXi - \tOmega = U\big(\Theta\circ U^T\mE[\tC^\star\tOmega(\tC^\star)^T]\; U \big) U^T \quad \text{ with }\quad  [\Theta]_{k,l} = 1/(\sigma_k+\sigma_l-1).
\end{equation}
The matrix $\Theta$ is positive semidefinite since for any vector $\bomega$,
\begin{align*}
\bomega^T\Theta\bomega & = \sum_{k,l=1}^{d+m}\frac{\bomega_k\bomega_l}{\sigma_k+\sigma_l-1} = \sum_{k,l=1}^{d+m} \bomega_k\bomega_l \int_0^\infty \exp(-s(\sigma_k+\sigma_l-1))ds \quad (\text{since } \sigma_k+\sigma_l-1>0)\\
& = \int_{0}^{\infty}\rbr{ \sum_{k=1}^{d+m}\bomega_k\exp(-s(\sigma_k-0.5))}^2 ds\geq 0.
\end{align*}
By \cite[7.5.P24]{Horn1985Matrix}, we have
\begin{multline*}
\|\tXi - \tOmega\| \stackrel{\eqref{nequ:8}}{=} \|\Theta\circ U^T\mE[\tC^\star\tOmega(\tC^\star)^T]\; U\| \\
\leq \max_{k}[\Theta]_{k,k}\|\mE[\tC^\star\tOmega(\tC^\star)^T]\| \leq \frac{1}{1-2\rho^\tau}\|\mE[\tC^\star\tOmega(\tC^\star)^T]\| \leq 3\|\mE[\tC^\star\tOmega(\tC^\star)^T]\|,
\end{multline*}
where the second last inequality is due to $\sigma_k\geq1-\rho^\tau$ and the last inequality is due to~$\rho^\tau<1/3$. Furthermore, we have
\begin{align}\label{nequ:9}
\0 & \preceq \mE[\tC^\star\tOmega(\tC^\star)^T]\preceq \|\tOmega\|\cdot\mE[\tC^\star(\tC^\star)^T] \nonumber\\
& = \|\tOmega\|\cdot\mE\sbr{\cbr{\prod_{j=1}^\tau(I - K^\star S_j(S_j^T(K^\star)S)^\dagger S_j^TK^\star)} \cbr{\prod_{j=1}^\tau(I - K^\star S_j(S_j^T(K^\star)S)^\dagger S_j^TK^\star)}^T } \nonumber\\
& =   \|\tOmega\|\cdot\mE\bigg[\cbr{\prod_{j=2}^\tau(I - K^\star S_j(S_j^T(K^\star)S_j)^\dagger S_j^TK^\star)}\mE[(I - K^\star S_1(S_1^T(K^\star)S_1)^\dagger S_j^TK^\star)\mid S_{2:\tau}] \nonumber\\ 
&\quad\quad \cbr{\prod_{j=2}^\tau(I - K^\star S_j(S_j^T(K^\star)S_j)^\dagger S_j^TK^\star)}^T\bigg] \nonumber\\
& \preceq \rho\|\tOmega\|\cdot\mE\sbr{\cbr{\prod_{j=2}^\tau(I - K^\star S_j(S_j^T(K^\star)S)^\dagger S_j^TK^\star)} \cbr{\prod_{j=2}^\tau(I - K^\star S_j(S_j^T(K^\star)S)^\dagger S_j^TK^\star)}^T } \nonumber\\
& \preceq \rho^\tau\|\tOmega\|\cdot I,
\end{align}
where the second last inequality is from Assumption \ref{ass:3} and $K_t\rightarrow K^\star$; and the last inequality applies the same reason for sketch matrices $S_{2:\tau}$. Combining the above two displays completes the proof.

\subsection{Proof of Theorem \ref{thm:5}}\label{pf:thm:5}

We have
\begin{multline}\label{pequ:20}
\nbr{\Xi_t - \Xi^\star} \leq \|\Xi^\star - \mE[(I+\tC^\star)\tOmega(I+\tC^\star)^T]/(2+\beta/\tbeta)\| \\
+ \|\mE[(I+\tC^\star)\tOmega(I+\tC^\star)^T] - \tOmega\|/(2+\beta/\tbeta) + \nbr{\tOmega - \Omega_t}/(2+\beta/\tbeta).
\end{multline}
For the first term in \eqref{pequ:20}, we have \citep[7.7.P27]{Horn1985Matrix}
\begin{align*}
\|\Xi^\star & - \mE[(I+\tC^\star)\tOmega(I+\tC^\star)^T]/(2+\beta/\tbeta)\| \\
& \; \stackrel{\mathclap{\eqref{equ:Xi:n}}}{=} \; \|(\Theta - \b1\b1^T/(2+\beta/\tbeta))\circ U^T\mE[(I+\tC^\star)\tOmega(I+\tC^\star)^T]U \| \nonumber\\
& \leq \|\Theta - \b1\b1^T/(2+\beta/\tbeta)\| \cdot \|\mE[(I+\tC^\star)\tOmega(I+\tC^\star)^T]\|\quad (\|A\circ B\|\leq \|A\|\cdot \|B\|) \nonumber\\
& \leq 4\|\Theta - \b1\b1^T/(2+\beta/\tbeta)\|\cdot \|\tOmega\| \quad (\|\tC^\star\|\leq 1),
\end{align*}
and for any $1\leq k, l\leq d+m$,
\begin{multline*}
|\Theta_{k,l} - 1/(2+\beta/\tbeta)| = |1/(\sigma_k+\sigma_l+\beta/\tbeta) - 1/(2+\beta/\tbeta)| = \frac{\abr{2 - \sigma_k-\sigma_l}}{(\sigma_k+\sigma_l+\beta/\tbeta)(2+\beta/\tbeta)} \\ 
\leq \frac{2\rho^\tau}{(2-2\rho^\tau + \beta/\tbeta)(2+\beta/\tbeta)} \stackrel{\mathclap{\eqref{cond:n:1}}}{\leq} \frac{2\rho^\tau}{(2- 2(1+\beta/(1.5\tbeta)) + \beta/\tbeta)(2+\beta/\tbeta)} = \frac{6\rho^\tau}{-\beta/\tbeta(2+\beta/\tbeta)}.
\end{multline*}
Therefore, the above two displays lead to
\begin{equation}\label{pequ:21}
\|\Xi^\star - \mE[(I+\tC^\star)\tOmega(I+\tC^\star)^T]/(2+\beta/\tbeta)\| = O(\rho^\tau).
\end{equation}
For the second term in \eqref{pequ:20}, we have
\begin{align}\label{pequ:24}
\|\mE[(I+\tC^\star)\tOmega(I+\tC^\star)^T] - \tOmega\| & \leq \|C^\star\tOmega\| + \|\tOmega C^\star\| + \|\mE[\tC^\star\tOmega(\tC^\star)^T]\| \quad (\mE[\tC^\star] = C^\star) \nonumber\\
& \leq 2\|\tOmega\|\rho^\tau + \|\mE[\tC^\star\tOmega(\tC^\star)^T]\|  \stackrel{\eqref{nequ:9}}{=} O(\rho^\tau).
\end{align}
For the third term in \eqref{pequ:20}, we have {\small
\begin{multline*}
\|\Omega_t - \tOmega\| 
\stackrel{\eqref{equ:Omega}}{=} O(\|K_t - K^\star\|) \\+  O\rbr{\nbr{\frac{1}{t}\sum_{i=0}^{t-1}\barg_i\barg_i^T - \rbr{\frac{1}{t}\sum_{i=0}^{t-1}\barg_i}\rbr{\frac{1}{t}\sum_{i=0}^{t-1}\barg_i}^T - \mE[\nabla f(\bx^\star;\xi)\nabla^Tf(\bx^\star; \xi)] + \nabla f(\bx^\star)\nabla^T f(\bx^\star) }}.
\end{multline*}}
\hskip-3pt Furthermore, we have
\begin{align*}
& \nbr{\frac{1}{t}\sum_{i=0}^{t-1}\barg_i\barg_i^T - \rbr{\frac{1}{t}\sum_{i=0}^{t-1}\barg_i}\rbr{\frac{1}{t}\sum_{i=0}^{t-1}\barg_i}^T - \mE[\nabla f(\bx^\star;\xi)\nabla^Tf(\bx^\star; \xi)] + \nabla f(\bx^\star)\nabla^T f(\bx^\star) }\\
& \leq \nbr{\frac{1}{t}\sum_{i=0}^{t-1}\barg_i\barg_i^T - \mE[\nabla f(\bx^\star;\xi)\nabla^Tf(\bx^\star; \xi)]} + \nbr{\rbr{\frac{1}{t}\sum_{i=0}^{t-1}\barg_i}\rbr{\frac{1}{t}\sum_{i=0}^{t-1}\barg_i}^T - \nabla f(\bx^\star)\nabla^T f(\bx^\star)}.
\end{align*}
We take the first term as an example, while the second term has the same guarantee~following the same derivations. We note that
\begin{multline*}
\nbr{\frac{1}{t}\sum_{i=0}^{t-1}\barg_i\barg_i^T - \mE[\nabla f(\bx^\star;\xi)\nabla^Tf(\bx^\star; \xi)]} \leq \nbr{\frac{1}{t}\sum_{i=0}^{t-1}(\barg_i\barg_i^T) - \mE[\barg_i\barg_i^T\mid \mF_{i-1}]} \\
+\nbr{\frac{1}{t}\sum_{i=0}^{t-1}\mE[\barg_i\barg_i^T\mid \mF_{i-1}]- \mE[\nabla f(\bx^\star;\xi)\nabla^Tf(\bx^\star; \xi)]}.
\end{multline*}
By \eqref{equ:BM:c}, we know the first term on the right hand side is a square integrable martingale.~The strong law of large number \cite[Theorem 1.3.15]{Duflo1997Random} suggests that for any $\nu>0$
\begin{equation*}
\nbr{\frac{1}{t}\sum_{i=0}^{t-1}(\barg_i\barg_i^T) - \mE[\barg_i\barg_i^T\mid \mF_{i-1}]} = o\big(\sqrt{(\log t)^{1+\upsilon}/t} \big).
\end{equation*}
By \eqref{Lip:map}, \eqref{nequ:6}, and the choices of $p, q$ in \eqref{nequ:p} and \eqref{nequ:q}, the second term on the~right~hand side can be bounded~by
\begin{equation}\label{pequ:35}
\nbr{\frac{1}{t}\sum_{i=0}^{t-1}\mE[\barg_i\barg_i^T\mid \mF_{i-1}]- \mE[\nabla f(\bx^\star;\xi)\nabla^Tf(\bx^\star; \xi)]} = O\rbr{\sqrt{\varphi_t\log(1/\varphi_t)}}.
\end{equation}
Combining the above five displays with Lemma \ref{lem:9}, we have
\begin{equation}\label{pequ:34}
\|\Omega_t - \tOmega\|= O\rbr{\sqrt{\varphi_t\log(1/\varphi_t)}} + o\big(\sqrt{ (\log t)^{1+\upsilon}/t} \big) .
\end{equation}
Combining \eqref{pequ:20}, \eqref{pequ:21}, \eqref{pequ:24}, and \eqref{pequ:34}, we complete the first part of the proof. For~the Berry-Esseen inequality, we simply note that
\begin{equation*}
\frac{\bw^T(\bx_t - \bx^\star,\blambda_t-\blambda^\star)}{\sqrt{\bw^T\Xi_t\bw}} = \frac{\bw^T(\bx_t - \bx^\star,\blambda_t-\blambda^\star)}{\sqrt{\bw^T\Xi^\star\bw} \cdot \sqrt{1 + \frac{\bw^T\Xi_t\bw - \bw^T\Xi^\star\bw}{\bw^T\Xi^\star\bw}}}.
\end{equation*}
Thus, we apply Theorem \ref{thm:4} and Lemma \ref{aux:lem:4}, and complete the proof.

\section{Additional Experimental Results}\label{sec:exp:more}

In this section, we provide more implementation details and show additional results. We follow the introduction in Section \ref{sec:5}, and implement the method on eight problems in~CUTEst test set and on linearly/nonlinearly constrained regression problems. For both implementation,~we~run $10^5$ iterations, set $\beta_t =1/t^{0.501}$, $\chi_t=\beta_t^2$, and $\baralpha_t\sim\text{Uniform}([\beta_t, \eta_t])$ with~\mbox{$\eta_t=\beta_t+\chi_t$}.~Regarding the Hessian regularization $\Delta_t$, we let ($\lambda_{\min}(\cdot)$ denotes the least eigenvalue)
\begin{equation*}
\Delta_t \coloneqq (-\lambda_{\min}(Z_t^T\sum_{i=0}^{t-1}\bnabla_{\bx}^2\mL_i Z_t)/t + 0.1)\cdot I\quad \text{whenever}\quad \lambda_{\min}(Z_t^T\sum_{i=0}^{t-1}\bnabla_{\bx}^2\mL_i Z_t)<0.
\end{equation*}
Here, $Z_t\in \mR^{d\times (d-m)}$ has orthonormal columns that span the space $\{\bx\in \mR^d: G_t\bx = \0\}$,~which is obtained from the QR decomposition.

\subsection{CUTEst problems}\label{sec:exp:more:1}

In this section, we testify the convergence rate in Theorem \ref{thm:3}. In particular, we randomly~pick one run across 200 runs, and show the~convergence plots of the KKT residual $\|\nabla\mL_t\|$, the~mean absolute error $\|(\bx_t-\bx^\star, \blambda_t-\blambda^\star)\|$, and~the Hessian error $\|K_t-K^\star\|$. By Theorem \ref{thm:3},~Lemma \ref{lem:9}, and the Lipschitz continuity of~the~Hessian, the theoretical convergence rate for these three quantities is $O(\sqrt{\beta_t\log(1/\beta_t)})$.

The convergence plots are shown in Figures \ref{fig:1} and \ref{fig:2}. We use six problems for illustration.
From the figures, we observe that~the method converges faster for small sampling variance~$\sigma^2$ and converges slower for large $\sigma^2$. Specifically, our theoretical convergence rate~precisely~characterizes asymptotic behavior of the method, and $\sigma^2$ only affects the~rate~as a constant~factor.\;

\begin{figure}[!htp]
\centering     
\subfigure[KKT residual]{\label{C11}\includegraphics[width=0.32\textwidth]{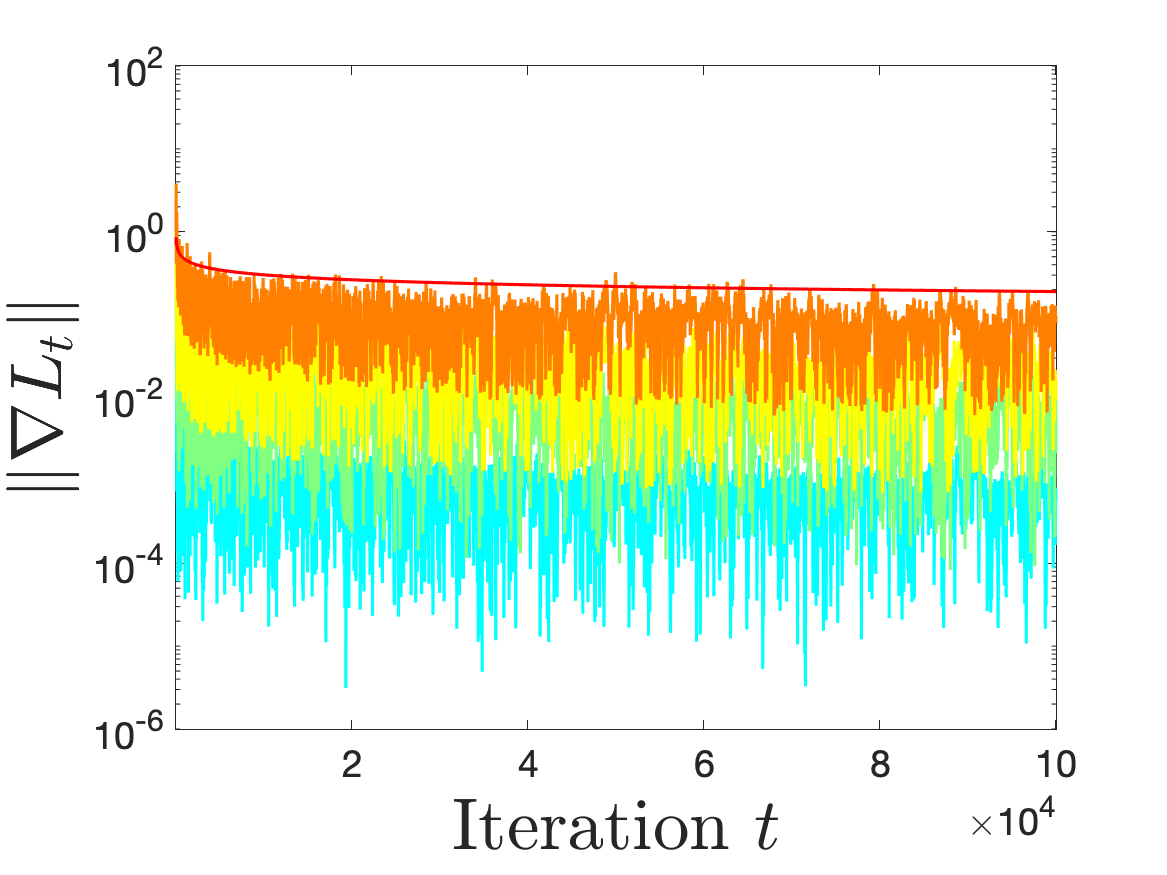}}
\subfigure[Iteration error]{\label{C12}\includegraphics[width=0.32\textwidth]{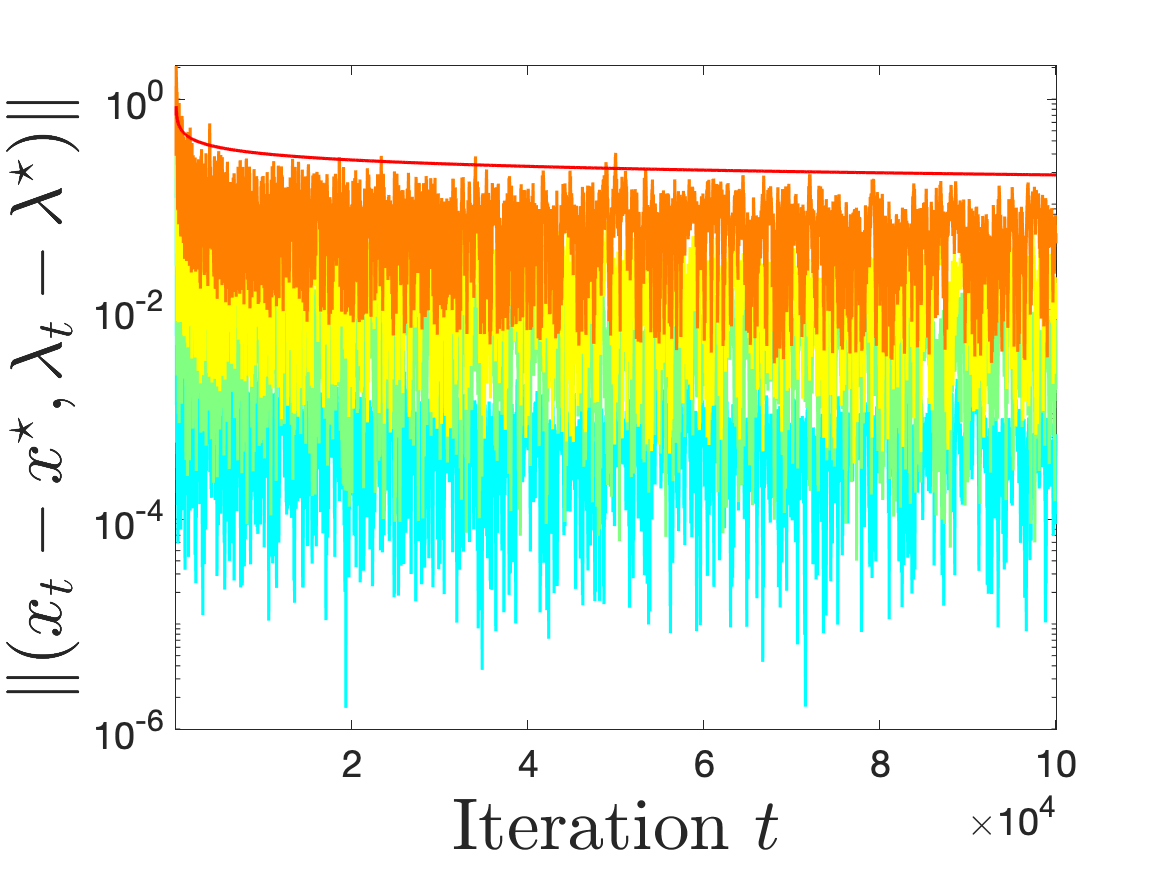}}
\subfigure[Hessian error]{\label{C13}\includegraphics[width=0.32\textwidth]{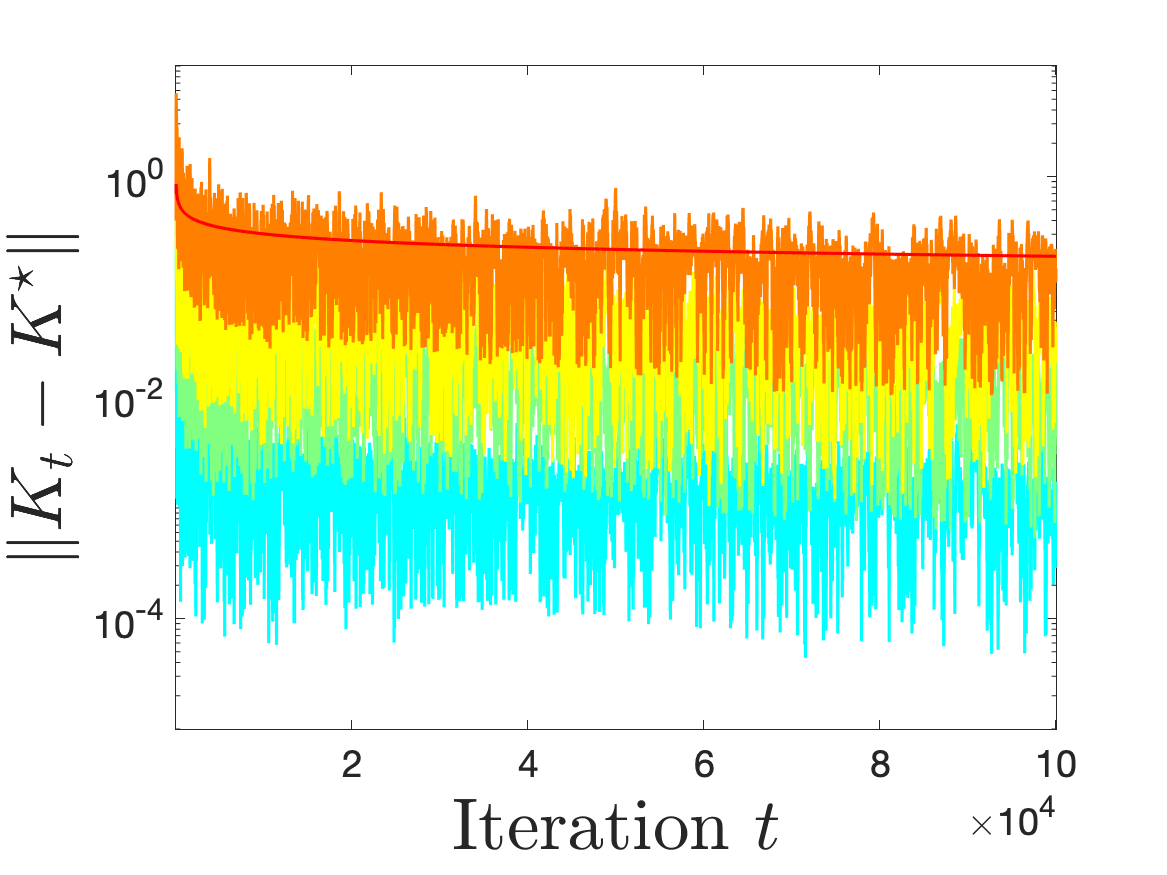}}
\vskip5pt
\centering{Problem \texttt{MARATOS}}
	
\subfigure[KKT residual]{\label{C21}\includegraphics[width=0.32\textwidth]{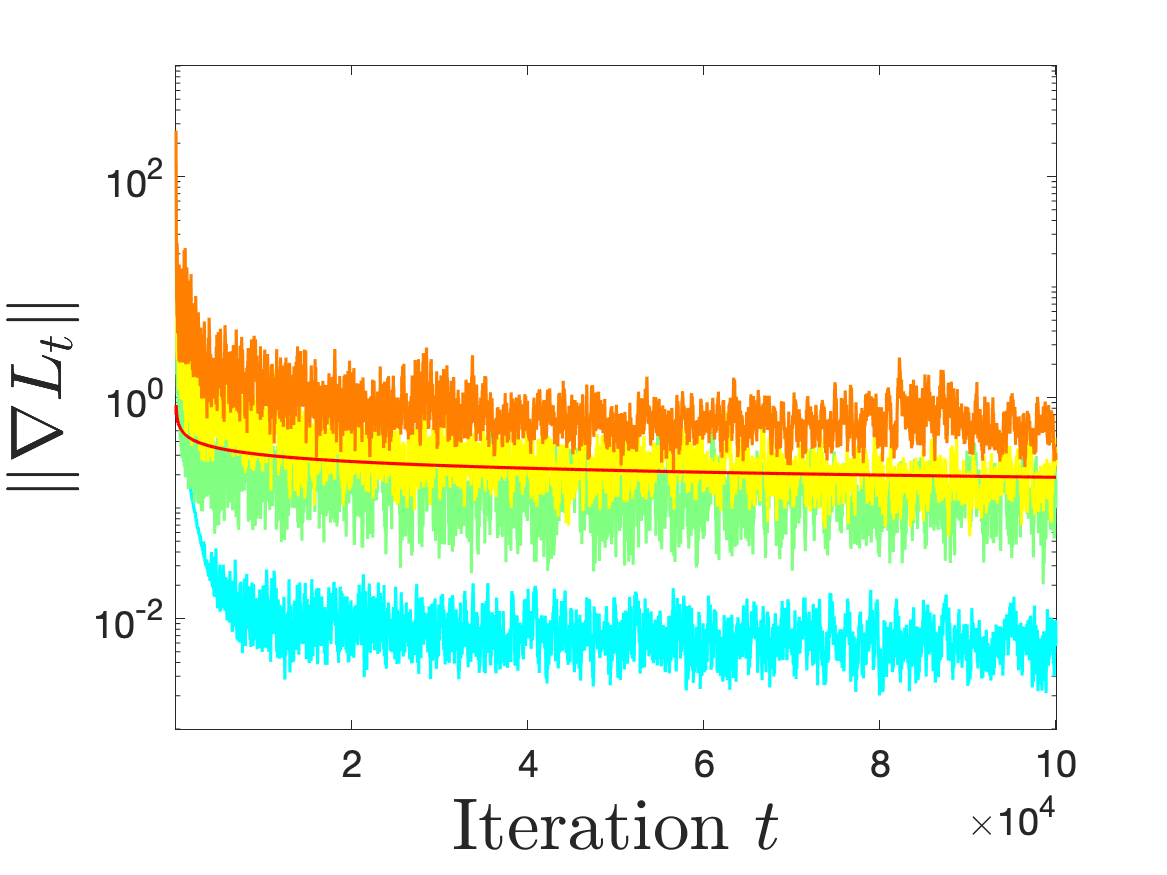}}
\subfigure[Iteration error]{\label{C22}\includegraphics[width=0.32\textwidth]{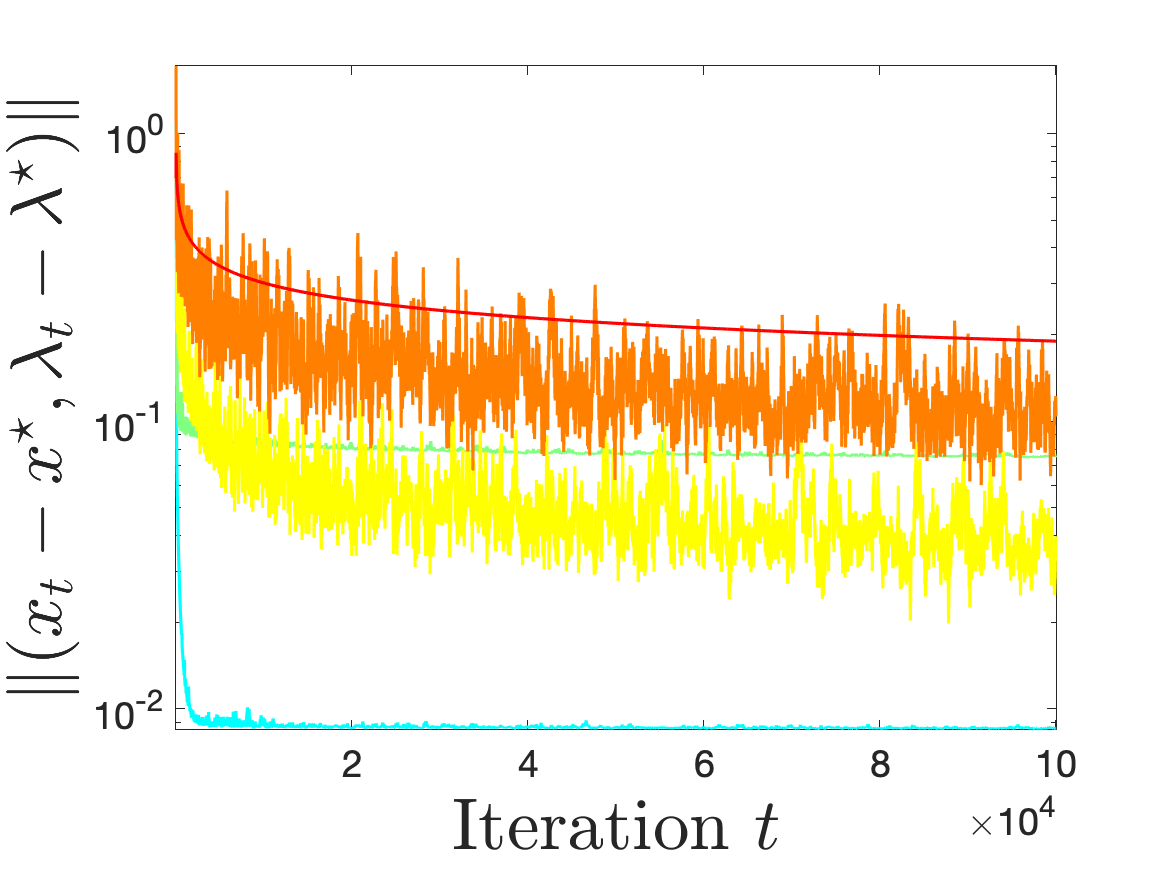}}
\subfigure[Hessian error]{\label{C23}\includegraphics[width=0.32\textwidth]{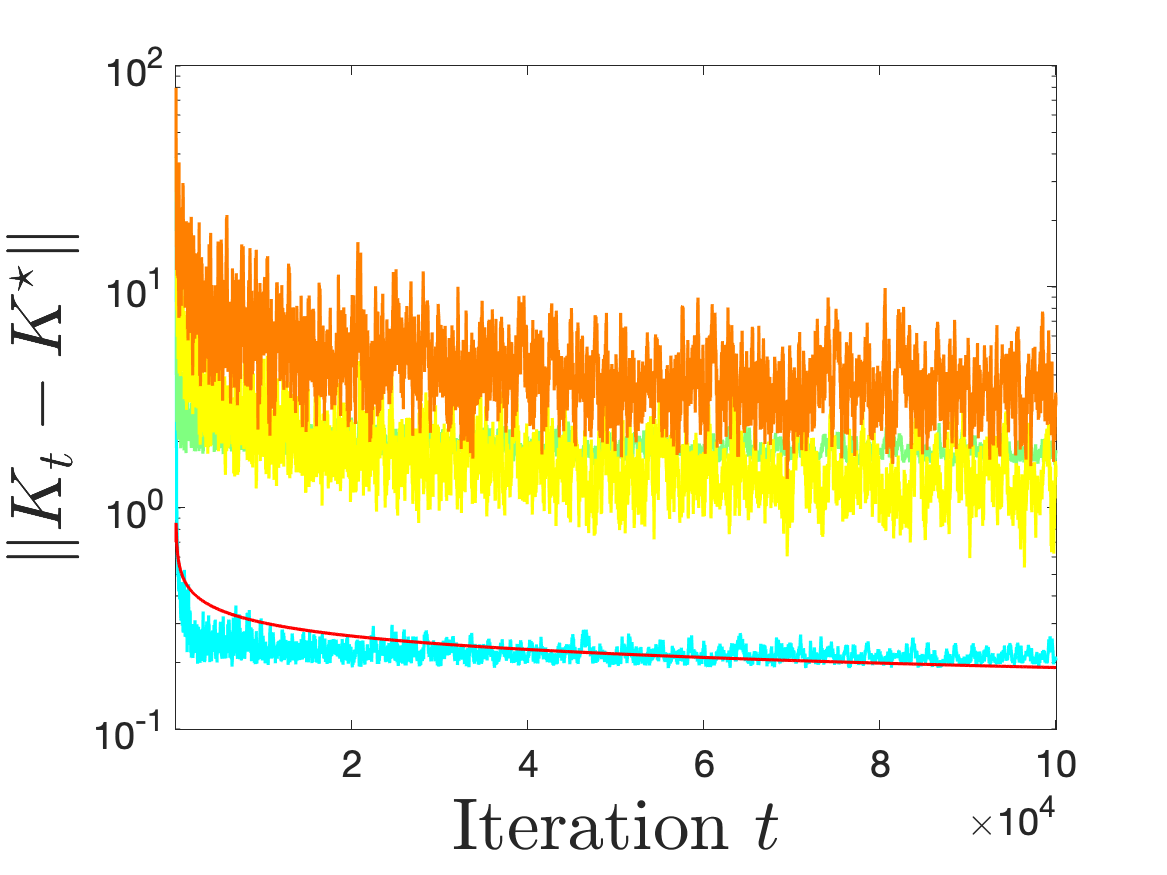}}
\vskip5pt
\centering{Problem \texttt{ORTHREGB}}
	
\subfigure[KKT residual]{\label{C31}\includegraphics[width=0.32\textwidth]{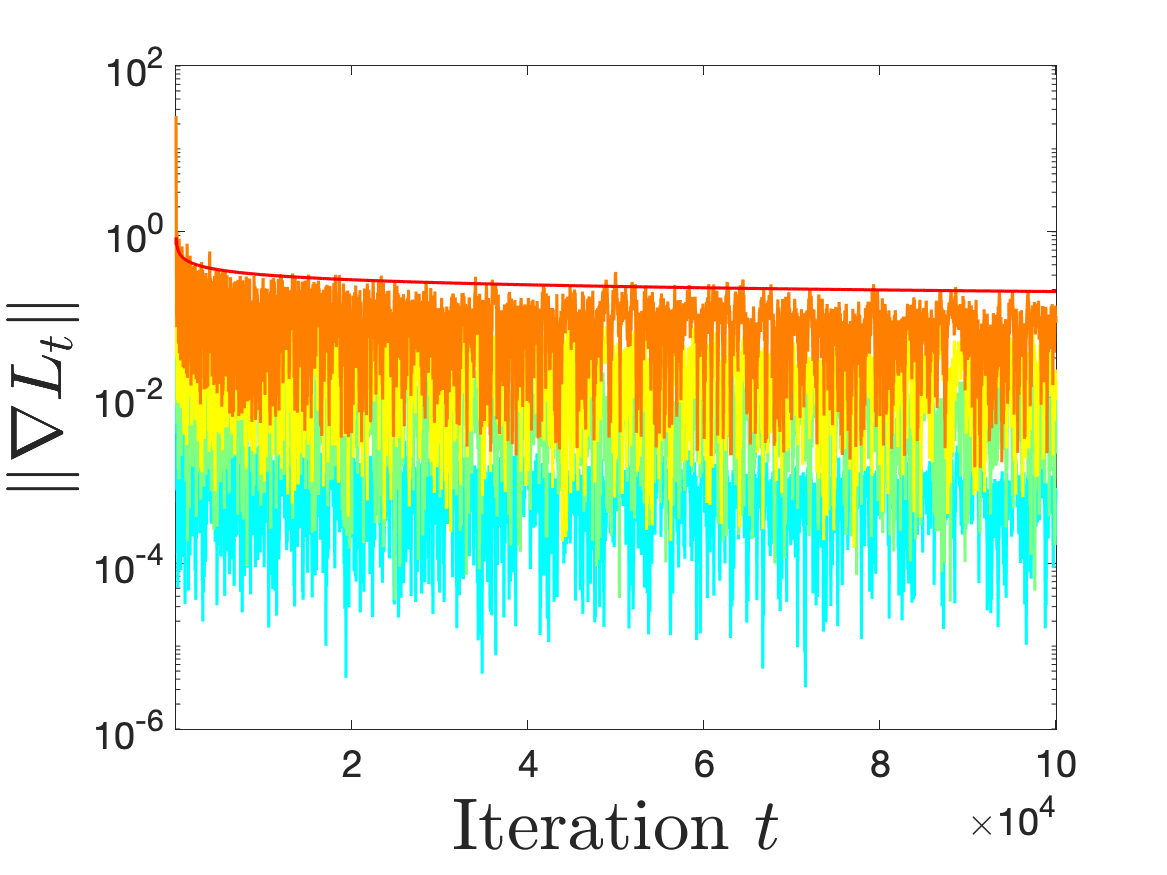}}
\subfigure[Iteration error]{\label{C32}\includegraphics[width=0.32\textwidth]{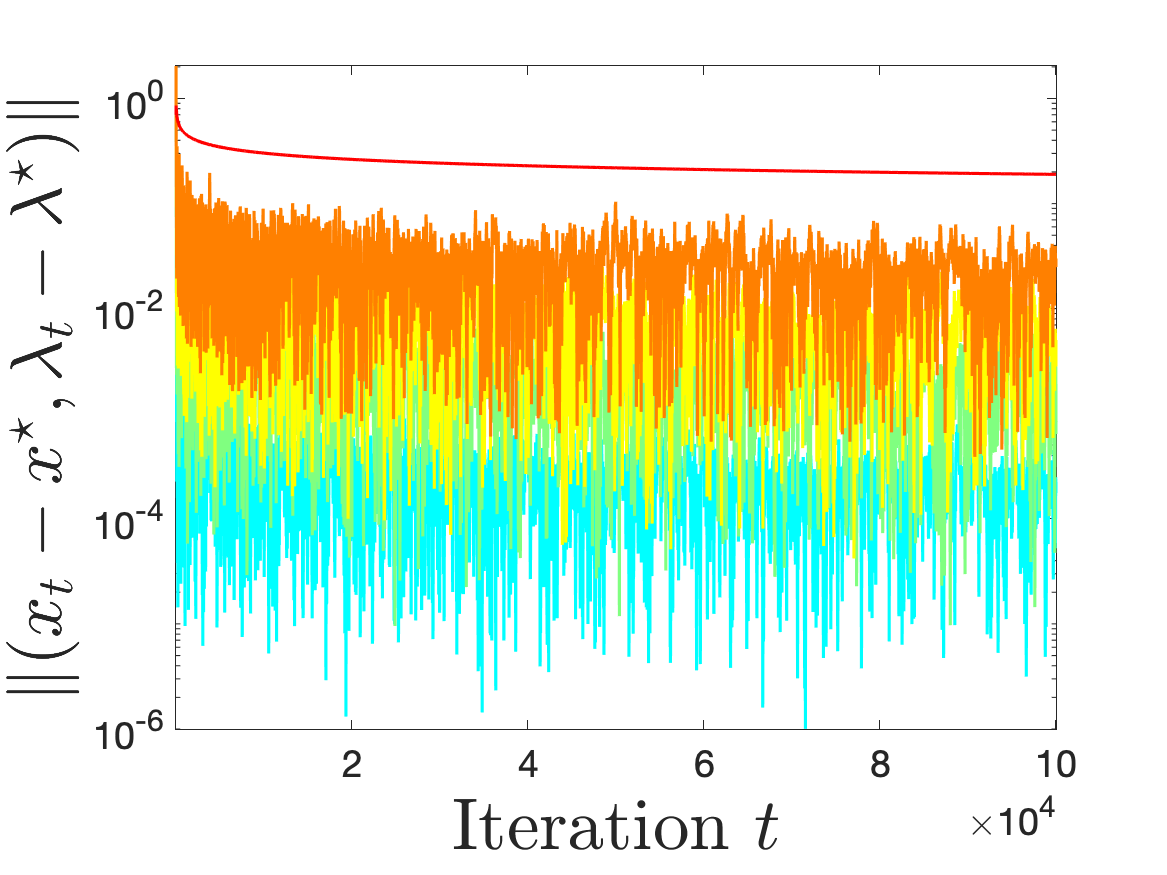}}
\subfigure[Hessian error]{\label{C33}\includegraphics[width=0.32\textwidth]{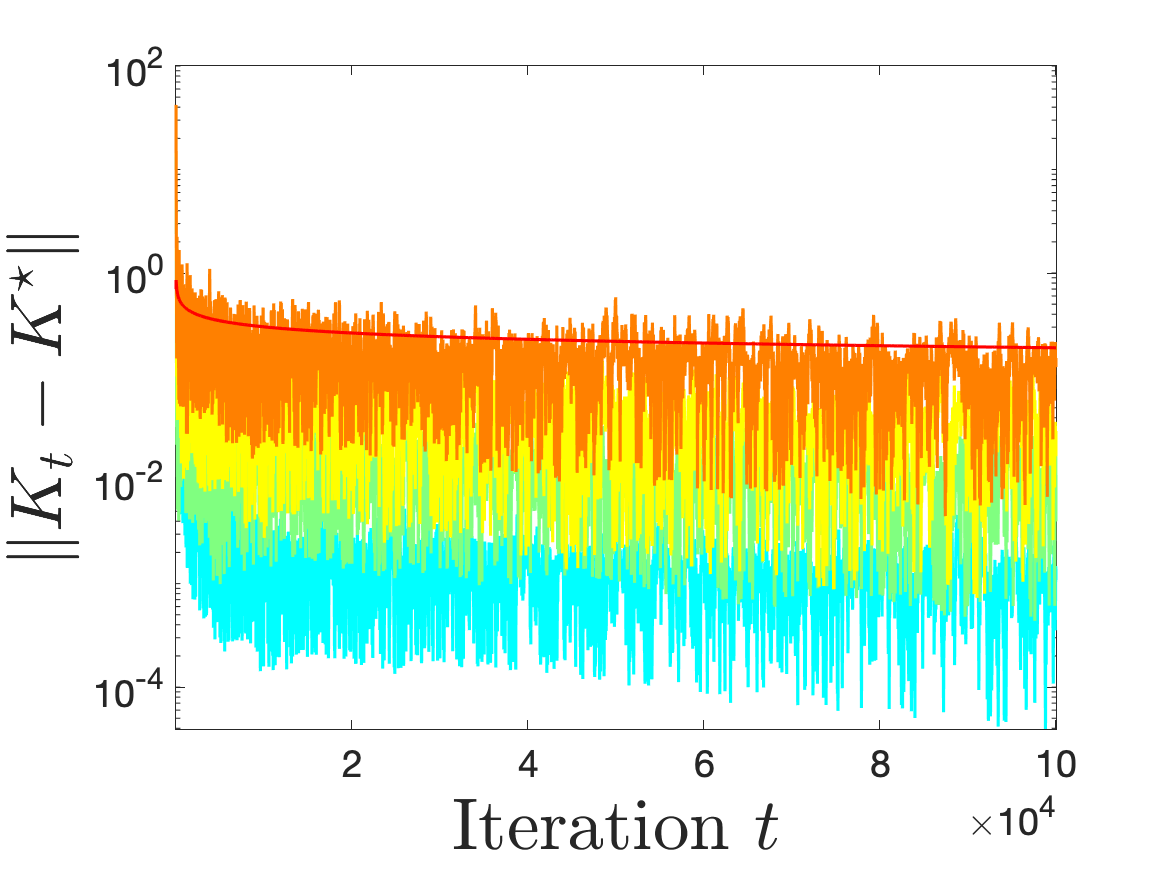}}
\vskip5pt
\centering{Problem \texttt{HS7}}
\vskip10pt
\includegraphics[width=0.6\textwidth]{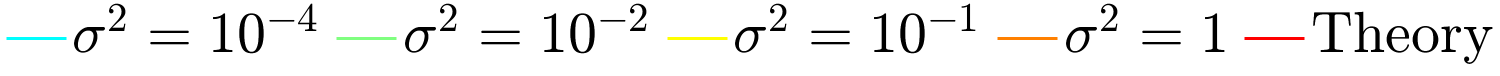}
\caption{\hskip-2pt\textit{Convergence plots of CUTEst problems.~Each row corresponds to one problem~and~has three~\mbox{figures}~in~the~$\log$~scale.~From~the~left~to~the~right,~they~\mbox{correspond}~to~$\|\nabla\mL_t\|$~v.s.~$t$,~$\|(\bx_t-\bx^\star, \blambda_t-\blambda^\star)\|$ v.s. $t$, and $\|K_t-K^\star\|$ v.s. $t$. Each figure has five lines; four lines correspond to four setups of $\sigma^2$, and the red line corresponds to $\sqrt{\beta_t\log(1/\beta_t)}$ v.s. $t$, which is the theoretical asymptotic rate.} }\label{fig:1}
\end{figure}

\begin{figure}[!htp]
\centering     
\subfigure[KKT residual]{\label{C41}\includegraphics[width=0.32\textwidth]{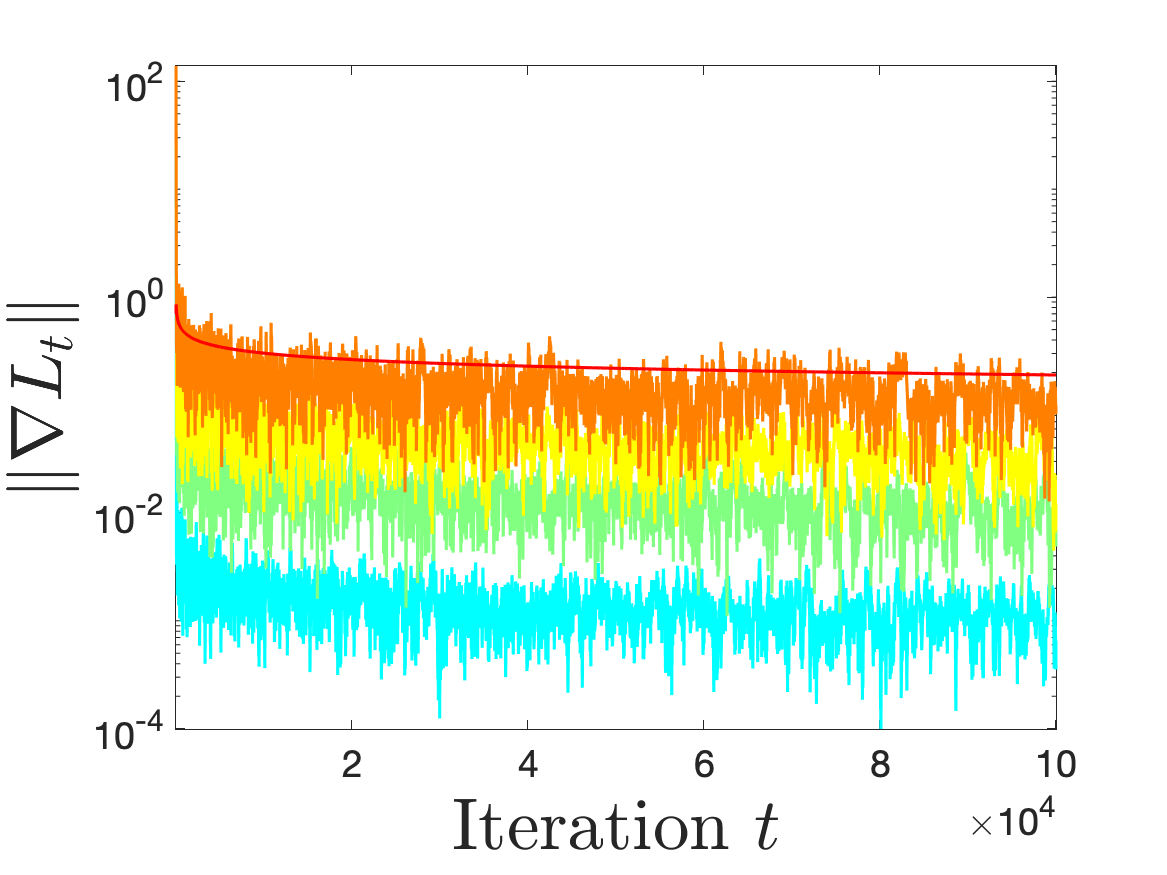}}
\subfigure[Iteration error]{\label{C42}\includegraphics[width=0.32\textwidth]{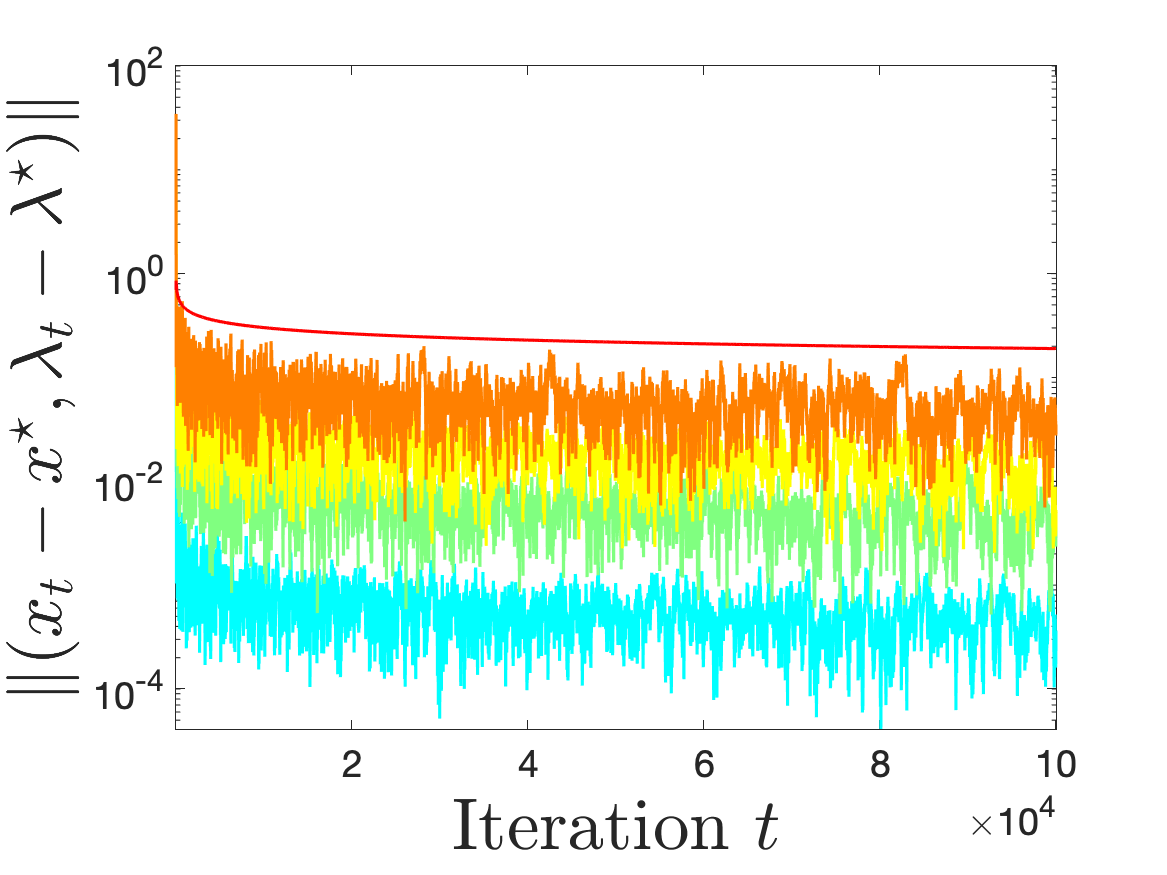}}
\subfigure[Hessian error]{\label{C43}\includegraphics[width=0.32\textwidth]{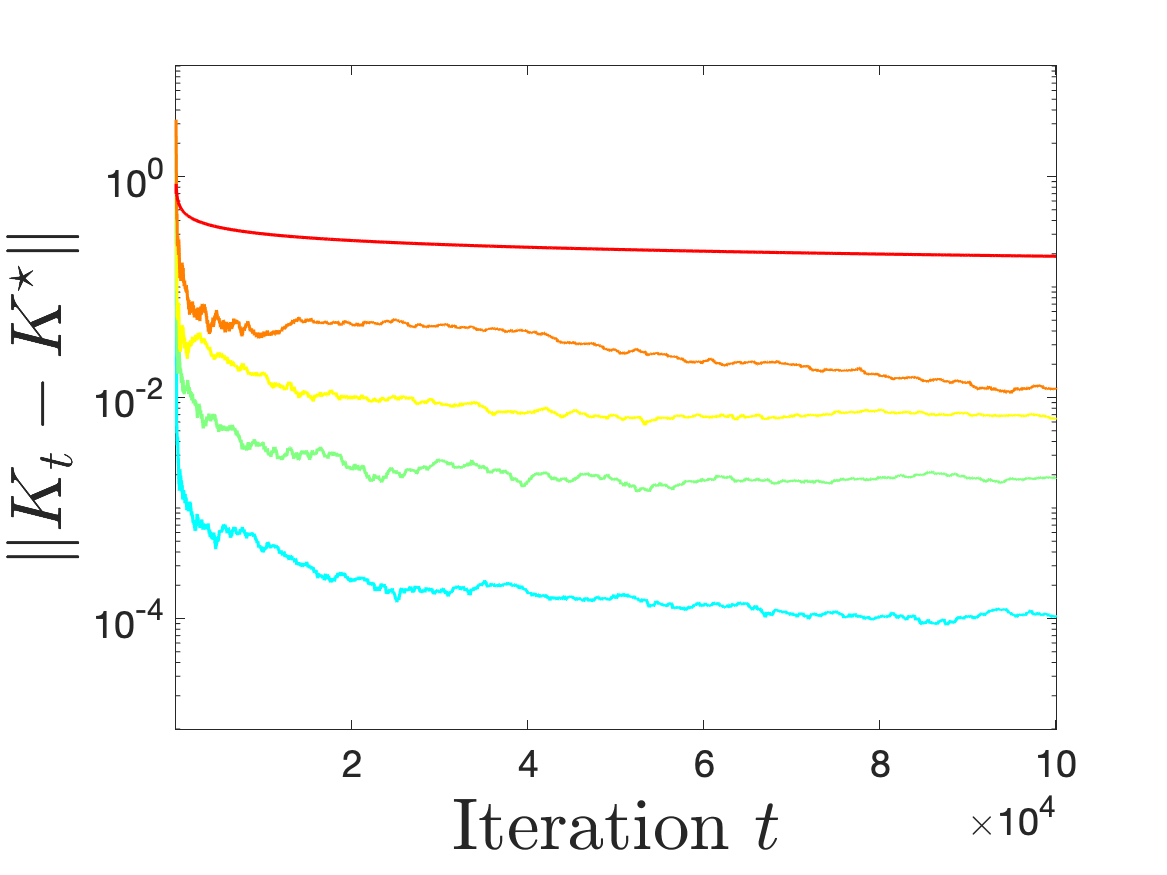}}
\vskip5pt
\centering{Problem \texttt{HS48}}

\subfigure[KKT residual]{\label{C51}\includegraphics[width=0.32\textwidth]{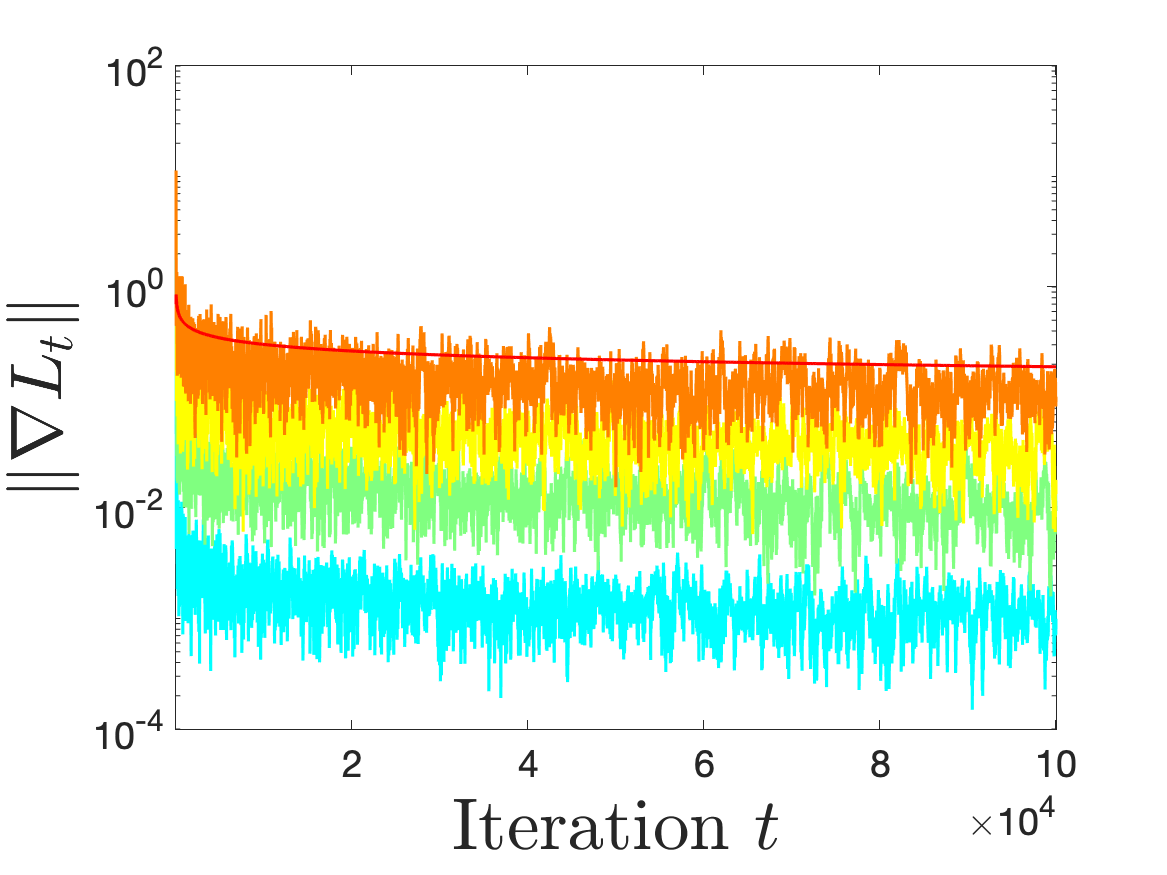}}
\subfigure[Iteration error]{\label{C52}\includegraphics[width=0.32\textwidth]{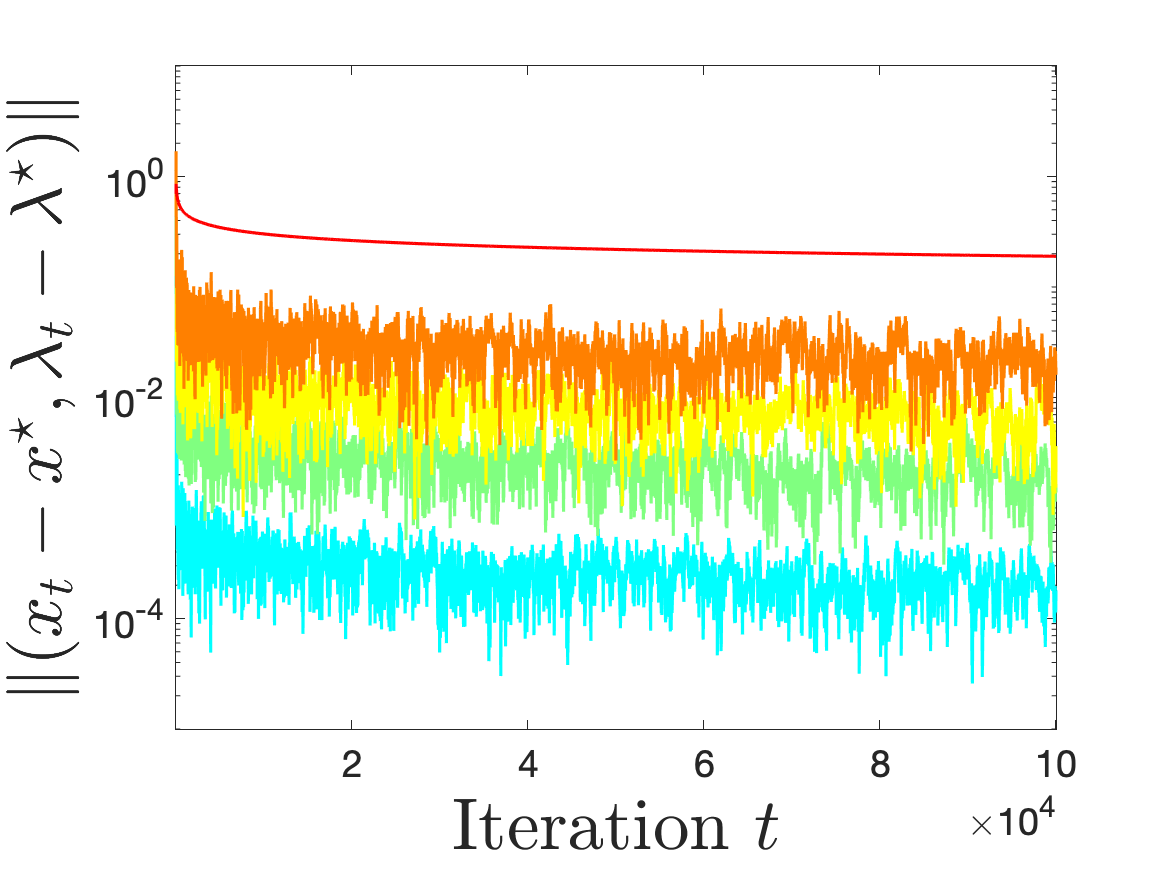}}
\subfigure[Hessian error]{\label{C53}\includegraphics[width=0.32\textwidth]{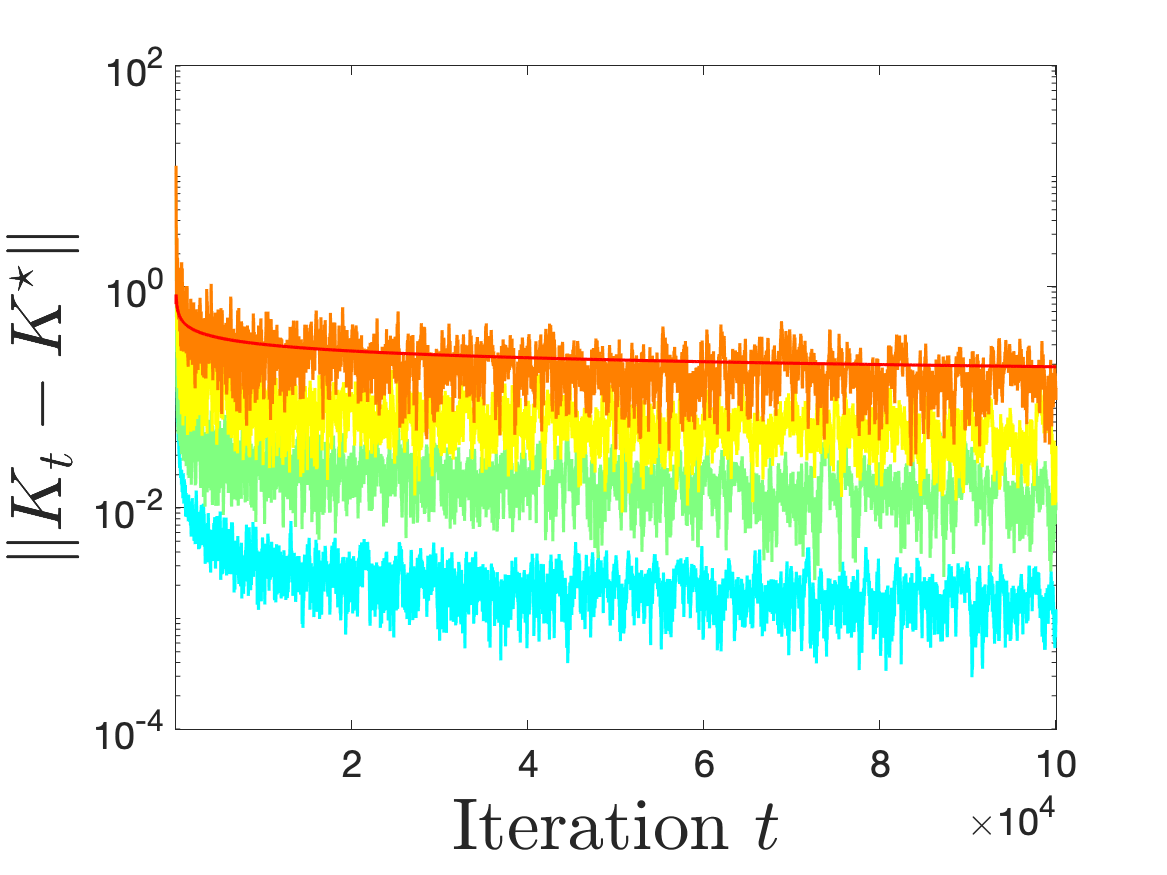}}
\vskip5pt
\centering{Problem \texttt{HS78}}

\subfigure[KKT residual]{\label{C51}\includegraphics[width=0.32\textwidth]{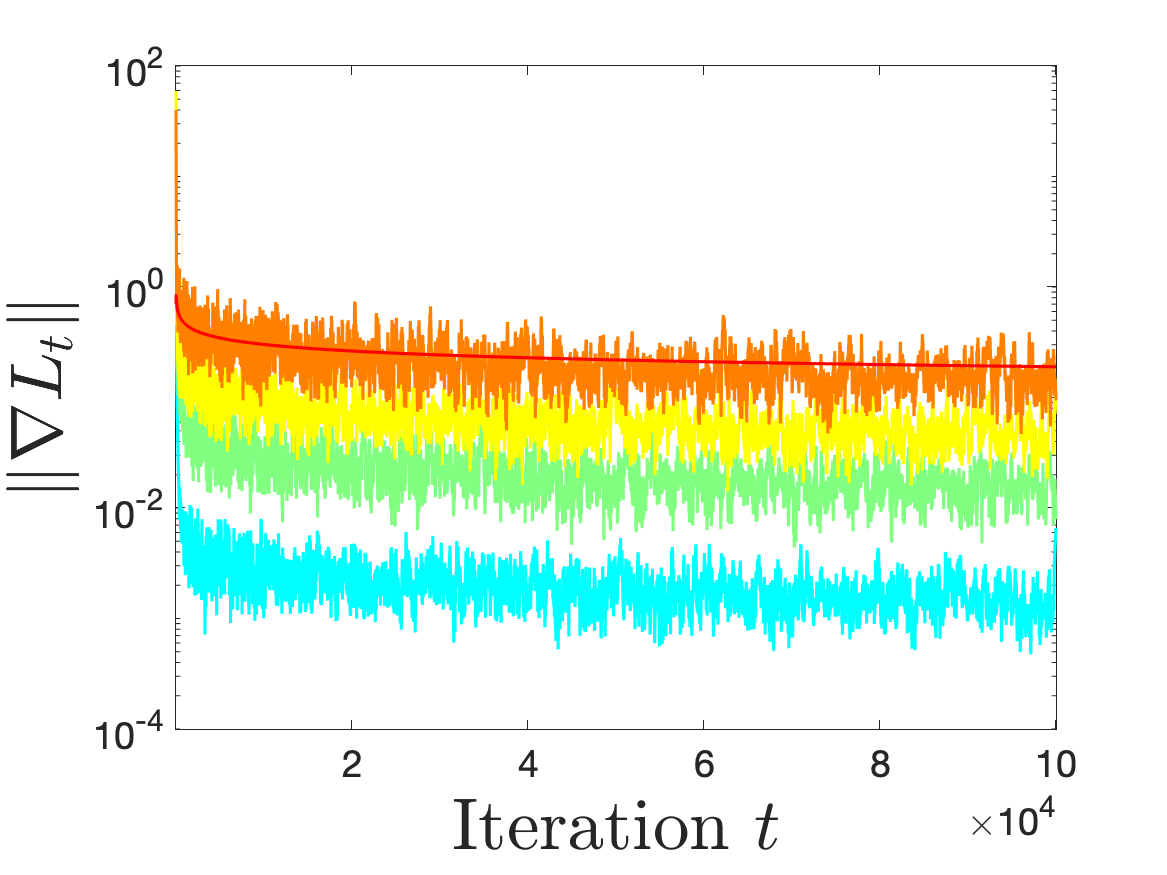}}
\subfigure[Iteration error]{\label{C52}\includegraphics[width=0.32\textwidth]{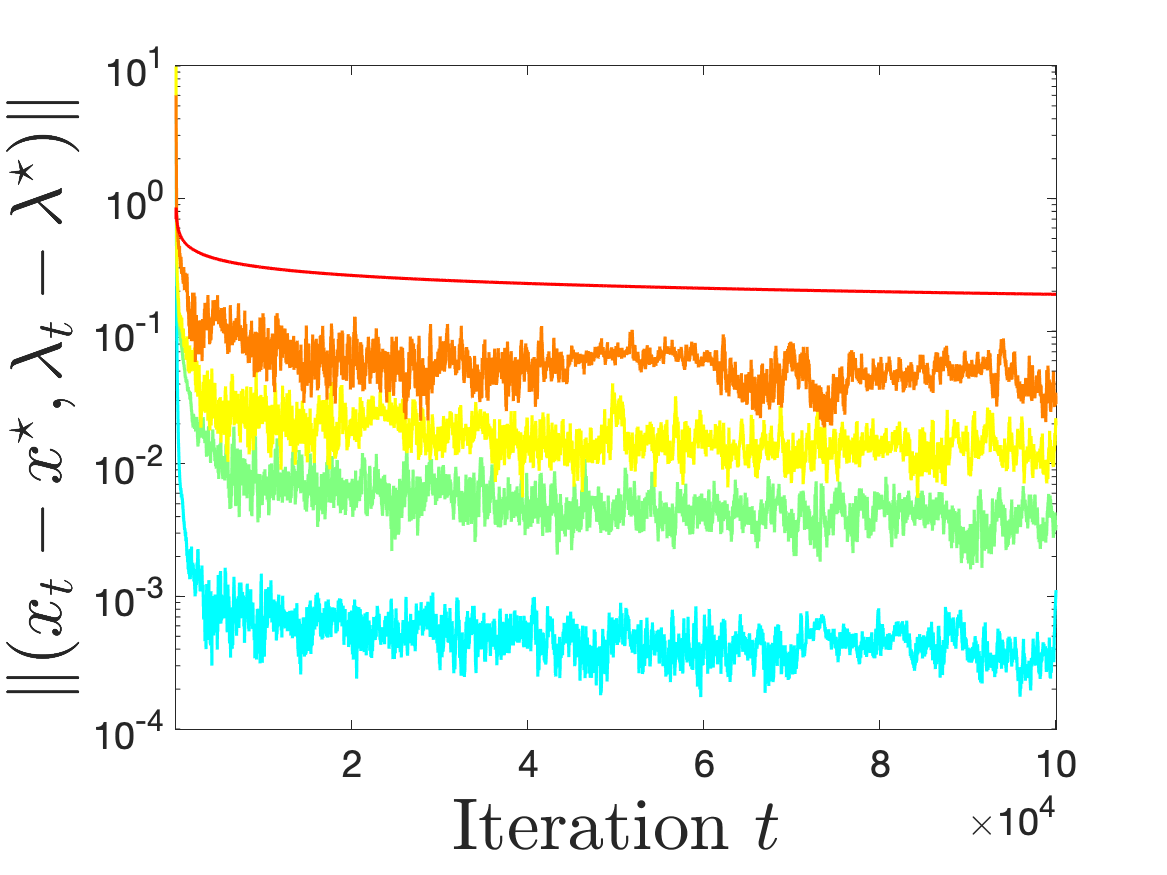}}
\subfigure[Hessian error]{\label{C53}\includegraphics[width=0.32\textwidth]{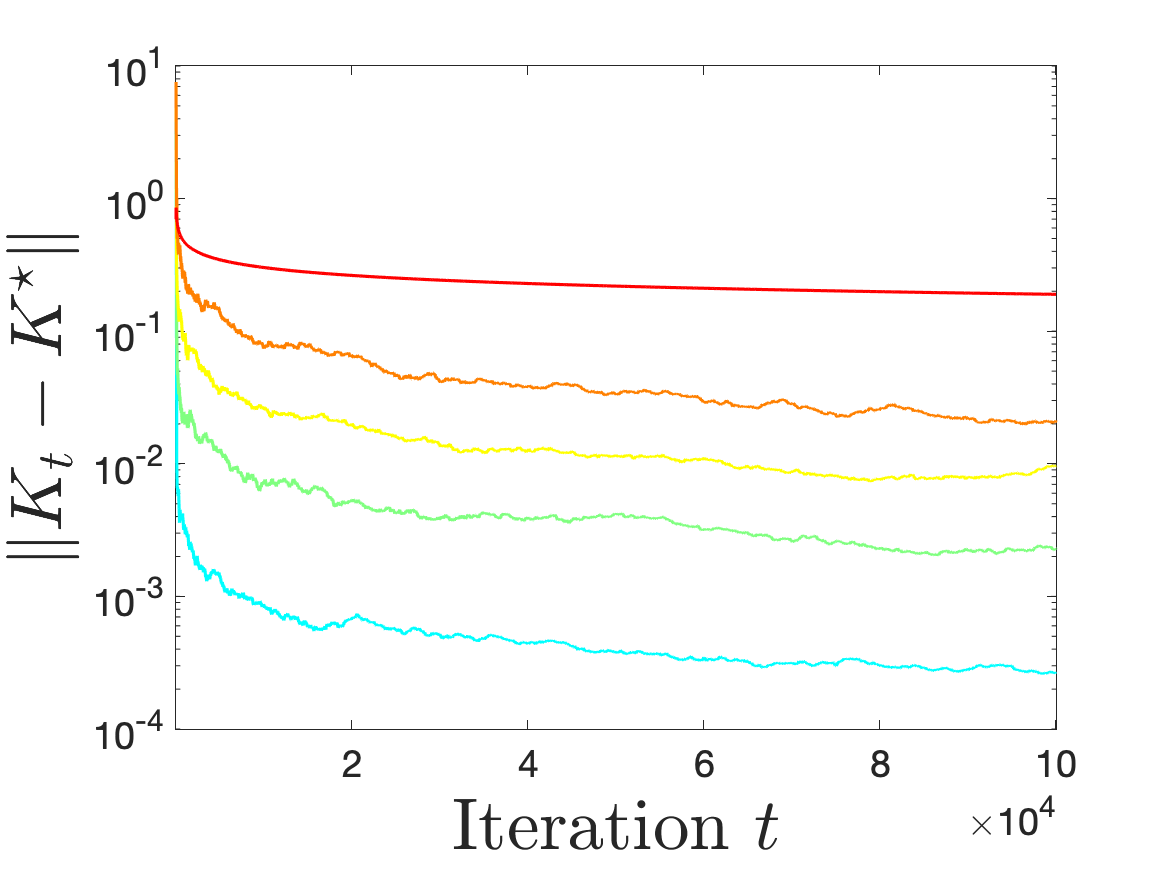}}
\vskip5pt
\centering{Problem \texttt{GENHS28}}
\vskip10pt
\includegraphics[width=0.6\textwidth]{Figure/CUTEst/legend}
\caption{\textit{Convergence plots of CUTEst problems. See Figure \ref{fig:1} for the interpretation.}}\label{fig:2}
\end{figure}

\subsection{Constrained regression problems}\label{sec:exp:more:2}

We follow the experiments in Section \ref{sec:5.2}, and provide comprehensive comparisons between~inexact and exact second-order methods on linearly/nonlinearly constrained regression problems. The coefficient matrix $A$ in linear constraints is sampled from standard normal; the objective of logistic loss is regularized by a quadratic penalty term with a unit~\mbox{parameter}.~We~vary~the~parameters
$d, r$~and $\tau$. In particular, we let $d\in\{5,20,40,60\}$, $r\in\{0.4,0.5,0.6\}$ for Toeplitz~$\Sigma_a$ and $r\in\{0.1,0.2,0.3\}$ for Equi-correlation $\Sigma_a$, and~$\tau \in\{\infty,20,40,60\}$.~We~\mbox{mention}~that~$\tau=\infty$ corresponds to the exact method. For each setup, we perform 200 independent runs.$\hskip1cm$

The extensive comparison results of offline $M$-estimation and StoSQP with different $\tau$~are reported in Tables \ref{tab:4}-\ref{tab:11}. {Specifically, Tables \ref{tab:4} and \ref{tab:5} summarize the results of linear~model~+ linear constraints; Tables \ref{tab:6} and \ref{tab:7} summarize the results of linear model + nonlinear constraints; Tables \ref{tab:8} and~\ref{tab:9} summarize the results of logistic model + \mbox{linear}~constraints;~while~Tables \ref{tab:10} and \ref{tab:11} summarize the results of  logistic model + nonlinear constraints.~For all~four cases, we have the following observations.}

{ 

For MAE, we observe that $M$-estimation achieves results that are an order of magnitude smaller than those of the StoSQP methods.~Among the different setups of $\tau$ of StoSQP,~we~find that exact StoSQP $(\tau=\infty)$ generally yields smaller MAE compared~to~\mbox{inexact} StoSQP.~Furthermore, a larger $\tau$ (i.e., more sketching steps) in StoSQP tends to result in smaller~MAE,~although the differences are less evident than those observed between StoSQP and $M$-estimation methods.~This trend is robust across different setups of the design covariance $\Sigma_a$.~For~instance, for $d=40$ in Table \ref{tab:5}, StoSQP with $\tau=20$ achieves an MAE of approximately~0.2–0.25, while StoSQP with $\tau=40, 60, \infty$ achieves an MAE of less than 0.1. This observation suggests that, given problem parameters such as the condition number of the Lagrangian Hessian and problem dimension, $\tau=20$ may be~\mbox{insufficient}~for~\mbox{solving}~\mbox{Newton}~\mbox{systems}~in~this~\mbox{scenario}.~Specifically, solving exact Newton systems requires $O(46^3) = O(97,336)$ flops while a \mbox{sketching}~solver with $\tau=20$ requires only $O(20 \times 46) = O(920)$ flops.~This substantial reduction in computational cost can result in insufficient precision for the Newton direction at each step,~leading~to a larger MAE when the sample size is fixed.

For coverage rate, we observe that StoSQP with different $\tau$ generally achieves a valid~coverage rate that is very close to the nominal rate 95\%, matching the performance of offline~$M$-estimation. There are two potential exceptions.
The first exception occurs when $d=5$, where 
StoSQP exhibits undercoverage with rates ranging from 87\%-92\%~(cf.~\mbox{Tables}~\ref{tab:4},~\ref{tab:8},~and~\ref{tab:10}).~This undercoverage occurs because the condition numbers of the Lagrangian Hessian in these~problems are significantly larger than in other scenarios, despite the small problem dimension~of~$5$. Consequently, the iteration sequence of StoSQP may not have reached stationarity given~the limited sample size.~As noted in \cite{Zhu2021Online}, even SGD-based estimators require~$3\times 10^5$~to $4\times 10^5$ samples to alleviate undercoverage issue in challenging online inference tasks for~unconstrained problems, while we only have $10^5$ samples for constrained~\mbox{problems}.~Furthermore,~applying a sketching solver to such small-scale problems seems an overkill. We recommend using a simple linear system solver for small-scale problems to reduce the additional uncertainty introduced by the sketching solver.
The second exception occurs when $d=40$, where~StoSQP with $\tau=20$ exhibits undercoverage with rates ranging from 82\%-86\%. In contrast, StoSQP with $\tau=40, 60, \infty$ achieves valid coverage in this case~(cf.~Table \ref{tab:5}).~As~\mbox{explained} for~MAE, this undercoverage results from insufficient sketching steps, which not only make the StoSQP iteration sequence noisy but also make the bias of covariance matrix estimation non-negligible. 
For this scenario, slightly increasing the number of sketching steps (e.g., setting $\tau=40$)~can~resolve the undercoverage issue while still preserving computational efficiency compared to~exact methods.

For average length of confidence intervals, we observe that $M$-estimation produces intervals that are an order of magnitude shorter than those of the StoSQP methods.~Among~the~StoSQP methods, the inexact settings $(\tau < \infty)$ yield average lengths very similar to the exact setting $(\tau = \infty)$, indicating that the inexact methods do not lead to overly conservative intervals.~For both Toeplitz and Equi-correlation $\Sigma_a$, the length of the confidence intervals remains largely unchanged across different setups of $r$. Moreover, for both linear and logistic models, nonlinear constraints tend to result in wider confidence intervals for both offline and online methods,~as shown by comparisons of Tables~\ref{tab:4},~\ref{tab:5} with Tables \ref{tab:6}, \ref{tab:7} and Tables \ref{tab:8}, \ref{tab:9} with Tables \ref{tab:10}, \ref{tab:11}.$\hskip1cm$

For computational flops per iteration, we observe that online StoSQP methods are~significantly more efficient than the offline method. A sketching solver can further reduce~the~computational costs of StoSQP. Choosing the sketching step $\tau$ involves a trade-off between~computational and statistical efficiency. As shown in Table \ref{tab:5}, when $d=40$, $\tau=20$ requires~fewer~flops but achieves larger MAE and lower coverage rates, whereas $\tau=60$ requires more flops~(though still fewer than those of $M$-estimation and $\tau=\infty$) but achieves lower MAE and better~coverage rates. In this case, $\tau=40$ strikes a better balance between the two aspects. We~should~also mention that for small-scale problems $(d=5)$, using a sketching solver with a large $\tau$ is~counterproductive, as it may cause the flops to exceed those of the exact solver.$\hskip3cm$

}

\newpage

\include{tables}

%% file: tables.tex
\begin{table}[thbp!]
\centering
\miniscule{\resizebox{\linewidth}{!}{
\begin{tabular}{|c|cc|cc|c|c|c|c|}
\hline
&&&&&&&& \\[-2.4ex]
$d$ & \multicolumn{2}{c|}{Design Cov} & \multicolumn{2}{c|}{Method} & MAE $(10^{-2})$ & Ave Cov (\%) & Ave Len $(10^{-2})$ & Flops/iter\\
\hline 
&&&&&&&& \\[-2ex]
\multirow{35}{*}{5} & \multicolumn{2}{c|}{\multirow{6}{*}{\vspace{2pt}Identity}}  & \multicolumn{2}{c|}{$M$-estimation} & 0.20(0.08) & 95.40 & 0.19(0.03) & 16189377.50\\
\cline{4-9}
&&&&\multicolumn{1}{|c|}{} &&&&\\[-2ex]
&&& \parbox[t]{2mm}{\multirow{5}{*}{\rotatebox[origin=c]{90}{\mminiscule{\vspace{5pt} StoSQP}}}} & \multicolumn{1}{|c|}{$\tau=\infty$} & 2.64(1.10) & 94.00 & 2.43(0.38) & 443.00 \\
\cline{5-9}
&&&&\multicolumn{1}{|c|}{} &&&&\\[-2ex]
&&&& \multicolumn{1}{|c|}{$\tau=20$} &3.19(1.03)&88.80&2.43(0.38)& 240.00\\
\hline
\end{tabular}}}
\vspace{-0.2cm}
\caption{\textit{Comparison results of online StoSQP and offline $M$-estimation for constrained~regression problems \textbf{(linear model + linear constraints)}.}}\label{tab:4}
\vspace{-0.2cm}
\end{table}

\begin{table}[thbp!]
\centering
\miniscule{\resizebox{\linewidth}{!}{\begin{tabular}{|c|cc|cc|c|c|c|c|}
\hline
&&&&&&&& \\[-2.4ex]
$d$ & \multicolumn{2}{c|}{Design Cov} & \multicolumn{2}{c|}{Method} & MAE $(10^{-2})$ & Ave Cov (\%) & Ave Len $(10^{-2})$ & Flops/iter\\
\hline 
&&&&&&&& \\[-2ex]
\multirow{35}{*}{40} & \multicolumn{2}{c|}{\multirow{6}{*}{\vspace{2pt}Identity}}  & \multicolumn{2}{c|}{$M$-estimation} & 
0.75 (0.09) & 94.99 & 0.23 (0.01) & 230134815.43 \\
\cline{4-9}
&&&&\multicolumn{1}{|c|}{} &&&&\\[-2ex]
&&& \parbox[t]{2mm}{\multirow{5}{*}{\rotatebox[origin=c]{90}{\mminiscule{\vspace{5pt} StoSQP}}}} & \multicolumn{1}{|c|}{$\tau=\infty$} & 
9.53 (1.24) & 95.30 & 3.00 (0.14) & 102655.95 \\
\cline{5-9}
&&&&\multicolumn{1}{|c|}{} &&&&\\[-2ex]
&&&& \multicolumn{1}{|c|}{$\tau=20$} &
25.15 (65.81) & 82.88 & 4.02 (5.33) & 6240.91 \\
\hline
\end{tabular}}}
\vspace{-0.2cm}
\caption{\textit{Comparison results of online StoSQP and offline $M$-estimation for constrained~regression problems \textbf{(linear model + linear constraints)}).}}\label{tab:5}
\vspace{-0.2cm}
\end{table}

\begin{table}[thbp!]
\centering
\miniscule{\resizebox{\linewidth}{!}{\begin{tabular}{|c|cc|cc|c|c|c|c|}
\hline
&&&&&&&& \\[-2.4ex]
$d$ & \multicolumn{2}{c|}{Design Cov} & \multicolumn{2}{c|}{Method} & MAE $(10^{-2})$ & Ave Cov (\%) & Ave Len $(10^{-2})$ & Flops/iter\\
\hline 
&&&&&&&& \\[-2ex]
\multirow{35}{*}{5} & \multicolumn{2}{c|}{\multirow{6}{*}{\vspace{2pt}Identity}}  & \multicolumn{2}{c|}{$M$-estimation} & 
0.25(0.10) & 94.90 & 0.22(0.03) & 9383817.60 \\
\cline{4-9}
&&&&\multicolumn{1}{|c|}{} &&&&\\[-2ex]
&&& \parbox[t]{2mm}{\multirow{5}{*}{\rotatebox[origin=c]{90}{\mminiscule{\vspace{5pt} StoSQP}}}} & \multicolumn{1}{|c|}{$\tau=\infty$} & 
3.09(1.13) & 94.90 & 2.82(0.37) & 306.00 \\
\cline{5-9}
&&&&\multicolumn{1}{|c|}{} &&&&\\[-2ex]
&&&& \multicolumn{1}{|c|}{$\tau=20$} &
3.33(1.13) & 93.50 & 2.82(0.37) & 210.00 \\
\hline
\end{tabular}}}
\vspace{-0.2cm}
\caption{\textit{Comparison results of online StoSQP and offline $M$-estimation for constrained~regression problems \textbf{(linear model + nonlinear constraints)}.}}\label{tab:6}
\vspace{-0.2cm}
\end{table}

\begin{table}[thbp!]
\centering
\miniscule{\resizebox{\linewidth}{!}{\begin{tabular}{|c|cc|cc|c|c|c|c|}
\hline
&&&&&&&& \\[-2.4ex]
$d$ & \multicolumn{2}{c|}{Design Cov} & \multicolumn{2}{c|}{Method} & MAE $(10^{-2})$ & Ave Cov (\%) & Ave Len $(10^{-2})$ & Flops/iter\\
\hline 
&&&&&&&& \\[-2ex]
\multirow{35}{*}{40} & \multicolumn{2}{c|}{\multirow{6}{*}{\vspace{2pt}Identity}}  & \multicolumn{2}{c|}{$M$-estimation} & 
0.79(0.09) & 95.23 & 0.25(0.01) & 67887974.71 \\
\cline{4-9}
&&&&\multicolumn{1}{|c|}{} &&&&\\[-2ex]
&&& \parbox[t]{2mm}{\multirow{5}{*}{\rotatebox[origin=c]{90}{\mminiscule{\vspace{5pt} StoSQP}}}} & \multicolumn{1}{|c|}{$\tau=\infty$} & 
10.44(1.14) & 94.99 & 3.24(0.06) & 73840.95 \\
\cline{5-9}
&&&&\multicolumn{1}{|c|}{} &&&&\\[-2ex]
&&&& \multicolumn{1}{|c|}{$\tau=20$} &
10.63(1.20) & 94.44 & 3.23(0.06) & 5740.63 \\
\hline
\end{tabular}}}
\vspace{-0.2cm}
\caption{\textit{Comparison results of online StoSQP and offline $M$-estimation for constrained~regression problems \textbf{(linear model + nonlinear constraints)}.}}\label{tab:7}
\vspace{-0.2cm}
\end{table}

\begin{table}[thbp!]
\centering
\miniscule{\resizebox{\linewidth}{!}{\begin{tabular}{|c|cc|cc|c|c|c|c|}
\hline
&&&&&&&& \\[-2.4ex]
$d$ & \multicolumn{2}{c|}{Design Cov} & \multicolumn{2}{c|}{Method} & MAE $(10^{-2})$ & Ave Cov (\%) & Ave Len $(10^{-2})$ & Flops/iter\\
\hline 
&&&&&&&& \\[-2ex]
\multirow{35}{*}{5} & \multicolumn{2}{c|}{\multirow{6}{*}{\vspace{2pt}Identity}}  & \multicolumn{2}{c|}{$M$-estimation} & 
0.13(0.06) & 95.00 & 0.13(0.03) & 18631964.09 \\
\cline{4-9}
&&&&\multicolumn{1}{|c|}{} &&&&\\[-2ex]
&&& \parbox[t]{2mm}{\multirow{5}{*}{\rotatebox[origin=c]{90}{\mminiscule{\vspace{5pt} StoSQP}}}} & \multicolumn{1}{|c|}{$\tau=\infty$} & 
1.89(0.85) & 91.70 & 1.57(0.42) & 443.00 \\
\cline{5-9}
&&&&\multicolumn{1}{|c|}{} &&&&\\[-2ex]
&&&& \multicolumn{1}{|c|}{$\tau=20$} &
2.16(0.86) & 86.10 & 1.56(0.42) & 240.00 \\
\hline
\end{tabular}}}
\vspace{-0.2cm}
\caption{\textit{Comparison results of online StoSQP and offline $M$-estimation for constrained~regression problems \textbf{(logistic model + linear constraints)}.}}\label{tab:8}
\vspace{-0.2cm}
\end{table}

\begin{table}[thbp!]
\centering
\miniscule{\resizebox{\linewidth}{!}{\begin{tabular}{|c|cc|cc|c|c|c|c|}
\hline
&&&&&&&& \\[-2.4ex]
$d$ & \multicolumn{2}{c|}{Design Cov} & \multicolumn{2}{c|}{Method} & MAE $(10^{-2})$ & Ave Cov (\%) & Ave Len $(10^{-2})$ & Flops/iter\\
\hline 
&&&&&&&& \\[-2ex]
\multirow{35}{*}{40} & \multicolumn{2}{c|}{\multirow{6}{*}{\vspace{2pt}Identity}}  & \multicolumn{2}{c|}{$M$-estimation} & 
0.40(0.05) & 95.05 & 0.13(0.01) & 368280379.99 \\
\cline{4-9}
&&&&\multicolumn{1}{|c|}{} &&&&\\[-2ex]
&&& \parbox[t]{2mm}{\multirow{5}{*}{\rotatebox[origin=c]{90}{\mminiscule{\vspace{5pt} StoSQP}}}} & \multicolumn{1}{|c|}{$\tau=\infty$} & 
5.39(0.83) & 93.46 & 1.59(0.09) & 102655.95 \\
\cline{5-9}
&&&&\multicolumn{1}{|c|}{} &&&&\\[-2ex]
&&&& \multicolumn{1}{|c|}{$\tau=20$} &
5.11(0.68) & 94.85 & 1.59(0.09) & 6240.91 \\
\hline
\end{tabular}}}
\vspace{-0.2cm}
\caption{\textit{Comparison results of online StoSQP and offline $M$-estimation for constrained~regression problems \textbf{(logistic model + linear constraints)}.}}\label{tab:9}
\vspace{-0.2cm}
\end{table}

\begin{table}[thbp!]
\centering
\miniscule{\resizebox{\linewidth}{!}{\begin{tabular}{|c|cc|cc|c|c|c|c|}
\hline
&&&&&&&& \\[-2.4ex]
$d$ & \multicolumn{2}{c|}{Design Cov} & \multicolumn{2}{c|}{Method} & MAE $(10^{-2})$ & Ave Cov (\%) & Ave Len $(10^{-2})$ & Flops/iter\\
\hline 
&&&&&&&& \\[-2ex]
\multirow{35}{*}{5} & \multicolumn{2}{c|}{\multirow{6}{*}{\vspace{2pt}Identity}}  & \multicolumn{2}{c|}{$M$-estimation} & 
0.17(0.07) & 94.70 & 0.17(0.02) & 5284663.88 \\
\cline{4-9}
&&&&\multicolumn{1}{|c|}{} &&&&\\[-2ex]
&&& \parbox[t]{2mm}{\multirow{5}{*}{\rotatebox[origin=c]{90}{\mminiscule{\vspace{5pt} StoSQP}}}} & \multicolumn{1}{|c|}{$\tau=\infty$} & 
2.68(0.83) & 91.40 & 2.08(0.28) & 306.00 \\
\cline{5-9}
&&&&\multicolumn{1}{|c|}{} &&&&\\[-2ex]
&&&& \multicolumn{1}{|c|}{$\tau=20$} &
2.71(0.93) & 90.00 & 2.08(0.27) & 210.00 \\
\hline
\end{tabular}}}
\vspace{-0.2cm}
\caption{\textit{Comparison results of online StoSQP and offline $M$-estimation for constrained~regression problems \textbf{(logistic model + nonlinear constraints)}.}}\label{tab:10}
\vspace{-0.2cm}
\end{table}

\begin{table}[thbp!]
\centering
\miniscule{\resizebox{\linewidth}{!}{\begin{tabular}{|c|cc|cc|c|c|c|c|}
\hline
&&&&&&&& \\[-2.4ex]
$d$ & \multicolumn{2}{c|}{Design Cov} & \multicolumn{2}{c|}{Method} & MAE $(10^{-2})$ & Ave Cov (\%) & Ave Len $(10^{-2})$ & Flops/iter\\
\hline 
&&&&&&&& \\[-2ex]
\multirow{35}{*}{40} & \multicolumn{2}{c|}{\multirow{6}{*}{\vspace{2pt}Identity}}  & \multicolumn{2}{c|}{$M$-estimation} & 
0.45(0.05) & 94.60 & 0.14(0.01) & 36407191.34 \\
\cline{4-9}
&&&&\multicolumn{1}{|c|}{} &&&&\\[-2ex]
&&& \parbox[t]{2mm}{\multirow{5}{*}{\rotatebox[origin=c]{90}{\mminiscule{\vspace{5pt} StoSQP}}}} & \multicolumn{1}{|c|}{$\tau=\infty$} & 
5.92(0.78) & 92.85 & 1.73(0.05) & 73840.95 \\
\cline{5-9}
&&&&\multicolumn{1}{|c|}{} &&&&\\[-2ex]
&&&& \multicolumn{1}{|c|}{$\tau=20$} &
5.82(0.74) & 94.05 & 1.73(0.06) & 5740.63 \\
\hline
\end{tabular}}}
\vspace{-0.2cm}
\caption{\textit{Comparison results of online StoSQP and offline $M$-estimation for constrained~regression problems \textbf{(logistic model + nonlinear constraints)}.}}\label{tab:11}
\vspace{-0.2cm}
\end{table}